\documentclass[a4paper]{amsart}
\pdfoutput=1
\usepackage[margin=1.2in]{geometry}
\usepackage{amssymb,amsmath}
\usepackage{tikz-cd}
\usepackage{hyperref}
\usepackage[all]{xy}
\usepackage{enumitem}
\usepackage{mathtools}
\usepackage{multirow}
\usepackage{floatrow}

\usetikzlibrary{patterns}

\theoremstyle{plain}
\newtheorem{theorem}[subsection]{Theorem}
\newtheorem{lemma}[subsection]{Lemma}
\newtheorem{sublemma}[subsection]{Sublemma}
\newtheorem{proposition}[subsection]{Proposition}
\newtheorem{corollary}[subsection]{Corollary}
\theoremstyle{definition}
\newtheorem{definition}[subsection]{Definition}
\newtheorem{example}[subsection]{Example}
\newtheorem{remark}[subsection]{Remark}

\newcommand{\AC}{\mathcal{A}}
\newcommand{\GW}{\mathrm{GW}}
\newcommand{\K}{K}
\newcommand{\SO}{\mathcal{O}}
\newcommand{\Z}{\mathbb{Z}}
\newcommand{\id}{\mathrm{id}}
\newcommand{\op}{\textnormal{op}}
\newcommand{\Hom}{\textnormal{Hom}}
\newcommand{\quis}{\textnormal{quis}}
\newcommand{\can}{\textnormal{can}}
\newcommand{\Spec}{\mathrm{Spec}}
\newcommand{\ext}{\mathrm{Ext}}
\newcommand{\im}{\mathrm{im}}

\newcommand{\cone}{\textnormal{cone}}
\newcommand{\Sch}{\textnormal{Sch}}
\newcommand{\Sets}{\textnormal{Sets}}
\newcommand{\Fl}{\mathrm{Fl}}
\newcommand{\Gr}{\mathrm{Gr}}
\newcommand{\sHom}{\mathcal{H}\mathrm{om}}
\newcommand{\M}{\mathbf{M}}

\newcommand{\V}{\mathbf{V}}

\newcommand{\Ch}{\textnormal{Ch}}
\newcommand{\Res}{\mathrm{Res}}
\newcommand{\Dual}{\mathrm{Dual}}
\newcommand{\Tr}{\mathrm{Tr}}
\newcommand{\uG}{\underline{\Gamma}}
\newcommand{\uH}{\underline{\mathrm{H}}}
\newcommand{\Pic}{\mathrm{Pic}}
\newcommand{\SP}{\mathcal{P}}
\newcommand{\SV}{\mathcal{V}}

\newcounter{mydiag}
\renewcommand{\themydiag}{\theequation}
\newcommand{\diagram}{\stepcounter{equation} \refstepcounter{mydiag} {\rm \fbox{$\scriptstyle \themydiag$}}}
\newcommand{\diag}[1]{{\rm \fbox{$\scriptstyle #1$}}} 

\DeclareFontFamily{U}{mathx}{}
\DeclareFontShape{U}{mathx}{m}{n}{<-> mathx10}{}
\DeclareSymbolFont{mathx}{U}{mathx}{m}{n}
\DeclareMathAccent{\widehat}{0}{mathx}{"70}
\DeclareMathAccent{\widecheck}{0}{mathx}{"71}

\author{Tao Huang}
\address{Tao Huang, School of Mathematics, Sun Yat-sen University, No. 135 Xingang Xi Road, Guangzhou, 510275, China}
\email{huangt233@mail2.sysu.edu.cn}
\author{Heng Xie}
\address{Heng Xie, School of Mathematics, Sun Yat-sen University, No. 135 Xingang Xi Road, Guangzhou, 510275, China}
\email{xieh59@mail.sysu.edu.cn}
\thanks{MSC classes: 19G38, 11E81, and 14M15.}

\title[The connecting homomorphism for Hermitian $K$-theory]{The connecting homomorphism for Hermitian $K$-theory: Projective bundles and Grassmannians}

\begin{document}
\begin{abstract}
	We provide a geometric interpretation for the connecting homomorphism in the localization sequence of Hermitian $K$-theory. As an application, we compute the Hermitian $K$-theory of projective bundles and Grassmannians in the regular case. We provide an explicit basis for Hermitian $K$-theory of Grassmannians, which is indexed by even Young diagrams together with another special class of Young diagrams, that we call \textit{buffalo-check} Young diagrams. To achieve this, we develop pushforwards and pullbacks in Hermitian $K$-theory using Grothendieck's residue complexes, and we establish fundamental theorems for those pushforwards and pullbacks, including base change, projection, and excess intersection formulas.
\end{abstract}
\maketitle
\setcounter{tocdepth}{1}
\tableofcontents

\section{Introduction}

Pushforwards and pullbacks are foundational in many cohomology theories. For instance, in $K$-theory, Chow theory, and Witt theory, they serve as powerful computational tools.\ Interestingly, pushforwards in Witt theory always keep track of orientations, complicating computations compared to those in oriented cohomology theories, cf. \cite{calmes2011push}. Based on pushforwards and pullbacks, Balmer and Calm\`es \cite{balmer2009geometric} introduced the `Blow-up setup', a method that geometrically interprets the connecting homomorphism in the 12-term localization sequence for Witt groups. As an application, Balmer-Calm\`es computed Witt groups of Grassmannians, which was a long standing open problem in the Witt theory, cf. \cite{balmer2012witt}.

Hermitian $K$-theory, a generalization of Witt theory, has garnered attention  for its applications in recent works. Roughly speaking, Hermitian $K$-theory can be thought of a `combination' of $K$-theory and Witt theory in view of the algebraic Bott sequence. We mention some recent applications of Hermitian $K$-theory.\ Asok and Fasel utilized Hermitian $K$-theory as a tool to give a cohomological classification of vector bundles of rank $2$ on a smooth affine threefold \cite{asok2014a}, and they also made progress in Murthy's conjecture \cite{asok2015splitting}. Fasel and Srinivas \cite{fasel2009chow} showed that a vector bundle $\SV$ of rank $3$ over a smooth affine threefold splits off a trivial direct summand if and only if its Euler class $e(\SV)$ in Hermitian $K$-theory is zero.  By studying the slice filtration for Hermitian $K$-theory, R\"{o}ndigs, Spitzweck, and \O stv\ae r computed the first stable homotopy groups of motivic spheres \cite{rondigs2019the}. The second author showed that a numerical condition of Atiyah on Hurwitz's 1898 problem can be generalized from $\mathbb{R}$ to any field by computing Hermitian $K$-theory of deleted quadrics \cite{xie2014an}.

This work is devoted to the study of pushforwards and pullbacks in Hermitian $K$-theory and the geometric description of the connecting homomorphism in Hermitian $K$-theory.  Building on the work of Schlichting \cite{schlichting2010mayer} and \cite{schlichting2017hermitian}  in Hermitian $K$-theory, we show that the connecting homomorphism in the localization sequence of Hermitian $K$-theory can be interpreted by the pushforward, pullback and cupping with the Bott element. More precisely, we have the following result.

\begin{theorem}\label{thm:connecting-homo-1}
	Assume that we have the following diagram:
	\[\xymatrix{
		Z \ar@{^{(}->}[r]^-{\iota} & X  & \ar@{_{(}->}[l]_-{v}  \ar@{_{(}->}[ld]_-{\tilde{v}} U \ar[d]^-{\alpha}  \\
		E \ar[u]^-{\widetilde{\pi}} \ar@{^{(}->}[r]_-{\tilde{\iota}} & B  \ar[u]^-{\pi} \ar@{-->}[r]_-{\tilde{\alpha}} & Y
		}\]
	where $\iota: Z \to X$ is a regular immersion of codimension $c \geq 2$ between regular schemes, $B$ is the blow up of $X$ along $Z$, and $E$ is the exceptional fibre. Here, we set $U := X-Z \cong B- E$, and let $v,\tilde{v}$ be the corresponding open immersions. Suppose that there is a flat morphism $\tilde{\alpha} : B \to Y$ such that $\alpha := \tilde{\alpha} \tilde{v}$ is an affine bundle.
	Let $L$ be a line bundle on $X$. Then, there exists an integer $\lambda(L) \in \Z$ such that

	\begin{enumerate}[leftmargin=20pt,label={\rm (\Alph*)}]
		\item If $\lambda(L) \equiv c-1 \mod 2$, then we have an equivalence
		      \[ \GW^{[n]}(X,L)  \simeq \GW^{[n]}(Y,L_{Y}) \oplus \GW^{[n-c]}(Z, \omega_\iota \otimes L_{Z})  \]
		      of spectra in the stable homotopy category of spectra $\mathcal{SH}$.
		\item If $\lambda(L) \equiv c \mod 2$, then the following diagram
		      $$
			      \xymatrix{
			      \GW^{[n]}(U,L) \ar[d]^-{\simeq}_-{(\alpha^*)^{-1}} \ar[r]^-{\partial} & \GW^{[n-c]}(Z, \omega_\iota \otimes L)[1]  & \ar[l]_-{\eta \cup} \GW^{[n-c +1]}(Z, \omega_\iota \otimes L )  \\
			      \GW^{[n]}(Y,L_{Y})  \ar[r]^-{ \tilde{\iota}^* \tilde{\alpha}^*}  & \GW^{[n]}(E,L) \ar[r]^-{\simeq} & \GW^{[n]}(E,  \omega_{\widetilde{\pi}} \otimes \widetilde{\pi}^*(\omega_\iota \otimes L) ) \ar[u]^-{\widetilde{\pi}_*}
			      }
		      $$
		      of spectra is commutative in $\mathcal{SH}$. In particular, on the level of homotopy groups, the connecting homomorphism $$\partial: 	\GW^{[n]}_i(U,L) \to  \GW^{[n-c]}_i(Z, \omega_\iota \otimes L)[1]$$ is equal to the composition $  (\eta \cup) \widetilde{\pi}_* \tilde{\iota}^* \tilde{\alpha} (\alpha^*)^{-1}$.
	\end{enumerate}
\end{theorem}

Pushforwards and pullbacks in Hermitian $K$-theory which were conjectured to exist should be of independent interests, cf. \cite[Section 3.3]{bachmann2023euler}.
To demonstrate the applications of Theorem \ref{thm:connecting-homo-1}, we compute Hermitian $K$-theory of projective bundles in Theorem \ref{thm:projective_bundle_thm}. This computation has previously  appeared in the work of Walter in the case of $\GW_0$, but using a different method (see also \cite{rohrbach2021on},  \cite{karoubi2021grothendieck} and \cite{calmes2024a}). Furthermore, we proceed to compute Hermitian $K$-theory of Grassmannians in Section \ref{sect:grassmanian}.

Let $S$ be a regular scheme, and let $L$ be a line bundle over $S$. Let $\Gr_d(m)$ be the Grassmannian of $d$-dimensional subbundles of the trivial bundle $\SO^{d+m}$, and let $\Delta_d$ be the determinantal line bundle of the taotaulogical $d$-bundle on $\Gr_d(m)$. For $i \in \Z$, we define
$$ \GW^{[n]}_i(\Gr_d(m),L)^{\mathrm{tot}} : =  \GW^{[n]}_i(\Gr_d(m),L) \oplus \GW^{[n]}_i(\Gr_d(m),L \otimes \Delta_d).$$
\begin{theorem}[Theorem \ref{theo:GW-Grassmannian}]
	The map  
 $$ (\Theta^0+\Theta^1, \Omega^0 +\Omega^1): \bigoplus_{(\Lambda, c_\Lambda) \in \mathfrak{B}_{d,m}} \K_i(S)  \oplus \bigoplus_{\Pi \in \mathfrak{A}_{d,m}} \GW^{[n-|\Pi|]}_i(S,L) \xrightarrow{\simeq} \GW^{[n]}_i(\Gr_d(m),L)^{\mathrm{tot}}  $$
 is an equivalence of spectra, where $\mathfrak{A}_{d,m}$ is the set of even Young diagrams in $(d \times m)$-frame, and $\mathfrak{B}_{d,m}$ is the set of {buffalo-check} Young diagrams in $(d \times m)$-frame.\ Here, $|\Pi|$ denotes the number of boxes in an even Young diagram $\Pi$.
\end{theorem}

Even Young diagrams were introduced by Balmer-Calm\`es in order to index the additive basis for Witt groups of Grassmannian (cf. \cite[Definition 2.7]{balmer2012witt}). To understand the additive basis for the $K$-theory part of Hermitian $K$-theory of Grassmannians, we shall introduce a combinatorial object called \textit{buffalo-check} Young diagrams.

Consider an ordinary Young diagram sitting in the upper left corner of a $(d\times m)$-frame $\Xi$, which is a rectangle consisting of $d$ rows and $m$ columns. The \textit{boundary} $b(\Lambda)$ of a Young diagram $\Lambda$ in $\Xi$ is a lattice path that goes from  lower-left to upper-right corner in $\Xi$ such that the Young diagram $\Lambda$ lives precisely in the upper-left area above the path. The boundary $b(\Lambda)$ consists of \textit{segments} $s_1, \ldots, s_l$, where we order them from left to right.\

\begin{figure}[!ht]
	\begin{center}
		\begin{tikzpicture}[y=-1cm]

			\path[draw=white, fill=black!20] (5.6,3.4) -- (6.4,3.4) -- (6.4,2.6) -- (7.6,2.6) -- (7.6,1.4) -- (8.4,1.4) -- (8.4,1) -- (5.6,1);
			\draw[black] (5.6,1) rectangle (8.9,3.8);

			\draw[black, ultra thick] (5.6,3.8) --(5.6,3.4) -- (6.4,3.4) -- (6.4,2.6)-- (7.6,2.6) -- (7.6,1.4) -- (8.4,1.4) -- (8.4,1) -- (8.9,1)  ;
			\path (6.4,1.9) node[text=black,anchor=base west] {$\Lambda$};
			\path (5.15,3.7)
			node[text=black,anchor=base west] {\scriptsize $s_1$};
			\path (5.8,3.57)
			node[text=black,anchor=base west] {\scriptsize $s_2$};
			\path (5.95,3)
			node[text=black,anchor=base west] {\scriptsize $s_3$};
			\path (6.7,2.5)
			node[text=black,anchor=base west] {\scriptsize $s_4$};
			\path (7.15,2)
			node[text=black,anchor=base west] {\scriptsize $s_5$};
			\path (7.8,1.55)
			node[text=black,anchor=base west] {\scriptsize $\ldots$};
		\end{tikzpicture}
	\end{center}
	\caption{Young diagram $\Lambda$ in the $(d \times m)$-frame $\Xi$. The boundary $b(\Lambda)$ is thickened. }
	\label{Fig:Framed_fig_Young_diagram}%
\end{figure}
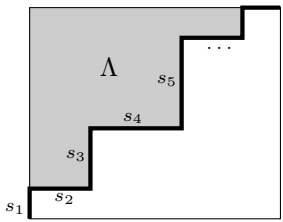

A Young diagram $\Lambda$ in a $(d \times m)$-frame $\Xi$ is called \textit{$K$-even} if there exists an integer $1 \leq w(\Lambda) < l$ such that segments $s_2, \ldots, s_{w(\Lambda)-1}$ have even length and the segment $s_{w(\Lambda)}$, which has odd length if $w(\Lambda) \geq 2$, is vertical. In other words, a Young diagram $\Lambda$ in $(d \times m)$-frame $\Xi$ is $K$-even if and only if (i) the first segment $s_1$ is vertical, or if (ii) there exists an integer $r(\Lambda) \geq 2$ such that the segments $s_2, \ldots, s_{r(\Lambda)-1}$ have even length and the segement $s_{r(\Lambda)}$ is vertical of odd length. A box in $\Xi$ is called a \textit{center} of a $K$-even Young diagram $\Lambda$, if it sits in the angle between the segments $s_{w(\Lambda)}$ and $s_{w(\Lambda)+1}$; Or equivalently, if it sits in the angle between the segments $s_1$ and $s_{2}$ when condition (i) holds, or between the segments $s_{r(\Lambda)}$ and $s_{r(\Lambda)+1}$ when condition (ii) holds. Note that a $K$-even Young diagram $\Lambda$ in $\Xi$ could have at most two centers.

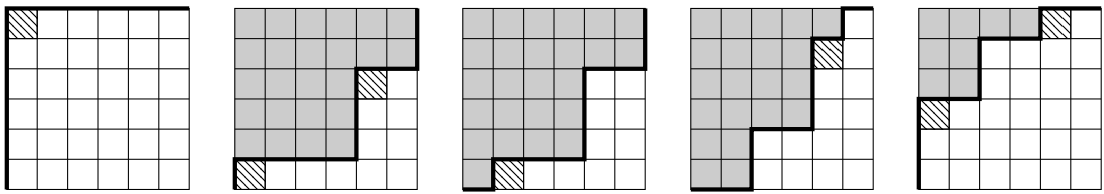
\begin{figure}[!ht]
	\begin{center}
		\begin{tikzpicture}[scale=0.4]
			\foreach \t in {0,...,4}{\draw (7.5*\t,0) rectangle (6+7.5*\t,6);}

			\foreach \y in {1,...,5}{
					\foreach \t in {0,...,4}
						{\draw (7.5*\t,\y) -- (6+7.5*\t,\y);}}

			\foreach \x in {1,...,5}{
					\foreach \t in {0,...,4}
						{\draw (\x+7.5*\t,0) -- (\x+7.5*\t,6);}}

			\draw[ultra thick] (0,0) -- (0,6) -- (6,6);
			\draw[pattern=north west lines] (0,6) rectangle (1,5);

			\draw[ultra thick] (7.5,0) -- (7.5,1) -- (7.5+4, 1) -- (7.5+4,4) -- (7.5+6,4) -- (7.5+6,6);
			\draw[pattern=north west lines] (7.5+4,4) rectangle (7.5+5,3);
			\draw[pattern=north west lines] (7.5,0) rectangle (7.5+1,1);
			\draw[fill, opacity = 0.2]
			(7.5,0) -- (7.5,1) -- (7.5+4, 1) -- (7.5+4,4) -- (7.5+6,4) -- (7.5+6,6) -- (7.5,6) -- (7.5,0) ;

			\draw[ultra thick] (15,0) -- (15+1,0) -- (15+1, 1) -- (15+4,1) -- (15+4,4) -- (15+6,4) -- (15+6,6);
			\draw[pattern=north west lines] (15+1,0) rectangle (15+2,1);
			\draw[fill, opacity = 0.2]
			(15,0) -- (15+1,0) -- (15+1, 1) -- (15+4,1) -- (15+4,4) -- (15+6,4) -- (15+6,6) -- (15,6) -- (15,0) ;

			\draw[ultra thick] (22.5,0) -- (22.5+2,0) -- (22.5+2, 2) -- (22.5+4,2) -- (22.5+4,5) -- (22.5+5,5) -- (22.5+5,6) -- (22.5+6,6);
			\draw[pattern=north west lines] (22.5+5,4) rectangle (22.5+4,5);
			\draw[fill, opacity = 0.2]
			(22.5,0) -- (22.5+2,0) -- (22.5+2, 2) -- (22.5+4,2) -- (22.5+4,5) -- (22.5+5,5) -- (22.5+5,6) -- (22.5,6) -- (22.5,0) ;

			\draw[ultra thick] (30,0) -- (30,3) -- (30+2, 3) -- (30+2,5) -- (30+4,5) -- (30+4,6) -- (30+6,6);
			\draw[pattern=north west lines] (30+5,6) rectangle (30+4,5);
			\draw[pattern=north west lines] (30,3) rectangle (30+1,2);
			\draw[fill, opacity = 0.2]  (30,0) -- (30,3) -- (30+2, 3) -- (30+2,5) -- (30+4,5) -- (30+4,6) -- (30,6) -- (30,0);
		\end{tikzpicture}
	\end{center}
	\caption{Five examples of $K$-even Young diagrams with centers hatched and boundaries thickened.
	}
	\label{Fig:K-even}%
\end{figure}
We order the rows (resp. columns) in $\Xi$ from left to right (resp. bottom to top). Observe that the centers of $K$-even Young diagrams can not locate on thoses boxes in the intersection of even rows and even columns, and therefore the boxes that $K$-even Young diagrams can center form a buffulo-check pattern, see Figure \ref{Fig:Buffalo-check}. A \textit{{buffalo-check} Young diagram} in $\Xi$ is a pair $(\Lambda, c_\Lambda)$ consisting of a $K$-even Young diagram $\Lambda$ in $\Xi$ and a center $c_\Lambda$ of $\Lambda$.\

\begin{figure}[!ht]
	\begin{center}
		\begin{tikzpicture}[scale=0.5]
			\draw (0,0) rectangle (6,6);
			\foreach \y in {1,...,6}{
					\draw (0,\y) -- (6,\y);}
			\foreach \x in {1,...,6}{
					\draw (\x,0) -- (\x,6);}

			\foreach \x in {0,...,2}{
			\foreach \y in {0,...,2}{
			\draw[fill=black] ( (2*\x,2*\y) -- (2*\x + 1,2*\y) -- (2*\x + 1,2*\y + 1) -- (2*\x,2*\y+1) -- (2*\x,2*\y);
			}
			}

			\foreach \x in {0,...,2}{
					\foreach \y in {0,...,2}{
							\draw[fill=black, opacity=0.3] (2*\x+1,2*\y) -- (2*\x + 2,2*\y) -- (2*\x + 2,2*\y + 1) -- (2*\x+1,2*\y+1) -- (2*\x+1,2*\y);
						}
				}

			\foreach \x in {0,...,2}{
					\foreach \y in {0,...,2}{
							\draw[fill=black,opacity=0.3] (2*\x,2*\y+1) -- (2*\x + 1,2*\y+1) -- (2*\x + 1,2*\y + 2) -- (2*\x,2*\y+2) -- (2*\x,2*\y+1);
						}
				}
		\end{tikzpicture}
	\end{center}
	\caption{An example of the buffalo-check pattern of size $(6\times 6)$. The centers of $K$-even Young diagrams can only sit on those non-white boxes. A black box (resp. grey box) is the location of the center of a buffalo-check Young diagram,  if the corresponding basis element lands in the even (resp. odd) twist.
	}
	\label{Fig:Buffalo-check}%
\end{figure}
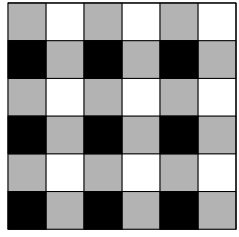

In Section \ref{sect:grassmanian}, we will describe the underlying map of the isomorphism in Theorem \ref{theo:GW-Grassmannian} in light of the combinatorial of even and buffalo-check Young diagrams.
If $S = \Spec(\mathbb{C})$, our result  agrees with \cite{zibrowius2011witt}. To compare, we note that
$$|\mathfrak{B}_{d,m}| = \binom{d+m}{d}-\binom{\lfloor \frac{d}{2} \rfloor +\lfloor \frac{m}{2} \rfloor  }{\lfloor \frac{d}{2} \rfloor}.   $$
Taking the negative homotopy groups, we recover Balmer-Calm\`es's computation on Witt groups of Grassmannians \cite{balmer2012witt}. Hermitian $K$-theory of Grassmannians could also be computed using the method of semi-orthogonal decomposition, as done by Walter in the case of projective bundles. This approach has been pursued in recent work \cite{rohrbach2021on}, where the author addresses half of the cases and assumes that the base is of characteristic zero. It is important to note that these restrictions are not necessary within our method.

The recent advancements in Hermitian $K$-theory, especially when the invertibility of the element two in the base is not assumed, have opened new avenues for exploration, cf. \cite{schlichting2021higher} and  \cite{calmes2023hermitian}. If we assume that two is invertible in the base, then one has the freedom to choose either the framework of Hermitian $K$-theory of dg categories \cite{schlichting2017hermitian} or that of the infinity category \cite{calmes2023hermitian} to work on. It is interesting to investigate our results in this paper when two is not invertible in the base.

\vspace{0.5em}
\noindent \textbf{Convention}.  Throughout this paper, we assume that every scheme is separated Noetherian of finite Krull dimension, over a commutative ring $k$ and has $\frac{1}{2}$ in its global sections.

\section{Hermitian \texorpdfstring{$K$}{K}-theory of schemes}

Let $X$ be a scheme.
Let $\M(X)$ be the category of $\SO_X$-modules, and let $\Ch^b(\M(X))$ be the dg category of bounded complexes of $\SO_X$-modules.\
The pair $$\big(\Ch^b(\M(X)), \mathrm{quis}\big)$$ is a pretriangulated dg category with weak equivalences, where $\mathrm{quis} \subset Z^{0}\Ch^b(\M(X))$ is the set of quasi-isomorphisms of complexes, cf.\ \cite[Section 9.5]{schlichting2017hermitian}.\
It's associated triangulated category
$\mathcal{T}\big(\Ch^b(\M(X)), \mathrm{quis} \big) $
is the derived category $\mathrm{D}^b(\M(X))$.\  Let $I^\bullet$ be an object of $\Ch^b(\M(X))$.\ Define the duality functor on $ \Ch^b(\M(X))$ as
$$\sharp_{I^\bullet}: \Ch^b(\M(X))^{\mathrm{op}} \to \Ch^b(\M(X)): \quad A^\bullet \mapsto  \left[A^\bullet,I^\bullet\right]_{X}  $$
where $[A^\bullet,I^\bullet]_{X}$ is the mapping complex of $\SO_X$-modules defined in \cite[Section 7.2]{schlichting2010mayer}. Define the natural transformation
$\can : 1 \rightarrow \sharp_{I^\bullet}\sharp_{I^\bullet} $
by the formula
\[\can_{A} : A^\bullet \rightarrow (A^\bullet)^{\sharp_{I^\bullet}\sharp_{I^\bullet}} : m \mapsto \big( \can_{A}(m) : f \mapsto (-1)^{|f||m|} f(m) \big). \]
We usually omit the bullet in the notation of complexes, and write $A = A^\bullet$ if no confusion occurs.
\subsection{Coherent Hermitian $K$-theory}

Let $\Ch^b_{c}(\M(X))$ be the full dg subcategory of $\Ch^b(\M(X))$ consisting of bounded complexes of $\SO_X$-modules with coherent cohomology.

\begin{definition}[{\cite[p.304]{hartshorne1966residues}}]
	A \textit{residue complex} on $X$ is a bounded complex $I^\bullet$ of quasi-coherent injective $\SO_X$-modules,  with coherent cohomology, and such that there is an isomorphism
	\[ I^j \cong \bigoplus_{\mu_I(x) = j} i_{x} \left(E(\kappa(x))\right) \]
	for all $j\in \Z$, where $\mu_I: X \to \Z$ is the codimension function (cf. \cite[Section V.7, p.282]{hartshorne1966residues}), $E(\kappa(x))$ is the injective hull of residue field $\kappa(x)$ over the local ring $\SO_{X,x}$ and $i_{x} (E(\kappa(x)))$ is the skyscraper sheaf at $x$ with value $E(\kappa(x))$. Denote the full additive subcategory of $Z^{0} \Ch^b_c(\M(X))$ of residue complexes by $\Res(X)$.
\end{definition}

\begin{definition}[{\cite[p.258]{hartshorne1966residues}}]
	\ A \textit{dualizing complex} $K^\bullet$ on $X$ is a complex of $\SO_X$-modules in $ D^b_{c}(\M(X))$ which is isomorphic to a complex of  injective $\SO_X$-modules
	\[I^\bullet :=  (\cdots  \rightarrow 0 \rightarrow I^m \stackrel{d^m}\longrightarrow I^{m+1}\rightarrow \cdots \rightarrow I^{n-1} \stackrel{d^{n-1}}\longrightarrow I^{n} \rightarrow 0 \rightarrow \cdots )  \in D^b_{c}(\M(X))\]
	such that the morphism of complexes $\can_{A} : A^\bullet \rightarrow (A^\bullet)^{\sharp_{I^\bullet}\sharp_{I^\bullet}}$ is an isomorphism for any $A^\bullet \in  D^b_{c}(\M(X))$.\
\end{definition}

\begin{remark}
	\begin{enumerate}[leftmargin=20pt]
		\item Since we assume that every scheme is noetherian of finite Krull dimension, a residue complex is dualizing cf.\ \cite[Proposition VI.1.1 (c), p.304]{hartshorne1966residues} (Note that a pointwise dualizing complex is dualizing \cite[Proposition V.8.2, p.288]{hartshorne1966residues} in this case).\
		\item The terminology ``dualizing complex'' used by \cite{schlichting2017hermitian}, \cite{gille2007a} and \cite{xie20a} is precisely a degreewise injective dualizing complex.\ 	A dualizing complex $I^\bullet$ is called \textit{minimal} if it is degreewise injective and $I^r$ is an essential extension of $\ker(d^r)$ for all $r \in \Z$.\ Note that a degreewise injective dualizing complex is a residue complex if and only if it is minimal, cf. \cite[Remark 1.16]{gille2007a}.
	\end{enumerate}
\end{remark}

\begin{proposition}
	Suppose that $I$ is a residue complex on $X$.\ Then, the quadruple
	\[
		\left(\Ch^b_{c}(\M(X)), \mathrm{quis}, \sharp_{I}, \can \right)
	\]
	is a dg category with weak equivalences and duality.
\end{proposition}
\begin{proof}
	See \cite[Section 9.5]{schlichting2017hermitian}.
\end{proof}
This result leads to the following definition.
\begin{definition}[{\cite[Definition 9.15]{schlichting2017hermitian}}]
	Let $I$ be a residue complex on $X$. The \textit{$n$-th shifted coherent Grothendieck-Witt spectrum of $X$ with coefficients in $I$} is the $n$-th shifted Grothendieck-Witt spectrum
	\[
		\GW^{[n]}\left(X,I\right) := \GW^{[n]}\big(\Ch^b_{c}(\M(X)), \mathrm{quis}, \sharp_{I}, \can \big).
	\]
	Its homotopy groups are denoted by
	$$\GW^{[n]}_i\left(X,I\right) : =   \pi_i \GW^{[n]}\left(X,I\right)$$
	for $i \in \Z$.
\end{definition}

\begin{definition}
	Let $\iota:Z \hookrightarrow X$ be a closed immersion and $I$ a residue complex on $X$.
	Define $\Ch_{c,Z}^b(\M(X))$ to be the full dg subcategory of $\Ch_c^b(\M(X))$ consisting of objects with cohomology supported in $Z$. The \textit{$n$-th shifted coherent Grothendieck-Witt spectrum of $X$ supported on $Z$ with coefficients in $I$} is the $n$-th shifted Grothendieck-Witt spectrum
	\[
		\GW_Z^{[n]}\left(X,I\right) := \GW^{[n]}\big(\Ch_{c,Z}^b(\M(X)), \mathrm{quis}, \sharp_{I}, \can \big).
	\]
\end{definition}

\begin{theorem}[Localization]
	Suppose that $X$ is a scheme admitting a dualizing complex $I$. Let $\iota:Z \hookrightarrow X$ be a closed immersion, and let $v: U= X - Z \hookrightarrow X$ be its open complement. Then for all $n \in \Z$, the sequence
	\[
		\GW^{[n]}_Z(X, I) \xrightarrow{} \GW^{[n]}(X, I) \xrightarrow{v^\ast} \GW^{[n]}(U, v^\ast I) \xrightarrow{\partial} \GW^{[n]}_Z(X, I)[1]
	\]
	is a distinguished triangle in $\mathcal{SH}$.
\end{theorem}
\begin{proof}
	See	\cite[Theorem 9.19]{schlichting2017hermitian}.
\end{proof}
\subsection{Hermitian $K$-theory of vector bundles}
Let $\Ch^b(\V(X))$ be the full dg subcategory of $\Ch^b(\M(X))$ consisting of complexes of vector bundles over $X$ of finite rank.

\begin{proposition}
	Let $L$ be an line bundle.\ Then, the quadruple
	\[
		\left(\Ch^b\left(\V(X)\right), \mathrm{quis}, \sharp_{L[m]}, \can \right)
	\]
	is a dg category with weak equivalences and duality, where $L[m]$ is the complex consisting of $L$ concentrated in degree $-m$.
\end{proposition}
\begin{proof}
	See \cite[Section 9.1]{schlichting2017hermitian}.
\end{proof}
\begin{definition}[{\cite[Definition 9.1]{schlichting2017hermitian}}]
	Let $X$ be a scheme with an ample family of line bundles and let $L$ be a line bundle on $X$. The \textit{$n$-th shifted Grothendieck-Witt spectrum of $X$ with coefficients in $L[m]$} is the $n$-th shifted Grothendieck-Witt spectrum
	\[
		\GW^{[n]}\left(X,L[m]\right) := \GW^{[n]}\big(\Ch^b\left(\V(X)\right), \mathrm{quis}, \sharp_{L[m]}, \can \big).
	\]
	Its homotopy groups are denoted by
	\[   \GW^{[n]}_i\left(X,L[m]\right) : =   \pi_i \GW^{[n]}\left(X,L[m]\right)  \]
	for $i \in \Z$. When $n=0$, or $L= \SO_X$, we may omit the notation corresponding to it. For instance, $\GW(X)$ (resp. $\GW^{[n]}(X)$) represents $\GW^{[0]}(X, \SO_X)$ (resp.  $\GW^{[n]}(X,\SO_X)$).
\end{definition}
\begin{remark}
	Note that there is a natural stable equivalence $\GW^{[n]}(X,L) \simeq \GW^{[0]}(X,L[n]) $.
\end{remark}

Suppose now that $X$ is a regular scheme, and $L$ is an line bundle.\ Then there is a residue complex $I$ on $X$ resolving $L$, and we denote $\rho: L \to I$ the canonical quasi-isomorphism of complexes.

\begin{lemma}\label{lma:regular_coherent_vector_gw}
	Let $X$ be a regular scheme. Then, the inclusion form functor
	\[
		(\id,\rho): 	\left(\Ch^b\left(\V(X)\right), \mathrm{quis}, \sharp_{L}, \can \right) \to \left(\Ch^b_{c}(\M(X)), \mathrm{quis}, \sharp_{I}, \can \right)
	\]
	with duality compatibility morphism $\rho: \sharp_{L} \to \sharp_{I}$ induces a stable equivalence
	\[
		\rho: \GW^{[n]}\left(X,L\right) \xrightarrow{\simeq} \GW^{[n]}\left(X,I\right).
	\]
	of Grothendieck-Witt spectra.
\end{lemma}
\begin{proof}
	See {\cite[Theorem 9.18]{schlichting2017hermitian}}.
\end{proof}

Let $X$ be a scheme with an ample family of line bundles and let $L$ be a line bundle on $X$. Let $Z \hookrightarrow X$ be a closed immersion. Let $\Ch_Z^b(\V(X))$ be the dg full subcategory of $\Ch^b(\V(X))$ consisting of objects which has cohomology supported on $Z$.
\begin{definition}
	The \textit{$n$-th shifted Grothendieck-Witt spectrum of $X$ supported on $Z$ with coefficients in $L[m]$} is the $n$-th shifted Grothendieck-Witt spectrum
	\[
		\GW_Z^{[n]}\left(X,L[m]\right) := \GW^{[n]}\big(\Ch_Z^b\left(\V(X)\right), \mathrm{quis}, \sharp_{L[m]}, \big).
	\]
\end{definition}

\begin{theorem}[D\'evissage]\label{thm:devissage}
	Let $X$ be a regular  scheme and $\iota: Z \hookrightarrow X$ be a regular immersion of codimension $c$.\ Let $L$ be a line bundle on $X$. Then the pushforward is an equivalence of spectra
	\[
		\iota_\ast: \GW^{[n-c]}\left(Z,\omega_{\iota} \otimes \iota^\ast L\right) \to \GW_Z^{[n]}\left(X,L\right),
	\]
	where $\omega_{\iota}$ is the relative canonical line bundle.
\end{theorem}
\begin{proof}
	See \cite[Theorem 6.1]{xie20a}.
\end{proof}

\begin{theorem}[Localization]\label{coro:loc_seq}
	Let $\iota: Z \hookrightarrow X$ be a regular immersion of regular schemes of codimension $c$, and let $v: U = X \backslash Z \to X$ be its open complement. For any line bundle $L$ on $X$, the sequence
	\[
		\GW^{[n-c]}(Z, \omega_{\iota} \otimes \iota^\ast L) \xrightarrow{\iota_\ast} \GW^{[n]}(X, L) \xrightarrow{v^\ast} \GW^{[n]}(U, L_{U}) \xrightarrow{\partial} \GW^{[n-c]}(Z, \omega_{\iota} \otimes \iota^\ast L)[1]
	\]
	is a distinguished triangle in $\mathcal{SH}$.
\end{theorem}
\begin{proof}
	In view of Theorem \ref{thm:devissage}, the results follows from \cite[Theorem 6.6]{schlichting2017hermitian}.
\end{proof}

\begin{theorem}[Algebraic Bott sequence]\label{theo:bott_seq_L} Let $L$ be a line bundle on $X$. The sequence of spectra
	\[\GW^{[n]}(X,L) \xrightarrow{F} K(X) \xrightarrow{H} \GW^{[n+1]}(X,L) \xrightarrow{\eta \cup } \GW^{[n]}(X,L)[1]\]
	is a distinguished triangle in $\mathcal{SH}$.
\end{theorem}
\begin{proof}
	See \cite[Theorem 6.1]{schlichting2017hermitian}.
\end{proof}
\begin{lemma}\label{lem:FH_cup}
	Let $L,L'$ be line bundles on $X$. The diagram
	\[
		\xymatrix{
		K_{i}(X) \ar[r]^-{H} \ar[d]_-{\cup F(\alpha) }& \GW^{[n]}_{i}(X,L) \ar[d]_-{\cup \alpha}\\
		K_{i+j}(X) \ar[r]^-{H} & \GW^{[n+m]}_{i+j}(X,L \otimes L')
		}
	\]
	commutes for any $\alpha  \in \GW^{[m]}_j(X,L')$. In particular, the composition
	\[
		\GW^{[n]}(X,L') \xrightarrow{F} K(X) \xrightarrow{H} \GW^{[n+1]}(X,L \otimes L')
	\]
	is homotopic to $\beta \cup -$, where $\beta = H(1) \in \GW^{[1]}_0(X,L)$ is the image of $1 \in K_0(X)$ under the the map $H:K(X) \to \GW^{[1]}(X,L)$.
\end{lemma}
\begin{proof}
	The result follows from the multiplicativity of the algebraic Bott sequence, cf. \cite[Proof of Theorem 6.1]{schlichting2017hermitian}.
\end{proof}

\section{Pushforwards for Hermitian \texorpdfstring{$K$}{K}-theory} \label{sec:pushforward}

\subsection{Residue complexes} In this subsection, we review the theory of residue complexes, cf. \cite{hartshorne1966residues} and \cite{conrad2000grothendieck}.

\begin{theorem}[Residue theorem {\cite{hartshorne1966residues}}]\label{thm:residue}
	Let $f: X \to Y$ be a proper morphism.\ Then, there is a functor of additive categories
	\[ f^\Delta: \Res(Y) \to \Res(X)\]
	together with a natural transformation $\Tr: f_\ast f^\Delta \to 1$ (so called trace map) from $\Res(Y)$ to $Z^0 \Ch^b _c(\M(Y))$ such that the following conditions hold
	\begin{enumerate}[label={\rm (T\arabic*)}]
		\item For every residue complex $J \in \mathrm{Res}(Y)$, denote by $\zeta : f_\ast \sharp_{f^\Delta J} \to \sharp_J f_\ast $ the composition $$f_\ast\left[-, f^\Delta J\right]_X \xrightarrow{\phi} \left[f_\ast (-), f_\ast f^\Delta J\right]_Y \xrightarrow{\Tr} \left[f_\ast (-), J\right]_Y$$
		      where $\phi$ is the canonical morphism, cf.\ \cite[(3.1.4), p.86 ]{lipman2009notes}. Then,
		      $$\zeta_A: f_\ast \sharp_{f^\Delta J}(A) \to \sharp_J f_\ast(A)$$ is a quasi-isomorphism for each $f_*$-acyclic complex $A$ in $\Ch^b_c(\M(X))$.
		\item If $f$ and $g$ are two consecutive proper morphisms, then there is a natural isomorphism
		      $$c: f^\Delta g^\Delta \to (gf)^\Delta$$
		      satisfying associativity, together with a commutative diagram
		      \[\xymatrix{g_\ast f_\ast f^\Delta g^\Delta \ar[r]^-{c} \ar[d]_-{\Tr} & (gf)_\ast (gf)^\Delta \ar[d]_-{\Tr}\\
			      g_\ast g^\Delta \ar[r]^-{\Tr} & \id . }\]
	\end{enumerate}
\end{theorem}
\begin{proof}
	For the existence of $f^\Delta$, see \cite[Theorem VI.3.1, p.318]{hartshorne1966residues}. For the existence of $\Tr$, see \cite[Theorem VI.4.2, p.339 and VII.2.1, p.369]{hartshorne1966residues} or \cite[Theorem 3.4.1, p.147]{conrad2000grothendieck}. For the proof of (T1), see \cite[Theorem VII.3.3, p.379]{hartshorne1966residues}. For the proof of (T2), See \cite[Theorem VI.4.2(TRA1), p.340]{hartshorne1966residues}. Note the subtle point is that $(gf)_\ast = g_\ast f_\ast$.
\end{proof}

\subsection{Pushforward} Let $f:X \to Y$ be a proper morphism throughout this section unless otherwise specified. The aim of this subsection is to prove the following result.

\begin{theorem} \label{thm:pushforward_gw}
	Let $J$ be a residue complex on $Y$.\ Then, the dg functor
	\[
		f_\ast: \Ch^b_c(\M(X)) \to \Ch^b_c(\M(Y))
	\]
	induces a map of Grothendieck-Witt spectra
	\[
		f_\ast: \mathrm{\GW}^{[n]}\left(X,f^\Delta J\right) \to \mathrm{\GW}^{[n]}\left(Y,J\right)
	\]
	which respects composition, i.e. if $f$ and $g$ are two consecutive morphisms, then $g_\ast f_\ast = (gf)_\ast$.
\end{theorem}

\begin{remark}
	Let $f:X\to Y$ be a proper map. Then $\mathrm{R}f_\ast$ maps complexes of $\SO_X$-modules with coherent cohomology to complexes of $\SO_Y$-modules with coherent cohomology, as shown in \cite[Chapter III Theorem 3.2.1]{ega31}.
\end{remark}

\begin{lemma}\label{lem:push-can}
	Let $f: X \to Y$ be a morphism of schemes, and let $\mathcal{F}$ and $\mathcal{G}$ be $\SO_X$-modules.\ Then, the following diagram
	\[
		\xymatrix{f_\ast \mathcal{F} \ar[r]^-{f_\ast \can} \ar[d]_-{\can} & f_\ast \sHom_X(\sHom_X(\mathcal{F},\mathcal{G}),\mathcal{G}) \ar[d]_{\phi}\\
		\sHom_Y(\sHom_Y(f_\ast \mathcal{F}, f_\ast \mathcal{G}), f_\ast \mathcal{G}) \ar[r]^-{\phi^\sharp} & \sHom_Y(f_\ast \sHom_X(\mathcal{F},\mathcal{G}), f_\ast \mathcal{G})}
	\]
	commutes in $\M(X)$.
\end{lemma}
\begin{proof}
	We check the commutativity locally by sections. Let $V$ be an open subset of $Y$. We need to check the diagram
	\[
		\xymatrix{\mathcal{F}(f^{-1}V) \ar[r]^-{\can} \ar[d]_-{\can} & \Hom_{f^{-1}V}(\sHom_X(\mathcal{F},\mathcal{G})|_{f^{-1}V},\mathcal{G}|_{f^{-1}V}) \ar[d]_{\phi}\\
		\Hom_V(\sHom_Y(f_\ast \mathcal{F}, f_\ast \mathcal{G})|_V, f_\ast \mathcal{G}|_V) \ar[r]^-{\phi^\sharp} & \Hom_V(f_\ast \sHom_X(\mathcal{F},\mathcal{G})|_V, f_\ast \mathcal{G}|_V)}
	\]
	commutes.\

	Let $s$ be a section in $\mathcal{F}(f^{-1}V)$. We check that $\phi(\can(s)) = \phi^\sharp(\can(s))$.\ On the left-hand side, the morphism of $\SO_V$-modules $\phi(\can(s)): f_\ast \sHom_X(\mathcal{F},\mathcal{G})|_V \to f_\ast \mathcal{G}|_V $ may be described as follows. Take an open subset $U \subseteq V$, $\phi(\can)(s)(U)$ sends an element $t \in \Hom_{f^{-1}(U)}(\mathcal{F}|_{f^{-1}U},\mathcal{G}|_{f^{-1}U})$ to $t(f^{-1}U)(s|_{f^{-1}U})$ in $\mathcal{G}(f^{-1}U)$. For the right-hand side, the morphism $\SO_V$-modules $\phi^\sharp(\can(s)): f_\ast \sHom_X(\mathcal{F},\mathcal{G})|_V \to f_\ast \mathcal{G}|_V $ sends $t$ to $\phi(t)(f^{-1}U)(s|_{f^{-1}U})$.\ The morphism $\phi(t) (f^{-1}U)$ is equal to $t(f^{-1}U)$ in $\Hom_{\SO_Y(f^{-1}U)}(\mathcal{F}(f^{-1}(U)), \mathcal{G}(f^{-1}(U)))$, cf. \cite[p. 86 (3.1.4)]{lipman2009notes}.
\end{proof}

\begin{lemma}\label{lem:mor-can}
	Let $\mathcal{F}, \mathcal{G}$ and $\mathcal{G}'$ be $\SO_X$-modules, and let $\psi :\mathcal{G}\to \mathcal{G}' $ be a morphism of $\SO_X$-modules. Then, the following diagram
	\[
		\xymatrix{\mathcal{F} \ar[r]^-{\can} \ar[d]_-{\can} & \sHom_X(\sHom_X(\mathcal{F},\mathcal{G}),\mathcal{G}) \ar[d]_-{\psi}\\
		\sHom_X(\sHom_X(\mathcal{F},\mathcal{G}'),\mathcal{G}') \ar[r]^{\psi^\sharp} & \sHom_X(\sHom_X(\mathcal{F},\mathcal{G}),\mathcal{G}')}
	\]
	commutes.
\end{lemma}
\begin{proof}
	Exercise.
\end{proof}

\begin{proposition}\label{lma:ac_zeta}
	The pushforward functor
	\[
		f_\ast: \Ch^b_c(\M(X)) \to \Ch^b_c(\M(Y))
	\]
	induces a dg form functor
	\[ (f_\ast, \zeta) : (\Ch^b_c(\M(X)), \sharp_{f^\Delta J},\can) \to ( \Ch^b_c(\M(Y)) , \sharp_J, \can).   \]
\end{proposition}
\begin{proof}
	It is enough to check that the following diagram
	\[
		\xymatrix{f_\ast A \ar[r]^-{f_\ast \left(\mathrm{can}\right)} \ar[d]_-{\mathrm{can}}& f_\ast \left[\left[A,f^\Delta J\right],f^\Delta J\right] \ar[d]_-{\zeta_{A^\sharp}}\\
		\left[\left[f_\ast A,J\right],J\right] \ar[r]^-{\zeta_A^\sharp} & \left[f_\ast \left[A,f^\Delta J\right],J\right]}
	\]
	commutes for every $A \in \Ch^b_c(\M(X))$. For the commutativity of this diagram, we depict the all the maps by definition in the following diagram.

	$$
		\xymatrix{f_\ast A \ar@{=}[d] \ar[rr]^-{f_\ast \left(\mathrm{can}\right)}  \ar@{}[drr]!<50pt,0pt>|-{\diagram \label{diag:push_dual_1}} && f_\ast \left[\left[A,f^\Delta J\right],f^\Delta J\right] \ar[d]_-{\phi}\\
		f_\ast A \ar[d]_-{\mathrm{can}} \ar[r]^-{\mathrm{can}} \ar@{}[dr]|-{\diagram \label{diag:push_dual_2}}& \left[\left[f_\ast  A,f_\ast f^\Delta J\right],f_\ast f^\Delta J\right] \ar[d]_-{\Tr} \ar[r]^-{\phi^\sharp} \ar@{}[dr]|-{\diagram \label{diag:push_dual_3}}& \left[f_\ast \left[A,f^\Delta J\right],f_\ast f^\Delta J\right]\ar[d]_-{\Tr} \\
		\left[\left[f_\ast A,J\right],J\right] \ar[r]^-{{\Tr}^\sharp} & \left[\left[f_\ast  A,f_\ast f^\Delta J\right], J\right] \ar[r]^-{\phi^\sharp} & \left[f_\ast \left[A,f^\Delta J\right],J\right]. }
	$$
	The diagram $\diag{\ref{diag:push_dual_1}}$ (resp. $\diag{\ref{diag:push_dual_2}}$) is commutative by Lemma \ref{lem:push-can} (resp. Lemma \ref{lem:mor-can}). The diagram $\diag{\ref{diag:push_dual_3}}$ commutes by naturality.
\end{proof}
\begin{remark}
	If $f$ is finite, the pushforward functor is exact and $\zeta$ is a quasi-isomorphism. Consequently, the dg form functor $(f_\ast, \zeta)$ is non-singular and exact in the sense of \cite{schlichting2010mayer}. Thus, the dg form functor $(f_\ast, \zeta)$ induces a map of Grothendieck-Witt spectra $\GW^{n}(X, f^\Delta J) \to \GW^{[n]}(Y, J)$. However, if $f$ is not finite, the above statement may not hold true. This limitation is evident in the special case of Witt groups. To construct the pushforward for Witt groups, one needs to use the right-derived functor $\mathrm{R} f_\ast$, as noted in \cite[Theorem 4.4]{calmes2011push}. In the context of higher algebraic $G$-theory, one might employ a technique from Thomason-Trobaugh \cite[3.16]{thomason1990higher} to construct a pushforward of $K$-theory as the composition:
	\[
		G(X) \xleftarrow{\cong} K(\mathcal{F}_X) \xrightarrow{f_\ast} G(Y),
	\]
	where $\mathcal{F}_X \subset \Ch^b_c( \M(X))$ denotes the category of complexes of flasque $\SO_X$-modules. Regrettably, the category $\mathcal{F}_X$ may not be closed under the duality $\sharp_{f^\Delta J}$. The issue persists even after replacing $\mathcal{F}_X$ with the subcategory of $\Ch^b_c( \M(X))$ consisting of degree-wise injective (or $f_\ast$-acyclic) complexes, as considered in \cite[1.13]{thomason1986lefschetz}. To overcome this problem, we introduce the subsequent construction.
\end{remark}
\begin{definition}\label{def:acyclic_dg}
	The full dg subcategory $\AC_f$ of $\Ch^b_{c}(\M (X))$ is defined to be
	\[
		\AC_f := \left\{ A \in \Ch^b_{c}(\M (X)) \,\middle\vert\, \theta_A: f_\ast (A) \to \mathrm{R}f_\ast (A) \text{ is an isomorphism. } \right\}
	\]
	where $\theta: f_\ast \to \mathrm{R}f_\ast$ is the canonical natural transformation from $\Ch^b_{c}(\M (X))$ to $\mathrm{D}^b_c(\M (Y))$.
\end{definition}
\begin{lemma}\label{lma:ac_pretriangulated}
	The full dg subcategory $\AC_f$ is pretriangulated.
\end{lemma}
\begin{proof}
	We need to check that $ H^0 \AC_f$ is triangulated with respect to the split exact triangulated structure.\	It's clear that $0 \in H^0 \AC_f$ and if $A_1, A_2 \in \AC_f$, so is $A_1 \oplus A_2$. Therefore $H^0\AC_f$ is additive. It's clear that if $A \in \AC_f$, so is $A[1]$.
	Next, for any $A_1 \simeq A_2 \in \mathrm{K}^b_{c} (\M(X))$ such that $A_1 \in \AC_f$, the diagram
	\[
		\xymatrix{f_\ast (A_1) \ar[r]_-{\simeq} \ar[d]_{\theta}^-{\simeq} & f_\ast (A_2) \ar[d]_{\theta} \\ \mathrm{R}f_\ast(A_1) \ar[r]_-{\simeq} & \mathrm{R}f_\ast(A_2)}
	\]
	commutes and thus $A_2 \in \AC_f$.\ For any distinguished triangle $(A_1,A_2,A_3,f,g,h)$ in $\mathrm{K}^b_{c} (\M(X))$ such that $A_1, A_2 \in H^0\AC_f$, the canonical natural transformation $\theta: f_\ast \to \mathrm{R}f_\ast$  induces the following map of distinguished triangles
	\[
		\xymatrix{f_\ast(A_1) \ar[r]^-{f_\ast(f)} \ar[d]_{\theta} & f_\ast(A_2) \ar[r]^-{f_\ast(g)} \ar[d]_{\theta} & f_\ast(A_3) \ar[r]^-{f_\ast(h)} \ar[d]_{\theta} & f_\ast(A_1)[1] \ar[d]_{\theta} \\
		\mathrm{R}f_\ast(A_1) \ar[r]^-{\mathrm{R}f_\ast(f)} & \mathrm{R}f_\ast(A_2) \ar[r]^-{\mathrm{R}f_\ast(g)} & \mathrm{R}f_\ast(A_3) \ar[r]^-{\mathrm{R}f_\ast(h)} & \mathrm{R}f_\ast(A_1)[1]. }
	\]
	Consequently, since all the other vertical arrows in the diagram are isomorphisms, the map $ \theta: f_\ast(A_3) \to \mathrm{R}f_\ast(A_3) $ is an isomorphism.
\end{proof}

\begin{lemma}\label{lma:ac_preserve}
	Let $I$ be a residue complex on $X$, then $\sharp_I(\AC_f) \subset \AC_f$, i.e.\ the duality functor $\sharp_{I}$ preserves the dg category $\AC_f$.
\end{lemma}
\begin{proof}
	Since $I$ is injective and bounded, $I$ is $K$-injective. For every object $A$ of $\Ch^b(\M(X))$, the mapping complex $[A,I]_X$ is weakly $K$-injective \cite[Proposition 5.14]{spaltenstein1988resolutions}, thus the map $\theta:f_\ast(\sharp_I(A)) \xrightarrow{\cong} \mathrm{R} f_\ast(\sharp_I(A))$ is an isomorphism, cf.\ \cite[Proposition 6.7]{spaltenstein1988resolutions}, i.e. $\sharp_{I}(A) \in \AC_f$.
\end{proof}
\begin{remark}\label{rmk:weakly-K-injective-dual}
	Note that the above proof implies that the full dg subcategory $\mathcal{W}_X$ of weakly $K$-injective complexes in $\Ch^b_c(\M(X))$ is also closed under the duality $\sharp_I$.
\end{remark}

\begin{proof}[Proof of Theorem \ref{thm:pushforward_gw}]
	By Lemma \ref{lma:ac_pretriangulated} and \ref{lma:ac_preserve}, the quadruple $(\AC_f, \mathrm{quis}, \sharp_{f^\Delta J}, \mathrm{can})$ forms a pretriangulated dg category with weak equivalences and duality.\ Furthermore, the composition of dg form functors
	$$(\AC_f, \quis, \sharp_{f^\Delta J}) \hookrightarrow (\Ch^b_c(\M(X)), \quis, \sharp_{f^\Delta J}) \xrightarrow{(f_\ast, \zeta)} (\Ch^b_{c}(\M (Y)),\quis, \sharp_J)$$
	is non-singular and exact, and therefore induces a map
	\[
		f_\ast: \GW^{[n]}\left(\AC_f, \quis, \sharp_{f^\Delta J}\right) \to \GW^{[n]}\left(Y,J\right)
	\]
	of Grothendieck-Witt spectra \cite[Section 2.8]{schlichting2010mayer}.

	We claim that the inclusion of dg categories with weak equivalences $\AC_f \hookrightarrow \Ch^b_{c}(\M (X))$ induces an equivalence $\mathcal{T}(\AC_f, \mathrm{quis}) \xrightarrow{\simeq} \mathrm{D}^b_c (\M(X))$ of associated triangulated categories.\ To see the essential surjectivity, for any object $A$ in $\mathrm{D}^b_c (\M(X))$, we have a canonical quasi-isomorphism $\can: A \to \sharp_I \sharp_I(A)$ for any residue complex $I$ on $X$, and we could take $I = f^\Delta J$.
	It follows that the inclusion $\AC_f \hookrightarrow \Ch^b_{c}(\M (X))$ induces a stable equivalence
	\[
		\GW^{[n]}\left(\AC_f, \sharp_{f^\Delta J}\right) \xrightarrow{\simeq} \GW^{[n]}\left(X,f^\Delta J\right)
	\]
	of Grothendieck-Witt spectra by \cite[Theorem 6.5]{schlichting2017hermitian}.

	The required pushforward morphism
	\[
		f_\ast: \GW^{[n]}\left(X,f^\Delta J\right) \to \GW^{[n]}\left(Y,J\right)
	\]
	is defined to be the composition
	\[
		\GW^{[n]}\left(X,f^\Delta J\right) \xleftarrow{\simeq} \GW^{[n]}\left(\AC_f, \sharp_{f^\Delta J}\right) \xrightarrow{f_\ast} \GW^{[n]}\left(Y,J\right).
	\]
	The fact that the pushforward morphism respects composition follows from Proposition \ref{prop:composition-push-forward} below.
\end{proof}

\begin{proposition}[Composition of pushforward]\label{prop:composition-push-forward}
	Let $f:X\to Y$ and $g:Y \to Z$ be proper morphisms. Let $K$ be a residue complex on $Z$.	Then, the following diagram
	\[
		\xymatrix{\GW^{[n]}\left(X,f^\Delta g^\Delta K\right) \ar[r]^-{c}_{\simeq} \ar[d]_-{f_\ast} & \GW^{[n]}\left(X,(gf)^\Delta K\right)  \ar[d]_-{(gf)_\ast}\\
		\GW^{[n]}\left(Y,g^\Delta K\right)	\ar[r]^-{g_\ast} & \GW^{[n]}\left(Z,K\right)}
	\]
	commutes.
\end{proposition}
\begin{proof}
	Let $\mathcal{W}_X$ be the full dg subcategory of $\Ch^b_{c}(\M(X))$ consisting of objects that are weakly $K$-injective.\ Note that $\mathcal{W}_X$ is pretriangulated, cf.\ \cite[Proposition 5.15 (a)]{spaltenstein1988resolutions}.\ By Remark \ref{rmk:weakly-K-injective-dual}, the category $\mathcal{W}_X$ is closed under the duality $\sharp_I$ for any residue complex $I$ on $X$.\ Moreover, the inclusion $\mathcal{W}_X \subset \Ch^b_{c}(\M(X))$ induces an equivalence $\mathcal{T}(\mathcal{W}_X, \mathrm{quis}) \xrightarrow{\simeq} \mathrm{D}^b_c (\M(X))$ of the associated triangulated category.
	By \cite[Proposition 5.15 (b)]{spaltenstein1988resolutions}, the functor $f_\ast$ maps $\mathcal{W}_X$ into $\mathcal{W}_Y\subseteq \AC_{g}$.

	To show the commutativity claimed in the proposition, we consider the following diagram

	$$
		\xymatrix{(\Ch^b_c (\M(X)), \sharp_{f^\Delta g^\Delta K}) \ar[rr]^-{(\id, c)} && (\Ch^b_c (\M(X)), \sharp_{(gf)^\Delta K})\\
		(\AC_{f}, \sharp_{f^\Delta g^\Delta K})  \ar[d]_-{(f_\ast, \zeta)} \ar@{^{(}->}[u]& \ar@{}[dr]|-{\diagram{} \label{diag:push_comp}} (\mathcal{W}_X, \sharp_{f^\Delta g^\Delta K}) \ar[d]_-{(f_\ast, \zeta)} \ar[r]^-{(\id, c)} \ar@{_{(}->}[l] & (\mathcal{W}_X, \sharp_{(gf)^\Delta K}) \ar[d]_-{((gf)_\ast, \zeta)} \ar@{^{(}->}[u]\\
		(\Ch^b_c (\M(Y)), \sharp_{g^\Delta K}) & (\AC_{g}, \sharp_{g^\Delta K}) \ar[r]^-{(g_\ast, \zeta)} \ar@{_{(}->}[l] & (\Ch^b_{c}(\M (Z)), \sharp_K). }
	$$
	It is enough to prove that the square $\diag{\ref{diag:push_comp}}$ commutes, that is to show the following diagram
	\[
		\xymatrix{(gf)_\ast \left[A, f^\Delta g^\Delta K\right] \ar[r]^-{c} \ar[d]_-{\zeta} & (gf)_\ast \left[A, (gf)^\Delta K\right] \ar[d]_-{\zeta} \\
		g_\ast \left[f_\ast A, g^\Delta K\right] \ar[r]^-{\zeta} & \left[g_\ast f_\ast A, K\right]}
	\]
	commutes for any $A \in \mathcal{W}_X$, where $\zeta = \Tr \phi$, cf.\ Theorem \ref{thm:residue}(T1). This follows directly from the following the diagram
	$$
		\xymatrix{g_\ast f_\ast \left[A, f^\Delta g^\Delta K\right] \ar[d]_-{\phi} \ar@{=}[r] & (gf)_\ast \left[A, f^\Delta g^\Delta K\right] \ar[d]_-{\phi} \ar[r]^-{c} & (gf)_\ast \left[A, (gf)^\Delta K\right] \ar[d]_-{\phi}\\
		g_\ast \left[f_\ast A, f_\ast f^\Delta g^\Delta K\right] \ar[d]_-{\Tr} \ar[r]^-{\phi} & \left[g_\ast f_\ast A, g_\ast f_\ast f^\Delta g^\Delta K\right] \ar[d]_-{\Tr} \ar[r]^-{c} \ar@{}[dr]|-{\diagram \label{diag:push_comp_dual}} & \left[(gf)_\ast A, (gf)_\ast (gf)^\Delta K\right] \ar[d]_-{\Tr}\\
		g_\ast \left[f_\ast A, g^\Delta K\right] \ar[r]^-{\phi} & \left[g_\ast f_\ast A, g_\ast g^\Delta K\right] \ar[r]^-{\Tr} & \left[g_\ast f_\ast A, K\right]}
	$$
	where $\diag{\ref{diag:push_comp_dual}}$ commutes by Theorem \ref{thm:residue}(T2).
\end{proof}

\begin{lemma}\label{lma:push_FH}
	Let $f:X \to Y$ be a proper morphism and let $J$ be a residue complex on $Y$, then the following diagram of spectra
	\[
		\xymatrix{
		G(X) \ar[r]^-{H} \ar[d]_-{f_\ast} & \GW^{[n]}(X,f^\Delta J) \ar[d]_-{f_\ast}\\
		G(Y) \ar[r]^-{H} & \GW^{[n]}(Y,J)
		}
	\]
	commutes, where $G(X): = K(\Ch^b_c(\M(X)))$ and $H$ is the hyperbolic functor (cf. \cite[Theorem 6.1]{schlichting2017hermitian}).
\end{lemma}
\begin{proof}
	This follows from the following diagram
	\[
		\xymatrix{
		G(X) \ar[d]_-{H} & K(\AC_f) \ar[l] \ar[r]^-{f_\ast} \ar[d]_-{H} & G(Y) \ar[d]_-{H} \\
		\GW^{[n]}(X,f^\Delta J) & \GW^{[n]}(\AC_f,\sharp_{f^\Delta J}) \ar[l] \ar[r]^-{f_\ast} & \GW^{[n]}(Y,J)
		}
	\]
	where both squares commute by the functoriality of algebraic Bott sequences (cf. \cite[Section 6]{schlichting2017hermitian}).
\end{proof}

\section{Pullbacks for Hermitian \texorpdfstring{$K$}{K}-theory}\label{sec:pullback}
\label{sec:flat_pullback}

\subsection{On the functor $\mathrm{E}$}
Let $X$ be a scheme admitting a dualizing complex $K$, and $\mu_{K}:X \to \Z$ be the associated codimension function.

Let $\mathcal{F}$ be a sheaf on $X$. Suppose that $Z$ is a closed subset of $X$, and $U$ is an open subset of $X$. Recall the sections of $\mathcal{F}$ on $U$ with supports in $Z$ are defined to be
$$\Gamma_{Z}(U, \mathcal{F}): = \Big\{ s \in \Gamma(U) \mid \mathrm{Supp}(s) \subseteq Z \cap U \Big\}$$
where $\mathrm{Supp}(s) := \{x \in U \mid s_x \neq 0\}$. Define $\uG_Z(\mathcal{F})$ to be the sheaf
$U \mapsto \Gamma_Z(U, \mathcal{F})$. \footnote{ The presheaf $\uG_Z(\mathcal{F})$ is already a sheaf \cite{hartshorne1967local}. }

Let $Z^i:= \{x \in X \mid \mu_K(x) \geq i \}$.  Assume further that $K$ is represented by a complex $I$ of injective $\SO_X$-modules.\ Then, we have a filtration of closed subsets
$$ \emptyset= Z^m \subseteq \cdots \subseteq Z^{p+1} \subseteq Z^p \subseteq Z^{p-1} \subseteq \cdots \subseteq Z^n = X, $$
and a filtration of subcomplexes of $\SO_X$-modules
$$ \cdots \subseteq \uG_{Z_{p+1}}(I^\bullet) \subseteq \uG_{Z_p}(I^\bullet) \subseteq \uG_{Z_{p-1}}(I^\bullet) \subseteq \cdots . $$
Define
$$\uG_{Z^{p}/Z^{p+1}}(I^{p+q}) := \uG_{Z_p}(I^{p+q})/\uG_{Z_{p+1}}(I^{p+q})  $$
for any $p,q \in \Z$. The differential on $I$ induces a morphism of complexes
$$ \cdots  \rightarrow \uG_{Z^{p}/Z^{p+1}}(I^{p+q-1}) \xrightarrow{\delta^{p,q-1}} \uG_{Z^{p}/Z^{p+1}}(I^{p+q}) \xrightarrow{\delta^{p,q}} \uG_{Z^{p}/Z^{p+1}}(I^{p+q+1}) \rightarrow \cdots  $$
\begin{definition} The local cohomology sheaf $\uH^{p+q}_{Z^p/Z^{p+1}}(I^\bullet)$ is defined to be $\ker(\delta^{p,q})/ \im(\delta^{p,q-1})$.
\end{definition}
\begin{definition}[{\cite[p.225 Variation 8]{hartshorne1966residues}}]
	Let $x\in X$ and $\mathcal{F}$ be an $\SO_X$-module. Recall that $\Gamma_x(\mathcal{F})$ consists of elements $\overline{s}$ in the stalk $\mathcal{F}_x$, which have a representative $s \in \Gamma(U,\mathcal{F}|_U)$  in a suitable neighborhood $U$ of $x$, whose support is $\overline{\{x\}} \cap U$. Denote by $\mathrm{H}^{q}_x(I^\bullet) := \mathrm{H}^{q} (\Gamma_x(I^\bullet))$.
\end{definition}
\begin{theorem}
	There is a canonical functorial isomorphism
	\[\uH^{p+q}_{Z^p/Z^{p+1}}(I^\bullet) \cong \coprod_{\mu(x)=p}i_x(\mathrm{H}^{p+q}_x(I^\bullet))\]
\end{theorem}
\begin{proof}
	See \cite[p.225 Motif F]{hartshorne1966residues}.
\end{proof}
Notice that the differential on $\delta^{p,q}$ induces a well-defined map $\ker(\delta^{p,q}) \to \ker(\delta^{p+1,q})$, which induces a map on local cohomology
$$d : \uH^{p+q}_{Z^p/Z^{p+1}}(I^\bullet) \to \uH^{p+q+1}_{Z^{p+1}/Z^{p+2}}(I^\bullet) $$
which yields a complex
$$  \cdots \rightarrow \uH^{p+q-1}_{Z^{p-1}/Z^{p}}(I^\bullet) \xrightarrow{d} \uH^{p+q}_{Z^p/Z^{p+1}}(I^\bullet) \xrightarrow{d} \uH^{p+q+1}_{Z^{p+1}/Z^{p+2}}(I^\bullet) \rightarrow \cdots $$
When $q=0$, we have the following definition.
\begin{definition}[{\cite[p.241]{hartshorne1966residues}}]
	The Cousin complex $E(K)$ of dualizing complex $K$ is defined to be the residue complex
	$$ \cdots \rightarrow \coprod_{\mu(x)=p-1}i_x(\mathrm{H}^{p-1}_x(I^\bullet)) \xrightarrow{d} \coprod_{\mu(x)=p}i_x(\mathrm{H}^{p}_x(I^\bullet)) \xrightarrow{d} \coprod_{\mu(x)=p+1}i_x(\mathrm{H}^{p+1}_x(I^\bullet)) \rightarrow \cdots $$
\end{definition}

\begin{remark}
	Let $\Dual(X)_{Z^\bullet}$ be the category of dualizing complexes in $D^b_c(\M(X))$ with respect to the filtration $Z^\bullet$ (cf. \cite[p.284]{hartshorne1966residues}), and let $\Res(X)_{Z^\bullet}$ be the additive category of residue complexes. Then, the functor
	$$\mathrm{E}: \Dual(X)_{Z^\bullet} \to \Res(X)_{Z^\bullet}$$
	that sends a dualizing complex to its Cousin complex is an equivalence of categories, cf.\ \cite[Proposition VI.1.1(c), p.304]{hartshorne1966residues} (and \cite[Lemma 3.2.1]{conrad2000grothendieck} for clarification)).\ There is a canonical natural isomorphism $\id \to E$ from $\Dual(X)_{Z^\bullet}$ to itself, cf.\ \textit{loc.\ cit.}.
\end{remark}

\begin{lemma}\label{lma:lifting-cousin}
	Let $f: X \to Y$ be a flat morphism.\ Let $J,\tilde{J} \in \Res(Y)_{Z^\bullet}$.\ Assume that $f^\ast J$ is a dualizing complex on $X$.\ Then, $f^*\tilde{J}$ is a dualizing complex with the same associated codimension function as $f^\ast J$.\ Moreover, the canonical morphism
	\[
		\Hom_{Z^0\Ch^b_c(\M(X))}(f^\ast \tilde{J}, Ef^\ast J) \to \Hom_{D^b_c(\M(X))}(f^\ast \tilde{J}, Ef^\ast J).
	\]
	is a bijection.
\end{lemma}
\begin{proof}
	Since $f^\ast J$ is a dualizing complex, $Ef^\ast J$ is degreewise injective and bounded. Thus, the canonical morphism
	\[
		\Hom_{K^b_c(\M(X))}(f^\ast \tilde{J}, Ef^\ast J) \to	\Hom_{D^b_c(\M(X))}(f^\ast \tilde{J}, Ef^\ast J).
	\]
	is an isomorphism, cf. \cite[Corollary 10.4.7]{weibel1994an}. All we need to do is to prove that
	\[
		\Hom_{X}((f^\ast \tilde{J})^{p}, (Ef^\ast J)^{p-1}) = 0
	\]
	for any $p \in \Z$.

	The question is local on $X$.\ Assume that $X=\Spec(S)$ and $Y=\Spec(R)$. By assumption,  $J^\bullet\otimes_R S$ is a dualizing complex on $X$.\ Let $E(J^\bullet\otimes_R S)$ be the associated Cousin complex of the $S$-module $J^\bullet\otimes_R S$.\ Denote the codimension function of the dualizing complex $J^\bullet$ (resp. $J^\bullet\otimes_R S$) by $\mu: \Spec(S) \to \Z$ (resp. $\mu ': \Spec(R) \to \Z$).

	Take any morphism $t \in \Hom_S (\tilde{J}^p \otimes_R S, E^{p-1}(J^\bullet\otimes_R S))$. Suppose that $t\neq 0$, then there exists $r \in \tilde{J}^p$ with $\mathrm{Supp}(rR) = V(\mathfrak{p})$ for some prime ideal $\mathfrak{p}$ in $R$ satisfying $\mu(\mathfrak{p})=p$, and $t(r\otimes 1) \neq 0$. It follows that the support of $t(r\otimes 1)$ is a finite union $V(\mathfrak{q}_1) \cup \ldots \cup V(\mathfrak{q}_n)$ with
	\begin{equation}\label{eq:contra}
		\mu'(\mathfrak{q}_i) = p-1
	\end{equation}
	and $\mathfrak{q}_i \in \Spec(S)$.\ Since $f$ is flat, the support of $(r\otimes 1) S \subseteq J^p \otimes_R S$ is $\mathrm{Supp}(rR \otimes_R S)$, which is $f^{-1}\mathrm{Supp}(rR)=V(\mathfrak{p}S)$, cf.\ \cite[Proposition 19, p.107]{bourbaki1972elements}.
	Let $\mathfrak{q}$ be a minimal prime ideal in $S$ containing $\mathfrak{p} S$. Since $f$ is flat, $\mathfrak{p} = f^{-1}(\mathfrak{q})$ by the going-down property (cf. \cite[Lemma 10.11, p.239]{eisenbud1995commutative}). Note that $S_{\mathfrak{q}} / \mathfrak{p} S_{\mathfrak{q}}$ is an artinian local ring.\ Let $\mathfrak{m}$ be the maximal ideal of $S_{\mathfrak{q}} / \mathfrak{p} S_{\mathfrak{q}}$. Note that $\mathfrak{m}$ is an $S_{\mathfrak{q}}$-module of finite length. As $S_{\mathfrak{q}}$-modules, $\mathfrak{m} \cong \mathfrak{q} S_{\mathfrak{q}} / \mathfrak{p} S_{\mathfrak{q}}$ and $S_{\mathfrak{q}} / \mathfrak{p} S_{\mathfrak{q}} \cong \kappa(\mathfrak{p}) \otimes_{R_{\mathfrak{p}}} S_{\mathfrak{q}}$.

	Set $\omega := (J^\bullet)_{\mathfrak{p}} \otimes_{R_{\mathfrak{p}}} S_{\mathfrak{q}}$.\ It is a dualizing complex, since the localization of a dualizing complex is dualizing, cf.\ \cite[Lemma V.2.3, p.259]{hartshorne1966residues}. The short exact sequence of $S_{\mathfrak{q}}$-modules
	\[
		0 \to \mathfrak{m} \to S_{\mathfrak{q}} / \mathfrak{p} S_{\mathfrak{q}} \to \kappa(\mathfrak{q}) \to 0
	\]
	induces a long exact sequence
	\[
		\begin{aligned}
			\cdots \to \ext^{q-1}_{S_{\mathfrak{q}}}(S_{\mathfrak{q}} / \mathfrak{p} S_{\mathfrak{q}}, \omega) \to \ext^{q-1}_{S_{\mathfrak{q}}}(\mathfrak{m}, \omega) \to \ext^{q}_{S_{\mathfrak{q}}}(\kappa(\mathfrak{q}), \omega) \to \ext^{q}_{S_{\mathfrak{q}}}(S_{\mathfrak{q}} / \mathfrak{p} S_{\mathfrak{q}}, \omega) \\ \to \ext^{q}_{S_{\mathfrak{q}}}(\mathfrak{m}, \omega) \to  \cdots
		\end{aligned}
	\]
	Note that
	\[
		\ext^{q}_{S_{\mathfrak{q}}}(S_{\mathfrak{q}} / \mathfrak{p} S_{\mathfrak{q}}, \omega) \cong \ext^{q}_{R_{\mathfrak{p}}}(\kappa(\mathfrak{p}), (J^\bullet)_{\mathfrak{p}}) \otimes_{R_{\mathfrak{p}}} S_{\mathfrak{q}} = 0, \quad \text{ for all } q\neq p.
	\]
	Hence $\ext^{q}(\kappa(\mathfrak{p}), \omega)  \cong \ext^{q-1}(\mathfrak{m}, \omega)$ for all $q<p$ and $\ext^{q+1}(\kappa(\mathfrak{p}), \omega)  \cong \ext^{q}(\mathfrak{m}, \omega)$ for all $q>p$.
	Since $\omega$ is a dualizing complex, $\ext^{q}(\kappa(\mathfrak{p}), \omega) = 0$ and $\ext^{q}(\mathfrak{m}, \omega) = \Hom^q(\mathfrak{m}, E(\kappa(\mathfrak{q}))[\mu'(\mathfrak{q})]) = 0$ if $q \neq \mu'(\mathfrak{q})$ (cf. \cite[\href{https://stacks.math.columbia.edu/tag/0A7Q}{Lemma 0A7Q}]{stacks-project}).
	It follows that $\ext^{q}(\kappa(\mathfrak{p}), \omega) = 0$ for all $q \neq p$. Thus $\mu'(\mathfrak{q}) = p$.

	Therefore for any $\mathfrak{q}$ in $V(\mathfrak{p} S)$, $\mu'(\mathfrak{q}) \geq p$ (cf. \cite[proposition V.7.1]{hartshorne1966residues}) and the equality holds if and only if $\mathfrak{q}$ is a minimal prime ideal containing $\mathfrak{p} S$. Since $ \mathrm{Supp}(t(r\otimes1)) \subseteq \mathrm{Supp}(r\otimes1) \subseteq V(\mathfrak{p}S)$,  this contradicts Equation \eqref{eq:contra}.
\end{proof}

\begin{corollary}\label{coro:lifting-cousin}
	Let $f: X \to Y$ be a flat morphism. Let $J$ be a residue complex on $Y$. Assume $f^\ast J$ is a dualizing complex on $X$. Denote by $\mu: Y \to \Z$ (resp. $\mu ': X \to \Z$) the codimension function of dualizing complexes $J$ (resp. $f^\ast J$). Then $\mu '(x) \geq \mu(f(x))$ for any $x \in X$ and the equality holds if and only $x$ is a generic point of $\overline{\{f^{-1}(f(x))\}}$.
\end{corollary}

\begin{remark}\label{rmk:res_res}
	Lemma \ref{lma:lifting-cousin} implies that any natural isomorphism $f^*(-) \to Ef^*(-)$ (which exists, cf.\ \cite[Proof of Proposition IV.3.4, p.249]{hartshorne1966residues}) from $\Res(Y)_{Z^\bullet}$ to $D^b_c(\M(X))$ has a unique lifting which is a natural quasi-isomorphism from $\Res(Y)_{Z^\bullet}$ to $\Ch^b_c(\M(X))$.\ We  fix a natural quasi-isomorphism $$\rho: f^*(-) \to Ef^*(-)$$ from $\Res(Y)_{Z^\bullet}$ to $\Ch^b_c(\M(X))$.
\end{remark}

\subsection{Pullback of flat morphisms}
Let $f: X \to Y$ be a flat morphism throughout this subsection.\ Assume that $J$ is a residue complex on $Y$.
\begin{theorem}\label{thm:flat_pullback}
	Suppose that $f^\ast J$ is a dualizing complex on $X$. Then the exact dg functor $$f^\ast : \Ch^b_{c}(\M(Y)) \to \Ch^b_{c}(\M(X))$$ induces a map of Grothendieck-Witt spectra
	\[
		f^\ast: \GW^{[n]}\left(Y, J\right) \to \GW^{[n]}\left(X,Ef^\ast J\right)
	\]
	which respects composition, i.e. if $f$ and $g$ are two consecutive flat morphisms, then $f^\ast g^\ast = (gf)^\ast$.
\end{theorem}

\begin{remark}
	Note that $f^* J$ is a residue complex, if $f$ is residually stable, i.e.\ $f$ is Gorenstein, the fibres of $f$ are discrete and for all $x \in X$, the extension $\kappa(x) / \kappa(f(x))$ is algebraic (e.g. open immersion), cf.\ \cite[Proposition VI.5.3, p.352]{hartshorne1966residues} (and \cite[p.132]{conrad2000grothendieck} for clarification).\ If $f$ is a surjective Gorenstein morphism of finite type, then $f^*J$ is a dualizing complex if and only if $J$ is a dualizing complex, cf.  \cite[\href{https://stacks.math.columbia.edu/tag/0E4N}{Lemma 0E4N}]{stacks-project}.
\end{remark}

\begin{lemma}\label{lma:flat_pullback_quasi_iso}
	The canonical natural transformation
	\[
		\beta: f^\ast \left[-,J\right]_Y \to \left[ f^\ast -, f^\ast J\right]_X.
	\]
	from $\Ch^b_c(\M(Y))$ to $\Ch^b_c(\M(X))$ induces a natural quasi-isomorphism
	\[
		\tilde{\beta}: f^\ast \left[- ,J\right]_Y \xrightarrow{\beta}  \left[ f^\ast -, f^\ast J\right]_X \xrightarrow{\rho} \left[f^*-, Ef^\ast J\right]_X.
	\]
\end{lemma}
\begin{proof}
	Since the complex $Ef^*J$ is degreewise injective and represents $f^*J$,	we have
	$$
		\mathrm{R}\sHom_X(-,f^*J) \cong \sHom_X(-, Ef^*J)
	$$
	in $\mathrm{D}^b_c(\M(X))$. The result follows from \cite[Proposition II.5.8]{hartshorne1966residues}.
\end{proof}

\begin{lemma}\label{lem:pull-can}
	Let $f: X \to Y$ be a morphism of schemes, and let $\mathcal{F}$ and $\mathcal{G}$ be $\SO_Y$-modules.\ Then, the following diagram
	\[
		\xymatrix{f^\ast \mathcal{F} \ar[r]^-{f^\ast \can} \ar[d]_-{\can} & f^\ast \sHom_Y(\sHom_Y(\mathcal{F},\mathcal{G}),\mathcal{G}) \ar[d]_{\beta}\\
		\sHom_X(\sHom_X(f^\ast \mathcal{F}, f^\ast \mathcal{G}), f^\ast \mathcal{G}) \ar[r]^-{\beta^\sharp} & \sHom_X(f^\ast \sHom_Y(\mathcal{F},\mathcal{G}), f^\ast \mathcal{G})}
	\]
	commutes in $\M(X)$.
\end{lemma}
\begin{proof}
	By the adjunction of $f^\ast$ and $f_\ast$, it suffices to prove that the following diagram
	\[
		\xymatrix{\mathcal{F} \ar[r]^-{\can} \ar[d]_-{\can \eta} & \sHom_Y(\sHom_Y(\mathcal{F},\mathcal{G}),\mathcal{G}) \ar[d]_{\beta \eta}\\
		f_\ast \sHom_X(\sHom_X(f^\ast \mathcal{F}, f^\ast \mathcal{G}), f^\ast \mathcal{G}) \ar[r]^-{f_\ast \beta^\sharp} & f_\ast \sHom_X(f^\ast \sHom_Y(\mathcal{F},\mathcal{G}), f^\ast \mathcal{G})}
	\]
	commutes in $\M(Y)$, where the natural transformation $\eta: \id \to f_\ast f^\ast$ is the unit of the adjunction. 	We check the commutativity locally by sections. Let $V$ be an open subset of $Y$.\ We show that the diagram
	\[
		\xymatrix{\mathcal{F}(V) \ar[r]^-{\can} \ar[d]_-{\can \eta} & \Hom_{V}(\sHom_Y(\mathcal{F},\mathcal{G})|_{V},\mathcal{G}|_{V}) \ar[d]_{\beta \eta}\\
		\Hom_{f^{-1}V}(\sHom_X(f^\ast \mathcal{F}, f^\ast \mathcal{G})|_{f^{-1}V}, f^\ast \mathcal{G}|_{f^{-1}V} ) \ar[r]^-{\beta^\sharp} & \Hom_{f^{-1}V}(f^\ast \sHom_Y(\mathcal{F},\mathcal{G})|_{f^{-1}V}, f^\ast \mathcal{G}|_{f^{-1}V})}
	\]
	commutes.\ Let $s\in \mathcal{F}(V)$. We explain the identity $\beta \eta(\can(s)) = \beta^\sharp(\can \eta(s))$.\
	Take an open subset $U$ of $f^{-1}V$. The left-hand side morphism
	$		\beta \eta(\can(s)): f^{*}\sHom_Y(\mathcal{F},\mathcal{G}) (U)  \to f^{*}\mathcal{G} (U) $ sends $$t \otimes u \in f^{*}\sHom_Y(\mathcal{F},\mathcal{G}) (U) :=  f^{-1}\sHom_Y(\mathcal{F},\mathcal{G}) \otimes_{f^{-1}\SO_Y } \SO_X(U) $$
	to the element  $t(s) \otimes u \in f^*\mathcal{G}(U) \cong f^{-1}\mathcal{G}\otimes_{f^{-1}\SO_Y} \SO_X(U)$.\ Strictly speaking, one should express $f^{-1}\sHom_Y(\mathcal{F},\mathcal{G})$ as a colimit and take a representative of $t$ to define the element $t(s)$. For the right-hand side, note first that the canonical morphism $\beta :  f^{*}\sHom_Y(\mathcal{F},\mathcal{G}) (U) \to  \sHom_X(f^*\mathcal{F},f^*\mathcal{G}) (U) $ sends an element $t\otimes u$ to the morphism $\beta(t\otimes u): f^\ast \mathcal{F}|_U \to f^\ast \mathcal{G}|_U$ such that $\beta(t\otimes u)(r \otimes v) := t(r) \otimes uv$. Hence, $\beta^\sharp(\can \eta(s))(t\otimes u)= \can(s\otimes 1)(\beta(t\otimes u)) = \beta(t\otimes u)(s\otimes 1) = t(s) \otimes u$.
\end{proof}

\begin{proposition}\label{prop:beta}
	The pullback functor
	\[
		f^\ast: \Ch^b_c(\M(Y)) \to \Ch^b_c(\M(X))
	\]
	induces a dg form functor
	\[ (f^\ast, \tilde{\beta}) : ( \Ch^b_c(\M(Y)) , \sharp_J) \to (\Ch^b_c(\M(X)), \sharp_{Ef^\ast J}).   \]
\end{proposition}

\begin{proof}
	All we need to do is demonstrate that the diagram of functors
	\[
		\xymatrix{f^\ast \ar[rr]^-{f^\ast \mathrm{can}} \ar[d]_-{\mathrm{can}_{f^\ast}}&& f^\ast \sharp_{J}\sharp_{J} \ar[d]_-{\tilde{\beta}_{\sharp_{J}}} \\
		\sharp_{Ef^\ast J}\sharp_{Ef^\ast J} f^\ast \ar[rr]^-{\sharp_{Ef^\ast J} (\tilde{\beta})} && \sharp_{Ef^\ast J} f^\ast \sharp_{J}}
	\]
	commutes. The proof requires that we establish the commutativity of the following diagram
	$$
		\xymatrix{f^\ast A \ar@{=}[d] \ar[rr]^-{f^\ast \mathrm{can}} \ar@{}[drr]|-{\diagram \label{diag:pull_dual_1}} & & f^\ast [[A,J],J] \ar[d]_-{\beta}\\
		f^\ast A \ar[r]^-{\mathrm{can}}  \ar[d]_-{\mathrm{can}} \ar@{}[dr]|-{\diagram \label{diag:pull_dual_2}} &\left[\left[f^\ast A, f^\ast J\right], f^\ast J\right] \ar[r]^-{\beta^\sharp} \ar[d]_-{\rho} \ar@{}[dr]|-{\diagram \label{diag:pull_dual_3}} & \left[f^\ast [A,J], f^\ast J\right] \ar[d]_-{\rho} \\
		\left[\left[f^\ast A, Ef^\ast J\right], Ef^\ast J\right] \ar[r]^-{\rho^\sharp} &\left[\left[f^\ast A, f^\ast J\right], Ef^\ast J\right] \ar[r]^-{\beta^\sharp} & \left[f^\ast [A,J], Ef^\ast J\right]}
	$$
	for any object $A$ in $\Ch^b_c(\M(Y))$.\ The diagram $\diag{\ref{diag:pull_dual_1}}$ (resp.\ $\diag{\ref{diag:pull_dual_2}}$) is commutative by Lemma \ref{lem:pull-can} (resp.\ Lemma \ref{lem:mor-can}). The diagram $\diag{\ref{diag:pull_dual_3}}$ commutes by naturality.
\end{proof}

\begin{proof}[Proof of Theorem \ref{thm:flat_pullback}]
	By Lemma \ref{lma:flat_pullback_quasi_iso} and Proposition \ref{prop:beta}, we have a nonsingular exact form functor $(f^\ast, \tilde{\beta}):( \Ch^b_c(\M(Y)) , \quis, \sharp_J) \to (\Ch^b_c(\M(X)), \quis, \sharp_{Ef^\ast J})$, which induces the required map
	\[
		f^\ast: \GW^{[n]}\left(Y, J\right) \to \GW^{[n]}\left(X,Ef^\ast J\right).
	\]
	The fact that the pullback morphism respects composition follows from Proposition \ref{prop:composition-pullback} below.
\end{proof}

Let $f:X\to Y$ and $g:Y \to Z$ be flat morphisms. Let $K$ be a residue complex on $Z$. Assume that  $g^\ast K$ and $f^\ast g^\ast K$ are both dualizing complexes. Let $a: f^\ast g^\ast \to (gf)^\ast$ be the canonical natural isomorphism, cf. \cite[$(3.6.1)^*$, p.118]{lipman2009notes}.
The functor $E: \Dual(X) \to \Res(X)$ sends the isomorphism $f^*(\rho)$ to an isomorphism $E(f^*(\rho)): Ef^\ast g^\ast K \to Ef^\ast Eg^\ast K$. Define $b$ to be the composition
\[
	b: E(gf)^*K \xrightarrow{E(a^{-1})} Ef^*g^*K \xrightarrow{E(f^*(\rho))} Ef^* Eg^* K,
\]
which is an isomorphism.
\begin{lemma}\label{lem:rho-compo}
	The diagram
	\[
		\xymatrix{ f^*g^*K \ar[r]^-{a} \ar[d]_-{f^*(\rho)} & (gf)^*K \ar[r]^-{\rho} &  E (gf)^*K \ar[d]_-{b} \\
		f^*Eg^*K \ar[rr]^-{\rho} && Ef^*Eg^*K }
	\]
	commutes in $\Ch^b_c(\M(X))$.
\end{lemma}
\begin{proof}
	By Lemma \ref{lma:lifting-cousin}, we have a canonical isomorphism
	\[
		\Hom_{\Ch^b_c(\M(X))}((fg)^*K, E(fg)^*K) \cong \Hom_{D^b_c(\M(X))}((fg)^*K, E(fg)^*K).
	\]
	Since $f^*g^*K \cong (gf)^*K$ and $Ef^*Eg^*K \cong E (gf)^*K$ in $\Ch^b_c(\M(X))$, we see that the canonical morphism
	\[
		\Hom_{\Ch^b_c(\M(X))}(f^*g^*K, Ef^*Eg^*K) \to \Hom_{D^b_c(\M(X))}(f^*g^*K, Ef^*Eg^*K)
	\]
	is a bijection.\ The diagram in question is commutative (i.e. $b\rho a = \rho f^*(\rho)$) in $\mathrm{D}^b_c(\M(X))$ by naturality, and therefore it is commutative in $\Ch^b_c(\M(X))$.
\end{proof}

\begin{proposition}[Composition of flat pullback]\label{prop:composition-pullback}
	The following diagram
	\[
		\xymatrix{\GW^{[n]}\left(Z,K\right) \ar[r]^-{(gf)^\ast} \ar[d]_-{g^\ast} & \GW^{[n]}\left(X,E(gf)^\ast K\right) \ar[d]_{b} \\
		\GW^{[n]}\left(Y,Eg^\ast K\right) \ar[r]^-{f^\ast} & \GW^{[n]}\left(X,Ef^\ast Eg^\ast K\right)}
	\]
	commutes.
\end{proposition}
\begin{proof}
	It is enough to prove that the following diagram of form functor
	\[
		\xymatrix{(\Ch^b_{c}(\M (Z)), \sharp_{K}) \ar[d]_-{(g^\ast, \tilde{\beta})} \ar[r]^-{((gf)^\ast, \tilde{\beta})} & (\Ch^b_{c}(\M (X)), \sharp_{I}) \ar[d]_{(\id, b)}\\
		(\Ch^b_{c}(\M (Y)), \sharp_{Eg^\ast K}) \ar[r]^-{(f^\ast, \tilde{\beta})}  & (\Ch^b_{c}(\M (X)), \sharp_{I})}
	\]
	commutes up to a natural weak equivalence of form functors. It suffices to prove that the following diagram
	$$
		\xymatrix{f^\ast g^\ast \left[A, K\right] \ar[r]^-{a} \ar[d]_-{\beta} \ar@{}[drr]|-{\diagram \label{diag:pull_comp_1}} & (gf)^\ast \left[A, K\right] \ar[r]^-{\beta} & \left[(gf)^\ast A, (gf)^\ast K\right] \ar[d]^-{a^\sharp} \ar[r] & \left[(gf)^\ast A, E(gf)^\ast K\right] \ar[d]^-{a^\sharp} \\
		f^\ast \left[g^\ast A, g^\ast K\right] \ar[r]^-{\beta} \ar[d] & \left[f^\ast g^\ast A, f^\ast g^\ast K\right] \ar[r]^-{a} \ar[d] \ar@{}[drr]|-{\diagram \label{diag:pull_comp_2}} & \left[f^\ast g^\ast A, (gf)^\ast K\right] \ar[r] & \left[f^\ast g^\ast A, E(gf)^\ast K\right] \ar[d]_{b} \\
		f^\ast \left[g^\ast A, Eg^\ast K \right] \ar[r]^-{\beta} & \left[f^\ast g^\ast A, f^\ast Eg^\ast K\right] \ar[rr] && \left[f^\ast g^\ast A, Ef^\ast Eg^\ast K\right]}
	$$
	commutes for any $A \in \Ch^b_{c}(\M (Z))$, where $\diag{\ref{diag:pull_comp_1}}$ commutes by Lemma \ref{lma:pullback_composition_hom} and $\diag{\ref{diag:pull_comp_2}}$ commutes by Lemma \ref{lem:rho-compo}.
\end{proof}

\begin{lemma}\label{lma:pullback_composition_hom}
	For any $\SO_Z$-modules $\mathcal{F},\mathcal{G}$, the following diagram
	\[
		\xymatrix{f^\ast g^\ast \sHom_Z(\mathcal{F}, \mathcal{G}) \ar[r]^-{a} \ar[d]_-{\beta} & (gf)^\ast \sHom_Z(\mathcal{F}, \mathcal{G}) \ar[r]^-{\beta} & \sHom_X((gf)^\ast \mathcal{F}, (gf)^\ast \mathcal{G}) \ar[d]^-{a^\sharp} \\
		f^\ast \sHom_Y(g^\ast \mathcal{F}, g^\ast \mathcal{G}) \ar[r]^-{\beta} & \sHom_X(f^\ast g^\ast \mathcal{F}, f^\ast g^\ast \mathcal{G}) \ar[r]^-{a} & \sHom_X(f^\ast g^\ast \mathcal{F}, (gf)^\ast \mathcal{G})}
	\]
	commutes.
\end{lemma}
\begin{proof}
	See \cite[p.126 Exercise 3.7.1.1]{lipman2009notes}.
\end{proof}

\section{On the base change formula}
In this section, we prove a base change formula. 	Let
\[
	\xymatrix{{X'} \ar[r]^-{\bar{g}} \ar[d]_-{\bar{f}}& X \ar[d]_-{f}\\
	Y' \ar[r]^-{g}& Y}
\]
be a fibre product. Then, we have a canonical natural transformation
$$\varepsilon : g^* f_* \to \bar{f}_*\bar{g}^*$$
from $\M(X)$ to $\M(Y')$, which is adjoint to the morphism $f_\ast \to f_\ast \bar{g}_\ast \bar{g}^\ast  = g_\ast \bar{f}_\ast  \bar{g}^\ast $, cf. \cite[Section 12.2, p.327]{gortz2020algebraic}.\ If $f$ is proper and $g$ is flat, then the canonical morphism $\varepsilon: g^* f_*(\mathcal{F}) \to \bar{f}_*\bar{g}^*(\mathcal{F})$ is an isomorphism for any quasi-coherent sheaf $\mathcal{F}$, cf. \textit{loc.\ cit.}.
\begin{lemma}\label{lma:gamma}
	Assume that $f$ is proper and $g$ is an open immersion.\ Then there is a natural isomorphism
	\[
		\gamma: \bar{g}^\ast f^\Delta \xrightarrow{\cong} \bar{f}^\Delta g^\ast
	\]
	from $\Res(Y)$ to $\Res(X')$ such that the following diagram
	\[
		\xymatrix{\bar{f}_\ast \bar{g}^\ast f^\Delta \ar[d]_-{\gamma} \ar[r]^-{\varepsilon^{-1}}_-{\cong}& g^\ast f_\ast f^\Delta \ar[d]_-{\mathrm{Tr}} \\ \bar{f}_\ast \bar{f}^\Delta g^\ast \ar[r]^-{\mathrm{Tr}} & g^\ast}
	\]
	commutes.
\end{lemma}
\begin{proof}
	Note that residue complexes are degreewise quasi-coherent, so $\varepsilon$ is an isomorphism. See \cite[Theorem VI.5.3, p.352, 5.5, p.354 and 5.6, p.355]{hartshorne1966residues} and \cite[Lemma 3.4.3, p.149]{conrad2000grothendieck}.
\end{proof}

\begin{theorem}\label{thm:residually_stable_pullback}
	Assume that $f$ is proper and $g$ is an open immersion.\ Then, the diagram
	\[
		\xymatrix{\GW^{[n]}\left(X^\prime, \bar{f}^\Delta g^\ast   I\right) \ar[d]_{\bar{f}_\ast} & \ar[l]_{\gamma}^{\simeq} \GW^{[n]}\left(X^\prime, \bar{g}^\ast f^\Delta  I\right) & \ar[l]_(0.43){\bar{g}^\ast} \GW^{[n]}(X, f^\Delta I)  \ar[d]_{f_\ast} \\
		\GW^{[n]}(Y', g^\ast I) && \ar[ll]_-{g^\ast} \GW^{[n]}(Y, I)}
	\]
	of Grothendieck-Witt spectra commutes up to homotopy.
\end{theorem}
\begin{proof}
	We first show that the functor $\bar{g}^\ast : \Ch^b_{c}(\M(X)) \to  \Ch^b_{c}(\M(X^\prime))$ satisfies $\bar{g}^*(\AC_{f}) \subseteq \AC_{\bar{f}}$.  By \cite[Proposition 3.9.5, p.142]{lipman2009notes} and \cite[Proposition IV 3.1.0, p.290]{SGA6}, we note that $\varepsilon : g^* \mathrm{R} f_* \to \mathrm{R} \bar{f}_*\bar{g}^*$ is a natural isomorphism from $\mathrm{D}_c^b(\M(X))$ to $\mathrm{D}_c^b(\M(Y'))$. Consider the following commutative diagram
	\[\xymatrix{\bar{f}_*\bar{g}^*(A) \ar[r]^-{\varepsilon} \ar[d] & g^* f_* (A)  \ar[d] \\
		\mathrm{R}\bar{f}_*\bar{g}^*(A)  \ar[r]^-{\varepsilon} & g^* \mathrm{R}f_* (A) } \]
	in $\mathrm{D}^b_c(\M(Y'))$ where $A $ is an object in $ \mathrm{Ch}_c^b(\M(X)) $, cf. \cite[Lemma 3.10.1.1, p.146]{lipman2009notes}.\ From this diagram, it is evident that the canonical morphism $\bar{f}_*\bar{g}^*(A) \to  \mathrm{R}\bar{f}_*\bar{g}^*(A) $ is an isomorphism if  the canonical morphism $A \to \mathrm{R}f_*A$ is an isomorphism. Therefore, $\bar{g}^*(A)$ is in $\AC_{\bar{f}}$, if $A$ is in the category $\AC_f$. Moreover, it becomes clear that $\varepsilon : g^* f_* \to \bar{f}_*\bar{g}^*$ is a natural quasi-isomorphism from $\AC_f$ to $\Ch^b_{c}(\M(Y^\prime))$.

	All we need to do is demonstrate that the diagram
	\[
		\xymatrix{\left(\AC_{\bar{f}}, \sharp_{\bar{f}^\Delta g^\ast I} \right) \ar[d]_{(\bar{f}_\ast, \zeta)} & \ar[l]_{(\id, \gamma)} \left(\AC_{\bar{f}}, \sharp_{\bar{g}^\ast f^\Delta  I}\right) & \ar[l]_-{(\bar{g}^\ast, \beta)} \left(\AC_{f}, \sharp_{f^\Delta I}\right)  \ar[d]_{(f_\ast, \zeta)} \\
		\left(\Ch^b_{c}(\M(Y')), \sharp_{g^\ast I}\right) && \ar[ll]_-{(g^\ast, \beta)} \left(\Ch^b_{c}(\M(Y)), \sharp_{I}\right)}
	\]
	commutes up to the natural quasi-isomorphism $\varepsilon$.\ The argument is finished, if we can prove the diagram
	\[
		\xymatrix{g^\ast f_\ast [A, f^\Delta I] \ar[r]^-{\zeta} \ar[d]_-{\varepsilon} & g^\ast [f_\ast A, I] \ar[rr]^-{\beta} && [g^\ast f_\ast A, g^\ast I] \\
		\bar{f}_\ast \bar{g}^\ast[A, f^\Delta I] \ar[r]^-{\beta} & \bar{f}_\ast [\bar{g}^\ast A, \bar{g}^\ast f^\Delta I] \ar[r]^-{\gamma} & \bar{f}_\ast[\bar{g}^\ast A, \bar{f}^\Delta g^\ast I]  \ar[r]^-{\zeta} & [\bar{f}_\ast \bar{g}^\ast A, \bar{g}^\ast I] \ar[u]^-{\varepsilon^{\sharp}}
		}
	\]
	commutes for every $A\in \AC_f$. This follows from the commutative diagram
	$$
		\xymatrix@C=8.7pt{\bar{f}_\ast \bar{g}^\ast[A, f^\Delta I] \ar[d]_-{\beta} \ar@{}[rrrrd]!UL|-{\diagram \label{diag:base_dual_1}} &&&& g^\ast f_\ast [A, f^\Delta I] \ar[d]_-{\phi} \ar[llll]_-{\varepsilon}\\
		\bar{f}_\ast [\bar{g}^\ast A, \bar{g}^\ast f^\Delta I] \ar[d]_-{\gamma} \ar[r]^-{\phi} & [\bar{f}_\ast \bar{g}^\ast A, \bar{f}_\ast \bar{g}^\ast f^\Delta I] \ar[d]_-{\gamma} \ar[r]^-{\varepsilon^{-1}} \ar@{}[dr]|-{\diagram \label{diag:base_dual_2}} & [\bar{f}_\ast \bar{g}^\ast A, g^\ast f_\ast f^\Delta I] \ar[d]_-{\Tr } \ar[r]^-{\varepsilon^{\sharp}} & [g^\ast f_\ast A, g^\ast f_\ast f^\Delta I] \ar[d]_-{\Tr } & g^\ast [f_\ast A, f_\ast f^\Delta I] \ar[d]_-{\Tr } \ar[l]_-{\beta}\\
		\bar{f}_\ast[\bar{g}^\ast A, \bar{f}^\Delta g^\ast I] \ar[r]^-{\phi} & [\bar{f}_\ast \bar{g}^\ast A, \bar{f}_\ast \bar{f}^\Delta g^\ast I] \ar[r]^-{\Tr} & [\bar{f}_\ast \bar{g}^\ast A, g^\ast I] \ar[r]^-{\varepsilon^{\sharp}} & [g^\ast f_\ast A, g^\ast I] & g^\ast [f_\ast A, I] \ar[l]_-{\beta}, }
	$$
	where $\diag{\ref{diag:base_dual_1}}$ commutes by \cite[Lemma 4.6.5]{lipman2009notes}, $\diag{\ref{diag:base_dual_2}}$ commutes by Lemma \ref{lma:gamma}, and other squares commute by naturality.
\end{proof}

\section{On the pullbacks in the regular case}
Let $f: X \to Y$ be a morphism of schemes.\ Suppose that $L$ is a line bundle on $Y$. One can always simply define a pullback morphism on the vector bundle $\GW$-theory as
\[f^*:\GW^{[i]}(Y,L) \to \GW^{[i]}(X,f^*L)\]
by the dg form functor $(f^*,\beta) :(\Ch^b(\V(Y)),\sharp_L) \to (\Ch^b(\V(X)),\sharp_{f^*L}) $ where $\beta: f^*[A,L] \to [f^*A,f^*L]$ is the canonical morphism, cf. \cite[Section 9.4]{schlichting2017hermitian}. In Section \ref{sec:pullback}, we have constructed the pullbacks in the case of flat morphisms.\ Unfortunately, the general pullback can not be simply defined in the coherent Hermitian $K$-theory. The situation is already reflected in the situation of coherent $K$-theory. If $f$ is of finite tor-dimension cf. \cite[II.4, p.99]{hartshorne1966residues},  then Serre's formula comes into the picture
\[f^*: G_0(Y) \to G_0(X), A \mapsto \sum_{i\geq 0} (-1)^i \mathrm{Tor}_{\SO_Y}^i(A,\SO_X).  \]
If we assume that $Y$ is a regular scheme, then the canonical morphism $K_0(Y)\to G_0(Y)$ is an isomorphism. Via this identification, the pullbacks on $K$-theory coincide with that on $G$-theory.

\subsection{Pullbacks in the regular case}
Suppose that $X$ and $Y$ are regular schemes throughout this subsection.\ Our aim in this section is to study the general pullbacks of $\GW$-theory  in the regular case.

\begin{theorem}\label{thm:reg_pullback}
	Suppose that $f: X \to Y$ is a morphism of regular schemes. Assume that $L$ is a line bundle on $Y$, and that $\rho: L\to EL$ is the associated residue complex of $L$.\ Then, the exact dg functor $$f^\ast : \Ch^b(\V(Y)) \to \Ch^b(\V(X))$$ induces a map of Grothendieck-Witt spectra
	\[
		f^\ast: \GW^{[n]}\left(Y, EL\right) \to \GW^{[n]}\left(X,Ef^\ast L\right)
	\]
	which respects composition,\ i.e.\ if $f$ and $g$ are two consecutive morphisms, then $f^\ast g^\ast = (gf)^\ast$.
\end{theorem}
\begin{proof}
	Form the following dg functors
	\[\Ch^b_c(\M(Y)) \leftarrow  \Ch^b(\V(Y)) \xrightarrow{f^*} \Ch^b(\V(X)) \to \Ch^b_c(\M(X)) \]
	which induce the following dg form functors
	\[
		(\Ch^b_c(\M(Y)),\sharp_{EL}) \xleftarrow{(\id, \rho)} (\Ch^b(\V(Y)), \sharp_{L}) \xrightarrow{(f^*, \beta)} (\Ch^b(\V(X)), \sharp_{f^*L}) \xrightarrow{(\id, \rho)} (\Ch^b_c(\M(X)), \sharp_{Ef^*L}).\]
	The required morphism of coherent Grothendieck-Witt spectra
	\[f^*: \GW^{[n]}(Y,EL) \to \GW^{[n]}(X,Ef^*L) \]
	is defined to be the composition
	\[\GW^{[n]}(Y,EL) \xleftarrow[\simeq]{\rho} \GW^{[n]}(Y,L) \xrightarrow{f^*} \GW^{[n]}(X,f^*L) \xrightarrow[\simeq]{\rho} \GW^{[n]}(Y,Ef^*L).  \]
	The final claim that pullback functor $f^*$ respects composition follows from the fact that the vector bundle $\GW$-theory respects the composition, cf. \cite[Section 9.3]{schlichting2017hermitian} or Lemma \ref{lma:pullback_composition_hom}.
\end{proof}

\begin{proposition}\label{prop:flat-compatible}
	Let $f:X\to Y$ be a flat morphism, and $L$ be a line bundle on $X$. Then, the isomorphism $E f^\ast(\rho): E f^\ast L \to E f^\ast EL$ induces an equivalence
	\[
		\GW^{[n]}\left(X, Ef^\ast L\right) \xrightarrow{\simeq} \GW^{[n]}\left(X, Ef^\ast EL\right)
	\]
	of Grothendieck-Witt spectra. Under this equivalence, two pullbacks defined in Theorem \ref{thm:flat_pullback} and Theorem \ref{thm:reg_pullback} coincide.
\end{proposition}

\begin{proof}
	By naturality, there is a commutative diagram
	\[
		\xymatrix{
		f^\ast L \ar[r]^-{\rho} \ar[d]_-{f^\ast(\rho)} & E f^\ast L \ar[d]_-{E f^\ast(\rho)}\\
		f^\ast EL \ar[r]^-{\rho} & E f^\ast E L.
		}\]

	Let $\tilde{\rho}$ be the composition $f^\ast L \xrightarrow{f^{\ast} \rho} f^\ast EL \to Ef^\ast EL$.
	We shall show that the following diagram
	\begin{equation}\label{eqn:regular_flat_pullback}
		\xymatrix{\GW^{[n]}\left(X, f^\ast L\right) \ar[d]_-{\tilde{\rho}} & \GW^{[n]}\left(Y,L\right) \ar[l]_-{f^\ast} \ar[d]_-{\rho}\\
		\GW^{[n]}\left(X, Ef^\ast EL\right) & \GW^{[n]}\left(Y, EL\right) \ar[l]_-{f^\ast}}
	\end{equation}
	commutes.
	The result follows from the fact that the following square
	\[
		\xymatrix{[f^\ast A, f^\ast L] \ar[d]_-{f^\ast (\rho)}& f^\ast [A, L] \ar[l]_-{\beta } \ar[d]_-{\rho}\\
		[f^\ast A, f^\ast EL]  & f^\ast [A, EL] \ar[l]_-{\beta}  }
	\]
	commutes for any $A\in \Ch^b(\V (Y))$.
\end{proof}

\begin{remark}
	The diagram \eqref{eqn:regular_flat_pullback} still commutes if we assume more generally that $X$ is a Gorenstein scheme with an ample family of line bundles. However, $\rho$ and $\tilde{\rho}$ may not be equivalences in this generality.
\end{remark}

\subsection{Projection formula}
Let $\iota: Z \to X$ be a morphism, and let $\mathcal{G}$ be a sheaf of $\SO_X$-modules. There is a canonical morphism
\begin{equation}\label{eqn:projection_morphism}
	\wp: \iota_* \mathcal{F} \otimes \mathcal{G} \to \iota_*(\mathcal{F}\otimes \iota^*\mathcal{G})
\end{equation}
given by the composition
\[ \iota_* \mathcal{F} \otimes \mathcal{G} \to \iota_*\mathcal{F} \otimes \iota_*\iota^*\mathcal{G} \to \iota_*(\mathcal{F}\otimes \iota^*\mathcal{G}), \]
cf. \cite[Example 3.4.6, p.107]{lipman2009notes}. This morphism is an isomorphism if $\mathcal{G}$ is locally free.

Let $I$ be a residue complex and $L$ be a line bundle on $X$. By the local nature of the definition of $\iota^\Delta$, there is a natural isomorphism of complexes (cf. \cite[(3.3.9)]{conrad2000grothendieck}):
\[
	\theta: \iota^\Delta I \otimes \iota^\ast L \to \iota^\Delta(I \otimes L).
\]
Assume that $\iota$ is finite, by \cite[Lemma VI.4.1, p.335]{hartshorne1966residues}, the natural morphism
\[
	\Hom_{\Ch^b_c(\M(X))}(\iota_\ast (\iota^\Delta I \otimes \iota^\ast L), I \otimes L) \xrightarrow{\cong} \Hom_{D^b_c(\M(X))}(\iota_\ast (\iota^\Delta I \otimes \iota^\ast L), I \otimes L).
\]
is an isomorphism and then by \cite[Proposition III.6.9(c), p.174]{hartshorne1966residues} (and \cite[(2.2.10)]{conrad2000grothendieck} for clarification), the diagram
$$
	\xymatrix{
	\iota_\ast (\iota^\Delta I \otimes \iota^\ast L) \ar[r]^-{\wp^{-1}} \ar[d]_-{\theta} \ar@{}[dr]|-{\diagram \label{diag:projection_residue}}& \iota_\ast \iota^\Delta I \otimes L \ar[d]_-{\Tr} \\
	\iota_\ast \iota^\Delta(I \otimes L) \ar[r]^-{\Tr} & I \otimes L.
	}
$$
is commutative.
\begin{proposition}[Projection formula]\label{prop:projection_formula}
	Let $\iota: Z \to X$ be a finite morphism of regular schemes. Let $I$ be a residue complex on $X$ and let $L$ be a line bundle on $X$.
	Then the following diagram of Grothendieck-Witt spectra
	\[
		\xymatrix{\GW^{[n]}(Z, \iota ^\Delta I) \times \GW^{[m]}(X, L) \ar[r]^-{\iota_\ast \times \id} \ar[d]_-{\id \times \iota^\ast} & \GW^{[n]}(X, I) \times \GW^{[m]}(X, L) \ar[d]_-{\cup} \\
		\GW^{[n]}(Z, \iota ^\Delta I) \times \GW^{[m]}(Z, \iota^\ast L) \ar[d]_-{\cup} & \GW^{[n+m]}(X, I\otimes L) \\
		\GW^{[n+m]}(Z, \iota ^\Delta I \otimes \iota^\ast L) \ar[r]^-{\theta}_-{\simeq} & \GW^{[n+m]}(Z, \iota ^\Delta (I \otimes L)) \ar[u]^-{\iota_\ast}}
	\]
	commutes up to homotopy. In a formula, we have that
	\[ \iota_* (\alpha \cup \iota^* \beta) = \iota_*(\alpha) \cup \beta \]
	for any $\alpha \in \GW^{[n]}_i(Z, \iota^\Delta I)$ and $\beta \in \GW^{[m]}_j(X,L)$.
\end{proposition}
\begin{proof}
	Our goal is to prove that the following diagram
	\[\xymatrixcolsep{50pt}
		\xymatrix{
		\left(\Ch^b_c(\M(Z)), \sharp_{\iota^\Delta I}\right)\times \left(\Ch^b(\V(X)), \sharp_{L}\right) \ar[d]_-{\id \times (\iota^\ast, \beta)}\ar[r]^-{(\iota_\ast, \zeta) \times \id}& \left(\Ch^b_c(\M(X)), \sharp_{I}\right)\times \left(\Ch^b(\V(X)), \sharp_{L}\right) \ar[d]_-{\cup}\\
		\left(\Ch^b_c(\M(Z)), \sharp_{\iota^\Delta I}\right)\times \left(\Ch^b(\V(Z)), \sharp_{\iota^\ast L}\right) \ar[d]_-{\cup} & \left(\Ch^b_c(\M(X)), \sharp_{I\otimes L}\right)\\
		\left(\Ch^b_c(\M(Z)), \sharp_{\iota^\Delta I \otimes \iota^\ast L}\right) \ar[r]^-{(\id,\theta)}& \left(\Ch^b_c(\M(Z)), \sharp_{\iota^\Delta (I \otimes L)}\right) \ar[u]^-{(\iota_\ast, \zeta)}
		}
	\]
	of dg categories with duality, commutes up to the natural homotopy $\wp$. In other words, we need to show that the following diagram
	\[\xymatrix{\iota_*[A, \iota^\Delta I] \otimes [B,L] \ar[rr]^-{\tau \zeta} \ar[d]^-{\wp} && [\iota_*A \otimes B, I \otimes L]  \\
		\iota_*([A,\iota^\Delta I] \otimes \iota^*[B,L]) \ar[rr]^-{ \zeta \theta \tau  \beta} && [\iota_*(A \otimes \iota^*B), I \otimes L] \ar[u]_-{\wp^\sharp} } \]
	commutes for $A \in \Ch^b_c(\M(Z))$ and $B \in \Ch^b(\V(X))$, where $\tau$ is defined in \cite[(1.10)]{schlichting2017hermitian}. By the commutativity of the diagram $\diag{\ref{diag:projection_residue}}$, the result follows from the commutativity of the diagram
	\[
		\xymatrix{
		f_* [A,N] \otimes [B,L] \ar[r]^-{\phi} \ar[d]_{\wp} & [f_* A,f_* N] \otimes [B,L] \ar[r]^{\tau} & [f_* A \otimes B,f_* N \otimes L] \\
		f_*([A,N] \otimes f^*[B,L]) \ar[d]_{\beta} & & [f_*(A \otimes f^* B),f_* N \otimes L] \ar[u]_{\wp^\sharp} \\
		f_*([A,N] \otimes [f^* B,f^* L]) \ar[r]^-{\tau} & f_* [A \otimes f^* B,N \otimes f^* L] \ar[r]^-{\phi} & [f_*(A \otimes f^* B),f_*(N \otimes f^* L)] \ar[u]_{\wp^{-1}}
		},
	\]
	where $N \in \Ch^b_c(\M(Z))$. The commutativity of this diagram can be checked directly on sheaves or follows from \cite[Theorem 5.5.1]{calmes2009tensor}.
\end{proof}

\subsection{Projection formula applied to regular immersions} \label{subsec:self_intersection}

Let $\mathcal{E}$ be a vector bundle of rank $d$ over a regular scheme $X$ with a section $s: \mathcal{E}^\vee \to \SO_X$. Assume that the closed embedding $\iota : Z\hookrightarrow X$ of the zero locus $Z:= Z(s)$ is a regular immersion of codimension $d$.\ Recall the Koszul complex $K(\mathcal{E})$ given by the section $s$ is
\[0 \to \wedge^d \mathcal{E}^\vee \to \wedge^{d-1} \mathcal{E}^\vee \to \wedge^{d-1} \mathcal{E}^\vee \to \cdots \to \wedge^{2} \mathcal{E}^\vee \to \mathcal{E}^\vee \xrightarrow{s} \SO_X\to 0, \]
endowed with a symmetric form $\sigma: K(\mathcal{E}) \to \big[K(\mathcal{E}),\ \wedge^d\mathcal{E}^\vee[d]\big]$ on $K(\mathcal{E})$ induced by the canonical pairing
$$\wedge^i \mathcal{E}^\vee \otimes \wedge^{d-i} \mathcal{E}^\vee \to \wedge^d \mathcal{E}^\vee. $$
Note that the Koszul complex $K(\mathcal{E})$ is acyclic off $Z$.\ Therefore, we have a canonical element $\kappa(\mathcal{E}):=(K(\mathcal{E}), \sigma)$ in $\GW^{[d]}_{0}(X~\mathrm{on}~Z, \wedge^d \mathcal{E}^\vee)$.

\begin{remark}
	Note that the normal bundle $N_\iota$ of the embedding is isomorphic to the bundle $ \iota^* \mathcal{E}^\vee$. On the one hand, the exact sequence $0 \to \mathfrak{I} \to \SO_X \to \iota_* \SO_Z \to 0$ provides that $\mathfrak{I}/\mathfrak{I}^2 \cong L^1\iota^* \iota_* \SO_Z$. On the other hand, we have $L^1\iota^* \iota_* \SO_Z \cong H^1(\iota^* K(\mathcal{E})) \cong \iota^* \mathcal{E}^\vee$, since the Koszul complex $K(\mathcal{E})$ provides a finite flat resolution, and the differentials in $\iota^* K(\mathcal{E})$ vanish.
\end{remark}
Consider the residue complex $E(\det \mathcal{E}^\vee)$. By the fundamental local isomorphism (cf. \cite[Proposition III.7.2, p.179]{hartshorne1966residues}), we have that $\iota^\Delta E(\det \mathcal{E}^\vee) [d] \cong \SO_Z$ in the derived category $D^b_c(\M(X))$. Thus, we get a trivial symmetric form
$$1_Z: \SO_Z \to \iota^\Delta E(\det \mathcal{E}^\vee)[d]$$
in $\GW^{[d]}_0 (Z,\iota^\Delta E(\det \mathcal{E}^\vee))$.

\begin{proposition}\label{prop:koszul}
	Let $L$ be a line bundle on $X$.\ The following diagram
	\[
		\xymatrix{
		\GW^{[n]}\left(X, L \right) \ar[d]_-{\iota^*} \ar[r]^-{ \kappa(\mathcal{E})\cup } & \GW^{[n + d]}_{Z}\left(X,  \det \mathcal{E}^\vee\otimes L \right) \ar[r]^-{\rho} & \GW^{[n+d]}_{Z}(X,E(\det \mathcal{E}^\vee)\otimes  L )  \\
		\GW^{[n]}(Z, \iota^*L) \ar[r]^-{1_Z  \cup} & \GW^{[n+d]}(Z, \iota^\Delta E(\det \mathcal{E}^\vee) \otimes \iota^*L  ) \ar[r]^-{\theta}& \GW^{[n+d]}(Z, \iota^\Delta  (E(\det \mathcal{E}^\vee) \otimes L) ) \ar[u]^-{\iota_\ast}
		}
	\]
	commutes up to homotopy. In a formula, we have that
	\[ \iota_* \iota^* (\alpha) = \kappa(\mathcal{E}) \cup \alpha \]
	for any $\alpha \in \GW^{[n]}_i(X,L)$.
\end{proposition}
\begin{proof}
	By the projection formula (cf.\ Proposition \ref{prop:projection_formula}), the following diagram
	\[
		\xymatrix{
		\GW^{[n]}\left(X, L \right) \ar[d]_-{\iota^*} \ar[rr]^-{\iota_*(1_Z) \cup } & & \GW^{[n+d]}_{Z}(X, L \otimes E(\det \mathcal{E}^\vee))\\
		\GW^{[n]}(Z, \iota^*L) \ar[r]^-{1_Z \cup } & \GW^{[n+d]}(Z, \iota^*L \otimes \iota^\Delta E(\det \mathcal{E}^\vee)) \ar[r]^-{\theta}& \GW^{[n+d]}(Z, \iota^\Delta (L \otimes E(\det \mathcal{E}^\vee))) \ar[u]^-{\iota_\ast}
		}
	\]
	commutes. Note that $\iota_*(1_Z) \simeq \rho (\kappa(\mathcal{E}))$ in $ \pi_0 \GW^{[d]}_{Z}(X , E(\det \mathcal{E}^\vee))$, cf. \cite[Proposition 7.1]{calmes2011push}. The result follows.
\end{proof}

\subsection{Some formulas in the regular case}
Let $f:X\to Y$ be a smooth morphism between regular schemes of relative dimension $d$. Define the relative canonical bundle $\omega_{X/Y} := \mathrm{det}\,\Omega_{X/Y}$. Let $L$ be a line bundle on $Y$. Then, the map $\rho: L \to EL$ induces a stable equivalence
\[
	\rho: \GW^{[n]}(Y,L) \rightarrow \GW^{[n]}(Y,EL).
\]
Note that, according to \cite[Theorem VI.3.1(d), p.318]{hartshorne1966residues}, there exists a resolution $\rho'$ given by the composition
\[
	\rho': f^\ast L \otimes \omega_{X/Y} [d] \xrightarrow{\rho} E\left(f^\ast L \otimes \omega_{X/Y} [d]\right) \cong E f^\sharp L  \xrightarrow{\rho} E f^\sharp EL \xrightarrow{\cong} f^\Delta EL.
\]
Recall that $f^\sharp = L f^\ast (-) \otimes \omega_{X/Y} [d]$. Consequently, $\rho'$ induces a stable equivalence
\[
	\rho':	\GW^{[n+d]}(X,f^\ast L \otimes \omega_{X/Y}) \rightarrow \GW^{[n]}(Y,f^\Delta EL).
\]

\begin{definition}
	The pushforward map
	\[
		f_\ast: \GW^{[n+d]}(X,f^\ast L \otimes \omega_{X/Y}) \rightarrow \GW^{[n]}(Y,L)
	\]
	of a smooth morphism $f$ in the regular case is defined as the composite
	\[
		\GW^{[n+d]}(X,f^\ast L \otimes \omega_{X/Y}) \xrightarrow{\rho '} \GW^{[n]}(Y,f^\Delta EL) \xrightarrow{f_\ast}  \GW^{[n]}(Y,EL) \xrightarrow{\rho^{-1}} \GW^{[n]}(Y,L).
	\]
\end{definition}

\begin{proposition}\label{prop:regular_smooth_base_change}
	Let
	\[
		\xymatrix{{X'} \ar[r]^-{\bar{g}} \ar[d]_-{\bar{f}}& X \ar[d]_-{f}\\
		Y' \ar[r]^-{g}& Y}
	\]
	be a fibre product of regular schemes, where $f$ is smooth of relative dimension $d$ and $g$ is an open immersion. Then,
	\[
		g^\ast f_\ast \simeq \bar{f}_\ast \bar{g}^\ast
	\]
	as maps from $\GW^{[n+d]}(X,f^\ast L \otimes \omega_{X/Y})$ to $\GW^{[n]}(Y', g^\ast L)$.
\end{proposition}
\begin{proof}
	By Theorem \ref{thm:residually_stable_pullback} and Proposition \ref{prop:flat-compatible}.
\end{proof}

\begin{proposition}\label{prop:projection_projective_space}
	Let $X$ be a regular scheme, and let $L, L'$ be line bundles on $X$. Let $\mathcal{E}$ be a vector bundle over $X$ of rank $r+1$, and let $p: \mathbb{P}(\mathcal{E}) \to X$ be the associated projective bundle. Then we have
	\[
		p_* (\alpha \cup p^* \beta) = p_*(\alpha) \cup \beta
	\]
	for any $\alpha \in \GW^{[n]}_i(\mathbb{P}(\mathcal{E}), p^\ast L \otimes \SO(-r-1))$ and $\beta \in \GW^{[m]}_j(X,L')$.
\end{proposition}

\begin{proof}
	Recall that the pushforward $p_*: \GW^{[n]}(\mathbb{P}(\mathcal{E}), p^\Delta EL) \to \GW^{[n-r]}(X, EL ) $ is defined to be the composition
	\[ \GW^{[n]}(\mathbb{P}(\mathcal{E}), p^\Delta EL) \xleftarrow{\simeq} \GW^{[n]}(\AC_p, \sharp_{p^\Delta EL}) \xrightarrow{f_\ast} \GW^{[n-r]}(X, EL )\]
	cf.\ Section \ref{sec:pushforward}.\ Take any $A \in \AC_p$.\ By the definition of $\AC_p$ (cf. Definition  \ref{def:acyclic_dg}), the canonical morphism $p_\ast A \to \mathrm{R}p_\ast A$ is an isomorphism in $D^b_c(\M(X))$. Then for any $B \in \Ch^b(\V(X))$, the following diagram
	\[
		\xymatrix{
		p_\ast (A \otimes p^\ast B) \ar[d] \ar[r]^-{\wp} & p_\ast A \otimes B \ar[d]\\
		\mathrm{R} p_\ast (A \otimes p^\ast B) \ar[r]^-{\wp_1} & \mathrm{R}p_\ast A \otimes B
		}
	\]
	commutes by \cite[(3.2.1.2), (3.1.2.3) and Proposition 3.2.4]{lipman2009notes}, where $\wp$ is the isomorphism defined in \eqref{eqn:projection_morphism}, and $\wp_1$ is the quasi-isomorphism given by \cite[Proposition 3.9.4]{lipman2009notes}. The right vertical map is also a quasi-isomorphism, since $B$ is flat. Therefore $A \otimes p^\ast B \in \AC_p$, i.e. the image of the following composition
	\[
		\AC_p \times \Ch^b(\V(X)) \xrightarrow{\id \times p^\ast} \AC_p \times \Ch^b(\V(Z)) \xrightarrow{\cup} \Ch^b_c(\M(X))
	\]
	of dg functor lies in $\AC_p$. Our goal turns to show that the following diagram
	\[\xymatrixcolsep{50pt}
		\xymatrix{
		\left(\AC_p, \sharp_{p^\Delta EL}\right)\times \left(\Ch^b(\V(X)), \sharp_{L'}\right) \ar[dd]_-{\id \cup (p^\ast, \beta)}\ar[r]^-{(p_\ast, \zeta) \times \id}& \left(\Ch^b_c(\M(X)), \sharp_{EL}\right)\times \left(\Ch^b(\V(X)), \sharp_{L'}\right) \ar[d]_-{\cup}\\
		& \left(\Ch^b_c(\M(X)), \sharp_{EL\otimes L'}\right)\\
		\left(\AC_p, \sharp_{p^\Delta EL \otimes p^\ast L'}\right) \ar[r]^-{(\id,\theta)}& \left(\AC_p, \sharp_{p^\Delta (EL \otimes L')}\right) \ar[u]^-{(p_\ast, \zeta)}
		}
	\]
	of dg categories with duality, commutes up to the natural homotopy $\wp: p_* A \otimes B \to p_*(A\otimes p^* B)$ for any $A\in \AC_p, B \in \Ch^b(\V(X))$.
	By the proof of Proposition \ref{prop:projection_formula}, the result follows from Lemma \ref{lma:projective_bundle_projection_residue} below.
\end{proof}
\begin{lemma}\label{lma:projective_bundle_projection_residue}
	Let $\mathcal{E}$ be a vector bundle over $X$ of rank $r+1$, and let $p: \mathbb{P}(\mathcal{E}) \to X$ be the associated projective bundle.
	Then, the following diagram
	$$
		\xymatrix{
		p_\ast (p^\Delta I \otimes p^\ast L) \ar[r]^-{\wp^{-1}} \ar[d]_-{\theta} & p_\ast p^\Delta I \otimes L \ar[d]_-{\Tr} \\
		p_\ast p^\Delta(I \otimes L) \ar[r]^-{\Tr} & I \otimes L.
		}
	$$
	commutes in $Z^0\Ch^b_c(\M(X))$ for any dualizing complex $I$ on $X$ and any line bundle $L$ on $\mathbb{P}(\mathcal{E})$.
\end{lemma}

\begin{proof}
	By \cite[Proposition III.4.4, p.158]{hartshorne1966residues} (and \cite[Theorem 2.3.2]{conrad2000grothendieck} for clarification), the diagram commutes in derived categories $D^b_c(\M(X))$. Since $I\otimes L$ is degreewise injective and bounded, if there is no homotopy between complexes $p_\ast (p^\Delta I \otimes p^\ast L)$ and $I \otimes L$, then
	\[
		\Hom_{Z^0\Ch^b_c(\M(X))}(p_\ast (p^\Delta I \otimes p^\ast L), I \otimes L) \cong \Hom_{D^b_c(\M(X))}(p_\ast (p^\Delta I \otimes p^\ast L), I \otimes L),
	\]
	and the result follows, cf. \cite[Corollary 10.4.7]{weibel1994an}.
	Under isomorphism $\wp$, it suffices to show that $\Hom_{\M(X)}((p_\ast p^\Delta I )^p, I^{p-1}) = 0$ for any $p \in \Z$. Take any morphism $t\in \Hom_{\M(X)}((p_\ast p^\Delta I )^p, I^{p-1})$. Denote by $\mu: X \to \Z$ and $\mu': \mathbb{P}(\mathcal{E}) \to \Z$ the codimensions associated with $I$ and $p^\Delta I$ respectively. Suppose that $t \neq 0$. Then there exists $r \in \Gamma(\mathbb{P}(\mathcal{E}), (p^\Delta I)^{p})$ such that $0 \neq t(p_\ast(r)) \in I^{p-1}$. The support of $r$ is a finite union $\overline{\{x_1\}} \cup \ldots \cup \overline{\{x_n\}}$ with $\mu ' (x_i) = p$. Then the support of $t(p_\ast(r))$ is contained in $\overline{\{p(x_1)\}} \cup \ldots \cup \overline{\{p(x_n)\}}$, and $\mu(p(x_i)) \geq p$ by \cite[Proposition VI.3.4, p.333]{hartshorne1966residues} (and \cite[(3.1.25)]{conrad2000grothendieck} for clarification), which is a contradiction.
\end{proof}

Let $\iota: Z \to X$ be the morphism with the same conditions as presented in Subsection \ref{subsec:self_intersection}. Set $\omega_{Z/X} = \mathrm{det}\, \iota^\ast \mathcal{E} \cong \iota^\ast \mathrm{det}\, \mathcal{E}$.\ Let $L$ be a line bundle on $X$, then there is a stable equivalence
\[
	\GW^{[n]}(X,L) \xrightarrow{\rho} \GW^{[n]}(X,EL).
\]
By the fundamental local isomorphism (cf. \cite[Proposition III.7.2, p.179]{hartshorne1966residues}), we have that $\iota^\ast L \otimes \omega_{Z/X} [-d] \cong \iota^\Delta EL$ in the derived category $D^b_c(\M(Z))$.
Thus there is a resolution $\rho ': \iota^\ast L \otimes \omega_{Z/X} [-d] \to \iota^\Delta EL$, hence a stable equivalence
\[
	\rho ':	\GW^{[n-d]}(Z,\iota^\ast L \otimes \omega_{Z/X}) \xrightarrow{} \GW^{[n]}(Z,\iota^\Delta EL).
\]
Then we define pushforward map
\[
	\iota_\ast:	\GW^{[n-d]}(Z,\iota^\ast L \otimes \omega_{Z/X}) \xrightarrow{ } \GW^{[n]}(X,L).
\]
of regular immersion $\iota$ in the regular case as the composite
\[
	\GW^{[n-d]}(Z,\iota^\ast L \otimes \omega_{Z/X}) \xrightarrow{\rho '} \GW^{[n]}(Z,\iota^\Delta EL) \xrightarrow{\iota _\ast}  \GW^{[n]}(X,EL) \xrightarrow{\rho^{-1}} \GW^{[n]}(X,L).
\]

\begin{lemma}\label{lma:koszul_regular}
	Let $\iota:Z\to X$ be the regular immersion as above, and let $L$ be a line bundle on $X$, then the following diagram
	\[
		\xymatrix{
		\GW^{[n]}(X,L) \ar[d]_-{\iota^\ast} \ar[r]^-{\kappa (\mathcal{E}) \cup} & \GW^{[n+d]}(X,\mathrm{det}\, \mathcal{E}^\vee \otimes L)\\
		\GW^{[n]}(Z,\iota^\ast L) \ar[r]_-{\simeq}& \GW^{[n]}(Z,\iota^\ast \left(\mathrm{det}\, \mathcal{E}^\vee \otimes L\right) \otimes \omega_{Z/X}) \ar[u]^-{\iota_\ast}
		}
	\]
	commutes.
\end{lemma}
\begin{proof}
	By Proposition \ref{prop:koszul}.
\end{proof}

\section{Geometric description of the connecting homomorphism}
\label{sec:connecting}
In this section, we incorporate the assumption presented by Balmer and Calm\`es \cite{balmer2009geometric} together with Appendix \ref{sect:connecting-and-cone} to investigate the connecting homomorphism in Hermitian $K$-theory.
\subsection{Hypothesis}
\label{hypo:BC} Consider a regular immersion $\iota: Z \hookrightarrow X$ of codimension $c \geq 2$. Let $\pi: B \to X$ be the blow-up of $X$ along $Z$, and $E$ be the exceptional fibre. Define $U := X - Z \cong B - E$ as the open complement. We assume that $v: U \hookrightarrow X$ and $\tilde{v}: U \hookrightarrow B$ are corresponding embeddings.\ Suppose that there exists a scheme $Y$ and an auxiliary flat morphism $\tilde{\alpha}: B \to Y$ such that $\alpha := \tilde{\alpha} \tilde{v}$ forms an affine bundle, i.e., a Zariski locally trivial bundle. This setup can be illustrated in the following diagram:
\[
	\xymatrix{
	Z \ar@{^{(}->}[r]^-{\iota} & X  & \ar@{_{(}->}[l]_-{v}  \ar@{_{(}->}[ld]_-{\tilde{v}} U \ar[d]^-{\alpha}  \\
	E \ar[u]^-{\widetilde{\pi}} \ar@{^{(}->}[r]_-{\tilde{\iota}} & B  \ar[u]^-{\pi} \ar@{-->}[r]_-{\tilde{\alpha}} & Y.
	}
\]

\subsection{The main theorem}
Let $L$ be a line bundle on a regular scheme $X$. Under the isomorphism
$$
	\Pic(B) \cong \Pic(X)\oplus \Z,
$$
we have
$$
	\tilde{\alpha}^* (\alpha^*)^{-1} v^*(L) \cong \pi^*(L) \otimes \SO(E)^{\otimes \lambda(L)}
$$
for some integer $\lambda(L)$ (cf. \cite[Remark 2.2]{balmer2009geometric}). Let $L_{Y} := (\alpha^*)^{-1} v^*(L)$. The aim of this section is to prove the following result.

\begin{theorem}\label{theo:connecting-hom} Suppose that  $Z, X$ are both regular schemes and that Hypothesis \ref{hypo:BC} holds.
	\begin{enumerate}[leftmargin=20pt,label={\rm (\Alph*)}]
		\item If $\lambda(L) \equiv c-1 \mod 2$, then we have
		      \[ \GW^{[n]}(X,L)  \simeq \GW^{[n]}(Y,L_{Y}) \oplus \GW^{[n-c]}(Z, \omega_\iota \otimes \iota^\ast L).  \]
		\item If $\lambda(L) \equiv c \mod 2$, then the following diagram
		      \begin{equation}\label{eqn:main_b}
			      \xymatrix{
			      \GW^{[n]}(U,v^\ast L) \ar[d]^-{\simeq}_-{(\alpha^*)^{-1}} \ar[r]^-{\partial} \ar[r] & \GW^{[n-c]}(Z, \omega_\iota \otimes \iota^\ast L)[1]  & \ar[l]_-{\eta \cup} \GW^{[n-c +1]}(Z, \omega_\iota \otimes \iota^\ast L )  \\
			      \GW^{[n]}(Y,L_{Y})  \ar[r]^-{ \tilde{\iota}^* \tilde{\alpha}^*}  & \GW^{[n]}(E, \tilde\iota^\ast \tilde\alpha^\ast L_Y) \ar[r]^-{\simeq} & \GW^{[n]}(E,  \omega_{\widetilde{\pi}} \otimes \widetilde{\pi}^*(\omega_\iota \otimes \iota^\ast L) ) \ar[u]^-{\widetilde{\pi}_*}
			      }
		      \end{equation}
		      is commutative up to homotopy. Note that $\eta$ is the Bott element defined in \cite[(6.1)]{schlichting2017hermitian}.
	\end{enumerate}
\end{theorem}

\begin{proof}
	(A)	In view of the localization sequence
	$$	\GW^{[n-c]}(Z, \omega_{\iota} \otimes \iota^\ast L) \xrightarrow{\iota_*} \GW^{[n]}(X, L) \xrightarrow{v^*} \GW^{[n]}(U, v^\ast L) \xrightarrow{\partial} \GW^{[n]}(Z, \omega_{\iota} \otimes \iota^\ast L)[1], $$
	it is enough to show that $v^*$ has a right splitting. To find this splitting, we form a diagram as follows
	$$
		\xymatrix{
		\GW^{[n]}(X, L) \ar[r]^-{v^\ast} \ar@{}[dr]|-{\diagram \label{diag:conn_split_1}} &  \GW^{[n]}(U, v^* L)   & \ar[l]_-{b}^-{\simeq} \GW^{[n]}(U, \alpha^*(\alpha^*)^{-1}v^*L)  \\
		\GW^{[n]}(B, \omega_\pi \otimes \pi^*L ) \ar[r]^-{\tilde{v}^*} \ar[u]^-{\pi_*} & \GW(U, \tilde{v}^*(\omega_\pi \otimes \pi^*L)) \ar[u]^-{\gamma}_-{\simeq}    \\
		\GW^{[n]}(B, {\tilde{\alpha}}^* (\alpha^*)^{-1} v^*L )  \ar[u]_-{\simeq}^-{\epsilon} \ar[r]^-{\tilde{v}^*} \ar@{}[drr]|-{\diagram \label{diag:conn_split_2}} & \GW^{[n]}(U, \tilde{v}^*{\tilde{\alpha}}^* (\alpha^*)^{-1} v^*L) \ar[r]^-{a}_-{\simeq} \ar[u]^-{\tilde{v}^*(\epsilon)}_-{\simeq} &  \GW^{[n]}(U, \alpha^*(\alpha^*)^{-1}v^*L) \ar@{=}[uu]\\
		\GW^{[n]}(B, {\tilde{\alpha}}^* (\alpha^*)^{-1} v^*L ) \ar@{=}[u]& & \ar[ll]_-{\tilde{\alpha}^*} \GW^{[n]}(Y,(\alpha^*)^{-1}v^*L) \ar[u]^-{\alpha^*}_-{\simeq}. }
	$$
	The squares $\diag{\ref{diag:conn_split_1}}$ and $\diag{\ref{diag:conn_split_2}}$ are commutative by the base change formula (cf. Proposition \ref{prop:regular_smooth_base_change}) and the functoriality of pullback respectively, and the isomorphisms $a$ and $\gamma$ are fixed by these properties. The isomorphism $\epsilon$ can be chosen to be the following composition
	$$ \xymatrix{ \GW^{[n]}(B, \SO(E)^{\otimes \lambda(L)} \otimes \pi^*L ) \ar[r]^-{\psi}_-\simeq &  \GW^{[n]}(B, \SO(E)^{\otimes (c-1)} \otimes \pi^*L ) \ar[d]_-{(\id,\rho_1)}^-{\simeq} \\
		\GW^{[n]}(B, {\tilde{\alpha}}^* (\alpha^*)^{-1} v^*L ) \ar[u]^-{(\id,\rho_0)}_-{\simeq} \ar[r]^-{\epsilon} &\GW^{[n]}(B, \omega_\pi \otimes \pi^*L )  } $$
	where $\psi$ is the two-periodicity on the twists, and $\rho_0$ (resp. $\rho_1$) is the isomorphism $  {\tilde{\alpha}}^* (\alpha^*)^{-1} v^*L  \cong \SO(E)^{\otimes \lambda(L)} \otimes \pi^*L$ (resp. $ \SO(E)^{\otimes (c-1)} \otimes \pi^*L  \cong \omega_\pi \otimes \pi^*L$), cf. \cite[Proposition A.11]{balmer2009geometric}.
	Then, we construct $b$ to be the composition $ \gamma \tilde{v}^*(\epsilon) a^{-1}$. It is more difficult to prove (B), and we postpone this case to Section \ref{subsec:proof-conn-hom} below.  \end{proof}

\begin{remark}
	In Theorem \ref{theo:connecting-hom} (B), compared to the work of Balmer-Calm\`es \cite{balmer2009geometric}, the additional data of cupping the Bott element become necessary to provide a geometric description of the connecting homomorphism in Hermitian $K$-theory.
\end{remark}

\subsection{On the codimension one case}

Let $B$ be a scheme with a line bundle $L$ and $\tilde\iota : E \hookrightarrow B$ be a prime divisor, and $U := B - E$ be its open complement. Let $\tilde v: U \to B$ be the open immersion. Let $\SO(E)$ be the line bundle on $B$ associated to $E$. There is a canonical morphism of line bundles
\[
	\sigma_E: \SO(E)^\vee \to \SO_{B}.
\]
Note that the morphism $\sigma_E$ induces an isomorphism $t:\tilde v^\ast \SO(E) \xrightarrow{\cong} \tilde v^\ast \SO_B$.\ By Theorem \ref{coro:loc_seq}, there is a distinguished triangle
\[
	\GW^{[n]}_E\left(B, L \right) \xrightarrow{} \GW^{[n]}\left(B, L \right) \xrightarrow{\tilde v^\ast}  \GW^{[n]}\left(U, \tilde v^\ast L\right) \xrightarrow{\partial} \GW^{[n]}_E\left(B, L \right)[1].
\]
in $\mathcal{SH}$.

\begin{lemma}\label{lem:codim-one}
	The following diagram  of Grothendieck-Witt spectra
	\[
		\xymatrix{{\GW^{[n]}(B, L \otimes \SO(E))} \ar[r]^-{\tilde\iota^\ast} \ar[d]_-{\tilde v^\ast} & {\GW^{[n]}\left(E,\tilde{\iota}^\ast L \otimes \omega_{\tilde{\iota}}\right)} \ar[r]^-{\tilde\iota_\ast} & {\GW^{[n+1]}_E(B, L)} \ar[d]_-{\eta \cup} \\
		{\GW^{[n]}\left(U,\tilde v^\ast \left(L\otimes \SO(E)\right)\right)} \ar[r]_-{\simeq}^-{t} & {\GW^{[n]}\left(U,\tilde v^\ast L\right)} \ar[r]^-{\partial} & {\GW^{[n]}_E(B, L)}[1]}
	\]
	commutes up to homotopy.
\end{lemma}\begin{proof}
	Let $\AC$ be the dg category $\Ch^b_{c}(\V(B))$. We define the set $w \subset Z^0\AC$ such that $\phi \in w$ if and only if $\tilde v^\ast(\phi) \in \mathrm{quis}$. Note that $\mathrm{quis} \subset w$ since the pullback of a quasi-isomorphism on $B$ via $\tilde{v}$ remains a quasi-isomorphism on $U$. Additionally, the morphism $\sigma_E$ belongs to $w$.\ Consider the full subcategory $\mathrm{Fun}_{w}([1], \AC)$ of $\mathrm{Fun}([1], \AC)$. It consists of functors $[1]\to \AC$ whose images are in $w$ (see \cite[p.367]{schlichting2010mayer}). Inside $\mathrm{Fun}_w([1], \AC)$, we have an object $\sigma_E$ equipped with a symmetric form:
	\[ \xymatrix{
		\SO(E)^\vee \ar[r]^-{\sigma_E} \ar[d]^-{1} & \SO_{B} \ar[d]^-{1} \\
		\SO(E)^\vee \ar[r]^-{\sigma_E} & \SO_{B}
		}\]
	with respect to the duality $\sharp_{\SO(E)^\vee}: \mathrm{Fun}_w([1], \AC)^\op \to \mathrm{Fun}_w([1], \AC)$.\ Now, let us examine the following diagram:
	$$
		\xymatrix{
		\GW^{[n]}\left(E,\tilde\iota^\ast \left(L \otimes \SO(E)\right)\right)  \ar[rr]^-{\simeq} \ar@{}[drr]|-{\diagram \label{diag:conn_codi_1}} & & {\GW^{[n]}\left(E,\tilde{\iota}^\ast L \otimes \omega_{\tilde{\iota}}\right)} \ar[d]_-{\tilde\iota_\ast}\\
		{\GW^{[n]}(B, L \otimes \SO(E))} \ar[u]^-{\tilde\iota^\ast}  \ar[rr]^-{ \cone(\sigma_E) \cup } \ar@{}[drr]!<-30pt,10pt>|-{\diagram \label{diag:conn_codi_2}} \ar[dr]^-{\sigma_E \cup} \ar[dd]_-{\tilde v^\ast} &&  {\GW^{[n+1]}_E(B, L)}  \ar[dd]_-{ \eta \cup} \\
		\ar@{}[dr]|-{\diagram \label{diag:conn_codi_3}} & \ar[ur]^-{\mathrm{cone}} \GW^{[n]}(\mathrm{Fun}_{w}([1], \AC), \sharp_{L})  \ar[d]_-{P} \ar@{}[dr]|-{\diagram \label{diag:conn_codi_4}} &   \\
		{\GW^{[n]}\left(U,\left(L \otimes \SO(E)\right)|_{U}\right)} \ar[r]^-{\simeq}_-t & {\GW^{[n]}\left(\AC, w ,\sharp_{L} \right)} \ar[r]^-{\partial} & {\GW^{[n]}_E(B, L)[1]}
		}
	$$
	where the diagram $\diag{\ref{diag:conn_codi_1}}$ is commutative up to homotopy, as stated in Lemma \ref{lma:koszul_regular}. The commutativity of the diagram $\diag{\ref{diag:conn_codi_2}}$ follows from \cite[Theorem 5.2]{balmer2005products}. The diagram $\diag{\ref{diag:conn_codi_3}}$ is commutative since the trivialization $t$ is induced by $\sigma_E$ via $\tilde{v}^*$. Lastly, the diagram $\diag{\ref{diag:conn_codi_4}}$ is commutative by Lemma \ref{lem:balmer-schlichting-coin}.
\end{proof}

\subsection{Proof of Theorem \ref{theo:connecting-hom} (B)}\label{subsec:proof-conn-hom}

Consider the following  diagram:
$$
	\xymatrix@C=13pt{
	\GW^{[n]}(U, v^* L) \ar@{}[dr]|-{\diagram \label{diag:conn_conn_1}} \ar[r]^-{\partial} & \GW_Z^{[n]}(X,L)[1] \ar@{}[dr]|-{\diagram \label{diag:conn_conn_2}} & \ar[l]_-{\iota_*} \GW^{[n-d]}(Z,\iota^\ast L \otimes \omega_{\iota})[1] \\
	\GW^{[n]}(U, \tilde{v}^* \left(\pi^\ast L \otimes \omega_{\pi}\right)) \ar[u]_-{\simeq} \ar[r]^-{\partial} \ar@{}[drr]|-{\diagram \label{diag:conn_conn_3}} & \GW^{[n]}_E(B, \pi^\ast L \otimes \omega_{\pi})[1] \ar[u]_{\pi_*} & \ar[l]_-{\tilde{\iota}_*} \GW^{[n-1]}(E, \tilde{\iota}^\ast \left(\pi^\ast L \otimes \omega_{\pi}\right) \otimes \omega_{\tilde{\iota}})[1] \ar[u]_-{\widetilde{\pi}_*} \\
	\GW^{[n]}(B, \SO(E)\otimes \pi^\ast L \otimes \omega_{\pi}) \ar[u]^-{\tilde{v}^*} \ar[rr]^-{\tilde{\iota}^*} && \GW^{[n]}(E, \tilde{\iota}^\ast \left(\pi^\ast L \otimes \omega_{\pi}\right) \otimes \omega_{\tilde{\iota}}) \ar[u]_{\eta \cup}
	}
$$
The diagram $\diag{\ref{diag:conn_conn_2}}$ commutes due to the composition of pushforwards. The commutativity of the diagram $\diag{\ref{diag:conn_conn_3}}$ follows from Lemma \ref{lem:codim-one} and the projection formula (cf. Proposition \ref{prop:projection_formula}) applied to the Bott elements. Lastly, the diagram $\diag{\ref{diag:conn_conn_1}}$ commutes by Lemma \ref{lma:conn_map_of_seq} below. Note that $\widetilde{\pi}_\ast$ commutes with cupping Bott element by Proposition \ref{prop:projection_projective_space}. Then \eqref{eqn:main_b} commutes since $\tilde{\alpha} \tilde{v} = \alpha$ and
\[
	\tilde{\alpha}^\ast L_{Y} = \pi^*(L) \otimes \SO(E)^{\otimes \lambda(L)} = \SO(E)^{\otimes (\lambda(L) - c - 1)}\otimes \pi^\ast L \otimes \omega_{\pi} = \SO(E)\otimes \pi^\ast L \otimes \omega_{\pi}
\]
in $\Pic(B) / 2$.

\begin{lemma}\label{lma:conn_map_of_seq}
	There is a map of homotopy fibration sequences
	\[
		\xymatrix{
		\GW^{[n]}_{E}(B, \pi^\ast L \otimes \omega_{\pi}) \ar[r]^-{ } \ar[d]_-{\pi_\ast }& \GW^{[n]}(B, \pi^\ast L \otimes \omega_{\pi}) \ar[d]_-{\pi_\ast} \ar[r]^-{\tilde{v}^*} & \GW^{[n]}(U, \tilde{v}^* \left(\pi^\ast L \otimes \omega_{\pi}\right)) \ar[d]^-{\simeq}\\
		\GW_{Z}^{[n]}(X, L) \ar[r]^-{ }& \GW^{[n]}(X, L) \ar[r]^-{v^*} & \GW^{[n]}(U, v^*L)
		}
	\]
\end{lemma}
\begin{proof}
	Denote by $\AC = \AC_{\pi}$ the full dg subcategory of $\Ch^b_c(\M(B))$ as defined in Definition \ref{def:acyclic_dg}. Denote by $\mathcal{B} = \Ch^b_c(\M(X))$. Let $w \subset Z^0 \mathcal{B}$ (resp. $\tilde{w} \subset Z^0 \AC$) be the set of morphisms which are quasi-isomorphism after pullback along $v$ (resp. $\tilde{v}$). Note that there is a cartesian square
	\[
		\xymatrix{
		B \ar[d]_-{\pi} & \ar[l]_-{\tilde{v}} U \ar@{=}[d]\\
		X & \ar[l]_-{v} U.
		}
	\]
	Then by base change theorem, $\pi_\ast (\tilde{w}) \subset w$.
	Thus $(\pi_\ast, \zeta)$ is a non-singular exact form functor between squares of dg categories with weak equivalences and duality
	\[
		\xymatrix{
			(\AC^{\tilde{w}}, \quis) \ar[r] \ar[d] & (\AC, \quis) \ar[d] & (\mathcal{B}^w, \quis) \ar[r] \ar[d] & (\mathcal{B}, \quis) \ar[d] \\
			(\AC^{\tilde{w}},\tilde{w}) \ar[r] & (\AC, \tilde{w}) & (\mathcal{B}^w,w) \ar[r] & (\mathcal{B}, w)
		}
	\]
	Then $(\pi_\ast, \zeta)$ induces a map of homotopy fibre sequences of Grothendieck-Witt spectra.
\end{proof}

\section{Excess intersection formula}

\subsection{Euler class}
Let $L$ be a line bundle on a regular scheme $X$. Let $p:V \to X$ be a vector bundle of rank $d$ and let $s: X \to V$ be the zero section. The \textit{Euler class map} of $V$ is the composition of maps of spectra
\[ \GW^{[n]}(X, {L}) \xrightarrow{s_*}  \GW^{[n+d]} ({V}, p^*(\det {V}^\vee \otimes {L})) \xrightarrow{(p^*)^{-1}}  \GW^{[n+d]} (X, \det {V}^\vee \otimes {L})  \]
which shall be denoted by $c({V})_L$.\ From Projection Formula (Proposition \ref{prop:projection_formula}), we see that
\begin{equation}\label{eqn:zero_section_pushforward}
	c(V)_L(\alpha) =   (p^*)^{-1}s_*(1\cup s^*p^*(\alpha)) = (p^*)^{-1} (s_*(1) \cup p^*(\alpha)) = [(p^*)^{-1} s_*(1)] \cup \alpha,
\end{equation}
for $\alpha \in \GW^{[n]}_i(X,L)$, where the last equality follows from the fact that the pullback is compatible with the cup product (See the line before  \cite[Remark 9.4]{schlichting2017hermitian}). Thus, the key information of the morphism $c(V)_L$ is determined by the image of $1 \in \GW_0(X)$.

\begin{definition}
	The Euler class $e(V)$ of $V$ is defined to be the image of the element $1_X \in \GW_0(X)$ under the composition of maps
	\[ \GW_0(X) \xrightarrow{s_*}  \GW^{[d]}_0 ({V}, p^*\det {V}^\vee) \xrightarrow{(p^*)^{-1}}  \GW^{[d]}_0 (X, \det {V}^\vee)  \]
	That is we define $e(V):=(p^*)^{-1}s_*(1_X)$.
\end{definition}

\begin{proposition}\label{prop:euler}
	Let $X,Y$ be regular schemes and let $p: V \to X$ be a vector bundle of rank $d$.
	\begin{enumerate}[leftmargin=20pt]
		\item If $f:Y \to X$ is a morphism of schemes, then $e(f^*V)= f^* e(V)$.
		\item If $0 \to V' \to V \to V'' \to 0$ is an exact sequence of vector bundles over $X$, then $e(V)= e(V'')\cup e(V')$.
	\end{enumerate}
\end{proposition}
\begin{proof}
	See \cite[Proposition 2.4]{fasel2009the}.
\end{proof}

\begin{remark}
	By \cite[Proposition 2.5]{fasel2009the}, we have $\kappa(V)= p^*e(V)$. Since $ps = \id$, $s^\ast \kappa(V) = e(V)$. Thus our definition of Euler class coincide with the definition in \cite[Section 2.5]{fasel2009the}.
\end{remark}

\subsection{Base change}

Consider the following fibre product
\[
	\xymatrix{X \ar[r]^-{\iota} & Y  \\ V \ar[r]^-{\tilde \iota} \ar[u]^-{\tilde \pi} & Z  \ar[u]^-{\pi}}
\]
of regular schemes, where $\iota$ is a regular immersion of codimension $c$, and $\pi: Z \to Y$ is the zero locus of a section $s: \mathcal{E}^\vee \to \SO_Y$, where $\mathcal{E}$ is a vector bundle on $Y$ of rank $d$. Suppose that $Z$ meets $X$ properly. Let $N_{Y} X$ denote the normal cone to $X$ in $Y$.

\begin{proposition}\label{prop:regular_base_change}
	Let $L$ be a line bundle on $Y$. Then, for any $\alpha \in \GW^{[n]}_i(X, \det N_{Y} X \otimes \iota ^\ast L )$, we have
	\[
		\tilde{\iota}_\ast \widetilde{\pi}^\ast (\alpha) = \pi^\ast \iota_\ast (\alpha)
	\]
	in $\GW^{[n+c]}_i(Z, \pi^\ast L)$.
\end{proposition}

\begin{proof}
	Note that $\widetilde{\pi}: V \to X$ is the zero locus of the section $\iota^\ast (s): \iota^\ast \mathcal{E} \to \iota^\ast \SO_Y \cong \SO_X$, and $\kappa(\iota^\ast \mathcal{E}) = \iota^\ast \kappa(\mathcal{E})$ in $\GW^{[d]}_0(X, \det \iota^\ast \mathcal{E}^\vee)$.	Since $\pi_\ast: \GW^{[n+c]}(Z, f^\ast L) \to \GW^{[n+c+d]}_{Z}(Y, L \otimes \det \mathcal{E})$ is an equivalence by D\'evissage Theorem \ref{thm:devissage}, it suffices to prove
	\[
		\pi_\ast \pi^\ast \iota_\ast (\alpha) = \pi_\ast \tilde{\iota}_\ast \widetilde{\pi}^\ast (\alpha).
	\]
	On the left-hand side, by Lemma \ref{lma:koszul_regular}, we have $\pi_\ast \pi^\ast \iota_\ast (\alpha)  = \kappa(\mathcal{E}) \cup \iota_\ast (\alpha)$. On the right-hand side, $\pi_\ast \tilde{\iota}_\ast \widetilde{\pi}^\ast (\alpha) = \iota_\ast \widetilde{\pi}_\ast \widetilde{\pi}^\ast (\alpha) = \iota_\ast (\kappa(\iota^\ast \mathcal{E}) \cup \alpha)$. Then by the projection formula (Proposition \ref{prop:projection_formula}), we have $\iota_\ast (\kappa(\iota^\ast \mathcal{E}) \cup \alpha) = \iota_\ast (\iota^\ast \kappa(\mathcal{E}) \cup \alpha) = \kappa(\mathcal{E}) \cup \iota_\ast  \alpha$. The result follows.
\end{proof}

\subsection{Deformation to the normal cone}
Consider a regular immersion $f: X \to Y$ of regular schemes of codimension $c$. Denote by $N_{Y} X$ the normal cone to $X$ in $Y$ and by $D(X,Y)$ the deformation to the normal cone space, as discussed in \cite[Section 3]{nenashev2007gysin}. Define $s: X \to N_X Y$ as the embedding corresponding to the zero section. There exists a closed immersion $\mu: X\times \mathbb{A}^1 \to D(X,Y)$ and a commutative diagram
\[
	\xymatrix{X \ar[r]^-{\tilde{\delta}_0}  \ar[d]_-{s} & X\times \mathbb{A}^1 \ar[d]_-{\mu} & \ar[l]_-{\tilde{\delta}_1} X \ar[d]_-{f}\\
	N_Y X \ar[r]^-{\delta_0} & D(X,Y) & \ar[l]_-{\delta_1} Y,}
\]
whose square are fibre products.

\begin{lemma}
	The following maps
	\[
		\GW^{[n]}_X(N_Y X, \delta_0^\ast L) \xleftarrow{\delta_0^\ast } \GW^{[n]}_{X \times \mathbb{A}^1}(D(X,Y), L) \xrightarrow{\delta_1^\ast} \GW^{[n]}_{X}(Y, \delta_1^\ast L)
	\]
	are stable equivalences for any line bundle $L$ over $D(X,Y)$.
\end{lemma}
\begin{proof}
	The result follows from the Karoubi induction, cf. \cite[Lemma 6.4]{schlichting2017hermitian}. The $K$-theory case is proved in \cite[Section 2.2.7]{panin2003oriented}, and the Witt theory case has been proved by \cite{nenashev2007gysin}.
\end{proof}
The composition $\delta_1^\ast (\delta_0^\ast)^{-1}$ is denoted by $d(X,Y)$.

\begin{proposition}\label{prop:gysin_pushforward}
	For any $\alpha \in \GW^{[n]}_i(X, \det N_{Y} X \otimes f^\ast L )$, we have
	\[
		f_\ast (\alpha) = d(X,Y) s_\ast (\alpha),
	\]
	in $\pi_i \GW^{[n+c]}_{X}(Y, L)$.
\end{proposition}
\begin{proof}
	This result can be derived from the proof of \cite[Proposition 2.9]{fasel2009the} and by invoking Proposition \ref{prop:regular_base_change}.
\end{proof}

\subsection{Excess intersection formula}
Consider the following fibre product of regular schemes:
\[
	\xymatrix{X \ar[r]^-{\iota} & Y  \\ V \ar[r]^-{\tilde \iota} \ar[u]^-{\tilde \pi} & Z  \ar[u]^-{\pi}.}
\]
Suppose that $\iota$ (resp. $\tilde{\iota}$) is a regular immersion of codimension $c$ (resp. $c'$), and that $\pi$ (resp. $\pi'$) is an arbitrary morphism. Let $N_{Y} X$ (resp. $N_{Z} V$) denote the normal cone to $X$ in $Y$ (resp. $V$ in $Z$), and define $E = \widetilde{\pi}^\ast N_{Y} X / N_{Z} V$.

\begin{theorem}\label{thm:excess_intersection}
	Given a line bundle $L$ on $Y$, for any $\alpha \in \GW^{[n]}_i(X, \det N_{Y} X \otimes \iota ^\ast L )$, we have:
	\[
		\pi^\ast \iota_\ast (\alpha) = \tilde{\iota}_\ast \left(e(E) \cup \widetilde{\pi}^\ast (\alpha)\right)
	\]
	in $\GW^{[n+c]}_i(Z, \pi^\ast L)$.
\end{theorem}
\begin{proof}
	Let $D(X,Y)$ (resp. $D(V,Z)$) represent the deformation to the normal cone space associated to $\iota:X\to Y$ (resp. $\tilde{\iota}:V \to Z$). This gives us a commutative diagram:
	\[
		\xymatrix{X \ar[r]^-{s} & N_X Y \ar[r]^-{} & D(X,Y) & \ar[l]_-{} Y \\
		V \ar[r]^-{\tilde{s}} \ar[u]^-{\widetilde{\pi}}& N_Z V \ar[r]^-{} \ar[u]^-{\widetilde{\pi}_N} & D(V,Z) \ar[u]^-{\widetilde{\pi}_D} & \ar[l]_-{} Z \ar[u]^-{\pi}.}
	\]
	Here, $s: X \to N_X Y$ and $\tilde{s}: V \to N_Z V$ represent the embeddings as zero sections. Let $p: N_X Y \to X$ and $\tilde{p}: N_Z V \to V$ be the projections, leading to another commutative diagram:
	\[
		\xymatrix{X & N_X Y \ar[l]_-{p} \\
		V  \ar[u]^-{\widetilde{\pi}}& N_Z V \ar[u]^-{\widetilde{\pi}_N} \ar[l]_-{\tilde{p}}.
		}
	\]
	We can then refer to \eqref{eqn:zero_section_pushforward} to obtain:
	\[
		\widetilde{\pi}_N ^\ast s_\ast (\alpha) = \widetilde{\pi}_N ^\ast p^\ast (e( N_X Y) \cup \alpha) =  \tilde{p}^\ast \widetilde{\pi} ^\ast(e( N_X Y) \cup \alpha) = \tilde{p}^\ast (\widetilde{\pi} ^\ast e( N_X Y) \cup \widetilde{\pi} ^\ast \alpha) .
	\]
	According to Proposition \ref{prop:euler}, we find:
	\[
		\tilde{p}^\ast (\widetilde{\pi} ^\ast e( N_X Y) \cup \widetilde{\pi} ^\ast \alpha)  = \tilde{p}^\ast (e(\widetilde{\pi} ^\ast  N_X Y) \cup \widetilde{\pi} ^\ast \alpha) = \tilde{p}^\ast (e(E) \cup e(N_Z V) \cup \widetilde{\pi} ^\ast \alpha).
	\]
	Also, by invoking \eqref{eqn:zero_section_pushforward} once more, we have:
	\[
		\tilde{p}^\ast (e(E) \cup e(N_Z V) \cup \widetilde{\pi} ^\ast \alpha) = \tilde{p}^\ast e(E) \cup \tilde{p}^\ast (e(N_Z V) \cup \widetilde{\pi} ^\ast \alpha)  = \tilde{p}^\ast e(E) \cup \tilde{s}_\ast (\widetilde{\pi} ^\ast \alpha).
	\]
	Then, by applying the projection formula (cf.\ Proposition \ref{prop:projection_formula}), we arrive at:
	\[
		\tilde{p}^\ast e(E) \cup \tilde{s}_\ast (\widetilde{\pi} ^\ast \alpha) =  \tilde{s}_\ast ( \tilde{s}^\ast \tilde{p}^\ast e(E) \cup (\widetilde{\pi} ^\ast \alpha)) = \tilde{s}_\ast ( e(E) \cup (\widetilde{\pi} ^\ast \alpha)).
	\]
	Finally, using Proposition \ref{prop:gysin_pushforward}, we conclude that:
	\[
		\begin{aligned}
			\pi^\ast \iota_\ast (\alpha) & = \pi^\ast d(X,Y) s_\ast (\alpha) = d(V,Z) \widetilde{\pi}_N ^\ast s_\ast (\alpha)                                                             \\
			                             & = d(V,Z) \tilde{s}_\ast ( e(E) \cup (\widetilde{\pi} ^\ast \alpha)) = \tilde{\iota}_\ast \left(e(E) \cup \widetilde{\pi}^\ast (\alpha)\right).
		\end{aligned}
	\]
	The result follows.
\end{proof}

\section{Projective bundles}

Let $ \mathcal{E} \rightarrow S$ be a vector bundle of rank $r+1$ over a regular scheme $S$.
Let $q :  \mathbb{P}(\mathcal{E}) \rightarrow S$ be the projective bundle associated to the vector bundle $\mathcal{E}$.  The sequence of vector bundles on $\mathbb{P}(\mathcal{E})$
\[\xymatrix{ 0\ar[r] &\SO_{\mathcal{E}}(-1) \ar[r] & q^* \mathcal{E} \ar[r] & \mathcal{Q} \ar[r] & 0 }\]
is exact where $\mathcal{Q}$ is called the canonical quotient bundle.\ Following the ideas of Fasel \cite[Section 5]{fasel2013the}, we make the following construction.
\begin{definition}
	For any $m \geq 1$. Define the twisting map
	\[
		\mu^{m}:
		\K(S) \to \GW^{[n]}(\mathbb{P}(\mathcal{E}), \SO(-m))
	\]
	as the following composition
	\[
		\K(S) \xrightarrow{H q^\ast} \GW^{[n - m + 1]}(\mathbb{P}(\mathcal{E}), \SO(-1)) \xrightarrow{c(\SO(1))^{m-1}} \GW^{[n]}(\mathbb{P}(\mathcal{E}), \SO(-m)).
	\]
	Then we define
	\[
		\Theta_{\textnormal{even}}: \bigoplus_{2 \leq m \leq r }^{m \textnormal{ even}}
		\K_i(S) \xrightarrow{\sum t_m\mu^m} \GW^{[n]}_i(\mathbb{P}(\mathcal{E}))
		\textnormal{\quad \quad}
		\Theta_{\textnormal{odd}}: \bigoplus_{1 \leq m \leq r }^{m \textnormal{ odd}}
		\K_i(S) \xrightarrow{\sum t_m\mu^m} \GW^{[n]}_i(\mathbb{P}(\mathcal{E}), \SO(1))\]
	where $t_m$ are the two-periodicity on the twists.
\end{definition}
By \cite[Proposition A.9(i)]{balmer2009geometric}, we have that the canonical bundle $\omega_{\mathbb{P}(\mathcal{E})/S} $ is isomorphic to $ q^\ast \mathrm{det}\, \mathcal{E} \otimes \SO_{\mathbb{P}(\mathcal{E})}(-r-1)$.\
The aim of this section is to prove the following results.
\begin{theorem}\label{thm:projective_bundle_thm}
	Let $S$ be a regular scheme.
	\begin{enumerate}[label={\rm (\Alph*)}, leftmargin=30pt]
		\item Suppose that $r$ is even. The map
		      \[   (\Theta_{\textnormal{even}}, q^*): \bigoplus_{2 \leq m \leq r }^{m \textnormal{ even}} \K_i(S) \oplus \GW^{[n]}_i(S ) \xrightarrow{\cong} \GW^{[n]}_i(\mathbb{P}(\mathcal{E})) \]
		      is an isomorphism of abelian groups in the even twisted case.
		\item Suppose that $r$ is even.
		      The map
		      \[(\Theta_{\textnormal{odd}}, c(\mathcal{Q}) q^*) \colon \bigoplus_{1 \leq m \leq r }^{m \textnormal{ odd}} \K_i(S) \oplus \GW^{[n-r]}_i(S,\det \mathcal{E}) \xrightarrow{\cong} \GW^{[n]}_i(\mathbb{P}(\mathcal{E}),\SO(1))
		      \]
		      is an isomorphism of groups in the odd twisted case.
		\item Suppose that $r$ is odd. The map
		      \[
			      \Theta_{\textnormal{odd}} \colon \bigoplus_{1 \leq m \leq r }^{m \textnormal{ odd}}
			      \K_i(S)\xrightarrow{\cong} \GW^{[n]}_i(\mathbb{P}(\mathcal{E}), \SO(1))
		      \]
		      is an isomorphism of abelian groups in the odd twisted case.
		\item Suppose that $r$ is odd. The following sequence
		      \begin{equation}\label{eqn:projective_space_d}
			      \xymatrix@C=30pt{
			      \ldots  \ar[r] & \bigoplus\limits_{2 \leq m \leq r }^{m \textnormal{ even}} \K_{i}(S)\oplus \GW^{[n]}_{i}(S) \ar[r]^-{(\Theta_{\mathrm{even}}, q^*)} & \GW^{[n]}_{i}(\mathbb{P}(\mathcal{E})) \ar[r]^-{q_\ast} &    \GW^{[n-r]}_{i}(S, \det \mathcal{E}) \ar[d]^-{\everymath={\scriptstyle}\begin{pmatrix} 0 \\[3pt] \eta\cup   c(\mathcal{E})  \end{pmatrix}} \\
			      \ldots & \ar[l] \GW^{[n-r]}_{i-1}(S, \det \mathcal{E}) &  \ar[l]_-{q_*} \GW^{[n]}_{i-1}(\mathbb{P}(\mathcal{E})) &  \ar[l]_-{(\Theta, q^*)}  \bigoplus\limits_{2 \leq m \leq r }^{m \textnormal{ even}}
			      \K_{i-1}(S)\oplus \GW^{[n]}_{i-1}(S)
			      }
		      \end{equation}
		      of abelian groups is exact in the even twisted case.
	\end{enumerate}
\end{theorem}

We will prove Theorem \ref{thm:projective_bundle_thm} when the bundle $\mathcal{E}$ is trivial, and then apply the Mayer-Vietoris theorem to conclude the global case.

\subsection{The trivial bundle case}
Define $\mathbb{P}^r:= \mathbb{P}(\SO_S^{r+1})$.\ Let $\iota: \mathbb{P}^{r-1} \to \mathbb{P}^{r}$ be a hyperplane, and let $v: U \to \mathbb{P}^{r}$ be its open complement. Inside  $\mathbb{P}^{r-1}$ we can choose a hyperplane $\mathbb{P}^{r-2}$, and we have a composition of closed embeddings $\mathbb{P}^{r-2} \to \mathbb{P}^{r-1} \to \mathbb{P}^{r}$, which we still denote by $\iota$ for simplicity if no confusion occurs.

\begin{theorem}\label{thm:projective_space}
	Let $S$ be a regular scheme.
	\begin{enumerate}[label={\rm (\alph*)}, leftmargin=30pt]
		\item The map
		      \[
			      (\iota_* , q^*)  :\GW^{[n-1]}(\mathbb{P}^{r-1},\SO(1)) \oplus \GW^{[n]}(S) \to \GW^{[n]}(\mathbb{P}^r)
		      \]
		      is a stable equivalence of spectra in $\mathcal{SH}$.

		\item The map
		      \[
			      (H q^\ast, \iota_*  )  : \K(S) \oplus \GW^{[n-2]}(\mathbb{P}^{r-2},\SO(1)) \to \GW^{[n]}(\mathbb{P}^r, \SO(-1))
		      \]
		      is a stable equivalence of spectra in $\mathcal{SH}$.
	\end{enumerate}
\end{theorem}
\begin{proof}
	(a) The localization sequence (cf.\ Theorem \ref{coro:loc_seq}) yields a distinguished triangle
	\[ \GW^{[n-1]}(\mathbb{P}^{r-1}, \SO(1))\xrightarrow{\iota_*} \GW^{[n]}(\mathbb{P}^{r}) \xrightarrow{v^*} \GW^{[n]}(U) \xrightarrow{\partial} \GW^{[n-1]}(\mathbb{P}^{r-1}, \SO(1))[1]\]
	in $\mathcal{SH}$ where $\omega_{\iota} = \SO(1)$.\ Note that $U \cong \mathbb{A}^{r}$ and we can find a left splitting of $v^*$ by the map $q^*(\alpha^*)^{-1}$ as follows:
	$$
		\xymatrix@R=8ex{ \GW^{[n]}(\mathbb{P}^{r}) \ar[r]^-{v^*} & \GW^{[n]}(U)  \\
		& \GW^{[n]}(S) \ar[u]^-{\alpha^\ast}_-{\simeq} \ar[ul]^-{q^\ast}_{\diagram \label{diag:proj_split}}},
	$$
	where $\alpha$ is the projection map, $\alpha^*$ is an equivalence by the homotopy invariance, and the diagram $\diag{\ref{diag:proj_split}}$ commutes by the composition of pullbacks. It follows that the map $(\iota_*, q^*)$ induces a equivalence of spectra.

	(b)
	The localization theorem provides another distinguished triangle
	\[
		\xymatrix{\GW^{[n-1]}(\mathbb{P}^{r-1}) \ar[r]^-{\iota_\ast} & \GW^{[n]}(\mathbb{P}^{r}, \SO(-1)) \ar[r]^-{v^*} & \GW^{[n]}(U) \ar[r]^-{\partial} & \GW^{[n-1]}(\mathbb{P}^{r-1})[1]}
	\]
	in $\mathcal{SH}$, where we choose a trivialization $v^\ast \SO(-1) \cong \SO_{U}$. It is enough to show that the diagrams
	\begin{equation}\label{eqn:projective_space_local_seq_1}
		\xymatrix@C=22pt{
		\GW^{[n-1]}(\mathbb{P}^{r-2},\SO(1)) \ar[d]^-{\iota_*} \ar[r]^{\id} & \GW^{[n-1]}(\mathbb{P}^{r-2},\SO(1)) \ar[d]^-{\iota_*}\ar[r] & 0 \ar[d]\ar[r]                  & \ar[d]^-{\iota_*[1]} \GW^{[n-1]}(\mathbb{P}^{r-2},\SO(1))[1]
		\\ 
		\GW^{[n-1]}(\mathbb{P}^{r-1}) \ar[r]^-{\iota_\ast}                  & \GW^{[n]}(\mathbb{P}^{r}, \SO(-1)) \ar[r]^-{v^*}             & \GW^{[n]}(U) \ar[r]^-{\partial} & \GW^{[n-1]}(\mathbb{P}^{r-1})[1]
		}
	\end{equation}
	and
	$$
		\xymatrix{
		\GW^{[n-1]}(S) \ar[d]^-{q^*} \ar[r]^{F} \ar@{}[dr]|-{\diagram \label{diag:projective_space_1}} &   K(S) \ar[d]^-{Hq^*}\ar[r]^-{H} \ar@{}[dr]|-{\diagram \label{diag:projective_space_2}} &
		\ar@{}[dr]|-{\diagram \label{diag:projective_space_3} }
		\GW^{[n]}(S) \ar[d]^-{\alpha^*}_-{\simeq}\ar[r]^-{\eta \cup } & \ar[d]^-{q^*[1]} \GW^{[n-1]}(S)[1]
		\\ 
		\GW^{[n-1]}(\mathbb{P}^{r-1}) \ar[r]^-{\iota_\ast} & \GW^{[n]}(\mathbb{P}^{r}, \SO(-1)) \ar[r]^-{v^*} & \GW^{[n]}(U) \ar[r]^-{\partial} & \GW^{[n-1]}(\mathbb{P}^{r-1})[1]
		}
	$$
	are maps of distinguished triangles, where the distinguished triangle
	$$
		\xymatrix{
		\GW^{[n-1]}(S) \ar[r]^{F} & K(S) \ar[r]^-{H} & \GW^{[n]}(S) \ar[r]^-{\eta \cup } & \GW^{[n-1]}(S)[1]
		}
	$$
	is the algebraic Bott sequence, cf. \cite[Theorem 6.1]{schlichting2017hermitian}.
	Taking the sum of both maps of distinguished triangles we see that the map $\big(Hq^\ast, \iota_*\big)$ is a stable equivalence, since the map $\alpha^*$ (resp.\ $(\iota_*, q^*)$) is a stable equivalence by the homotopy invariance (resp. (a)). All the squares in the diagram \eqref{eqn:projective_space_local_seq_1} are commutative.\ The commutativity of the diagram $\diag{\ref{diag:projective_space_1}}$ can be deduced from the commutativity of the following diagram
	$$
		\xymatrix{
		\GW^{[n-1]}(\mathbb{P}^{r-1}) \ar@{=}[d] \ar@{}[dr]|-{\diagram \label{diag:projective_space_pullback}}& \GW^{[n-1]}(S) \ar[l]_-{q^\ast} \ar[d]_-{q^\ast} \ar[r]^-{F} & K(S) \ar[d]_-{q^\ast}
		\\
		\GW^{[n-1]}(\mathbb{P}^{r-1}) \ar@{=}[d] \ar@{}[dr]|-{\diagram \label{diag:projective_space_projection}}  & \GW^{[n-1]}(\mathbb{P}^{r}) \ar[r]^-{F} \ar[d]_-{\iota_\ast (1) \cup} \ar[l]_-{1 \cup \iota^\ast (-)} \ar@{}[dr]|-{\diagram \label{diag:FH_cup}}& K(\mathbb{P}^{r}) \ar[d]^-{H}
		\\
		\GW^{[n-1]}(\mathbb{P}^{r-1}) \ar[r]^-{\iota_\ast} &\GW^{[n]}(\mathbb{P}^{r}, \SO(-1)) \ar@{=}[r]& \GW^{[n]}(\mathbb{P}^{r}, \SO(-1)).}$$
	The diagram $\diag{\ref{diag:projective_space_pullback}}$ commutes by composition of pullback. The diagram $\diag{\ref{diag:projective_space_projection}}$ commutes by projection formula (cf. Lemma \ref{lma:koszul_regular}). The diagram $\diag{\ref{diag:FH_cup}}$ commutes by the fact that $\iota_\ast (1) \simeq \SO(-1) [1] \oplus \SO = H(1) \in \GW^{[1]}_0(\mathbb{P}^{r}, \SO(-1))$ and Lemma \ref{lem:FH_cup}.\ The commutativity of the diagram $\diag{\ref{diag:projective_space_2}}$ follows from $v^*H = H v^* $, $\alpha^* H = H \alpha^*$ and $v^* q^* = \alpha^*$.\ For the commutativity of the diagram $\diag{\ref{diag:projective_space_3}}$, we consider the following commutative diagrams
	\[
		\xymatrix{
		\mathbb{P}^{r-1} \ar[rd]_{q} \ar[r]^-{\iota} & \mathbb{P}^{r} \ar[d]_-{q} & U \ar[l]_-{v} \ar[dl]^-{p}\\
		&   S
		}
	\]
	of schemes.\ The results follow from the diagram below:
	$$
		\xymatrix{
		& \GW^{[n-1]}(\mathbb{P}^{r-1})[1]\\
		\GW^{[n]}(\mathbb{P}^{r-1}) \ar[ur]^-{\eta \cup} &\GW^{[n]}(\mathbb{P}^{r}) \ar[l]_-{\iota^\ast} \ar[r]^-{v^\ast} \ar@{}[u]|-{\diagram \label{diag:proj_space_partial}}& \GW^{[n]}(U) \ar[ul]_-{\partial} \\
		& \GW^{[n]}(S) \ar[u]^-{q^\ast} \ar[ur]_-{\alpha^\ast} \ar[ul]^-{q^\ast} \\
		}
	$$
	where the diagram $\diag{\ref{diag:proj_space_partial}}$ commutes by Lemma \ref{lem:codim-one}.
\end{proof}

\begin{lemma}\label{lma:twist_closed_immersion}
	The following diagram
	\[
		\xymatrix{
		K(S) \ar[r]^-{\mu^m} \ar[d]_-{q^\ast}& \GW^{[n]}(\mathbb{P}^{r}, \SO(-m)) \\
		K(\mathbb{P}^{r-m+1}) \ar[r]^-{H} & \GW^{[n - m + 1]}(\mathbb{P}^{r - m + 1}, \SO(-1)) \ar[u]^-{\iota_\ast}
		}
	\]
	commutes for any $1 \leq  m \leq r$. Hence, the formula $\iota_* \mu^{m-1} = \mu^m$ holds.
\end{lemma}
\begin{proof}
	To show the result, we expand the diagram
	$$
		\xymatrix{
		K(S) \ar[d]_-{q^\ast} \ar[r]^-{q^\ast} & K(\mathbb{P}^{r}) \ar[r]^-{H} & \GW^{[n - m + 1]}(\mathbb{P}^{r}, \SO(-1)) \ar[r]^-{c(\SO(1))^{m-1}} \ar[d]_-{\iota^\ast} & \GW^{[n]}(\mathbb{P}^{r}, \SO(-m)) \ar@{=}[d]
		\\
		K(\mathbb{P}^{r-m+1}) \ar[rr]^-{H} && \GW^{[n - m + 1]}(\mathbb{P}^{r-m+1}, \SO(-1)) \ar@{}[ur]|-{\diagram \label{diag:proj_bundle_twist_1}} \ar[r]^-{\iota_*} & \GW^{[n]}(\mathbb{P}^{r}, \SO(-m))  .
		}
	$$
	We show the commutativity of the diagram $\diag{\ref{diag:proj_bundle_twist_1}}$ by induction on $m$. When $m=2$, the diagram $\diag{\ref{diag:proj_bundle_twist_1}}$ commutes by Lemma \ref{lma:koszul_regular}.
	The induction step follows from the following commutative diagram
	$$
		\xymatrix@C=37pt@R=30pt{
		\GW^{[n - m + 1]}(\mathbb{P}^{r}, \SO(-1)) \ar[r]^-{c(\SO(1))^{m-2}} \ar[d]_-{\iota^\ast} \ar@{}[dr]|-{\diagram \label{diag:pullback_euler}}& \GW^{[n - 1]}(\mathbb{P}^{r}, \SO(-m+1)) \ar[r]^-{c(\SO(1))} \ar[d]_-{\iota^\ast} & \GW^{[n]}(\mathbb{P}^{r}, \SO(-m)) \\
		\GW^{[n - m + 1]}(\mathbb{P}^{r-1}, \SO(-1)) \ar[r]^-{c(\SO(1))^{m-2}} \ar[d]_-{\iota^\ast} & \GW^{[n - 1]}(\mathbb{P}^{r-1}, \SO(-m + 1)) \ar[ru]_-{\iota_\ast} &  \\
		\GW^{[n - m + 1]}(\mathbb{P}^{r-m+1}, \SO(-1)) \ar[ru]_-{\iota_\ast} &  &
		}
	$$
	where the diagram $\diag{\ref{diag:pullback_euler}}$ commutes by Proposition \ref{prop:euler} and pullbacks commute with products.
\end{proof}

\begin{proposition}\label{thm:proj-bundle-E-trivial}
	Suppose that the bundle $\mathcal{E}$ is trivial.
	\begin{enumerate}[label={\rm (\alph*)}, leftmargin=30pt]
		\item Theorem \ref{thm:projective_bundle_thm} \textnormal{(A)} is true.
		\item If $r$ is even, the map
		      \[
			      (\Theta, \iota_\ast)  : \bigoplus_{1 \leq m \leq r }^{m \textnormal{ odd}}  \K(S) \oplus \GW^{[n-r]}(S) \to \GW^{[n]}(\mathbb{P}^{r},\SO(1))
		      \]
		      is an equivalence of spectra.
		\item Theorem \ref{thm:projective_bundle_thm} \textnormal{(C)} is true.
		\item If $r$ is odd, the map
		      \[
			      (\Theta, q^*, \iota_\ast)  : \bigoplus_{1 \leq m \leq r }^{m \textnormal{ even}}  \K(S) \oplus \GW^{[n]}(S) \oplus \GW^{[n-r]}(S) \to \GW^{[n]}(\mathbb{P}^{r})
		      \]
		      is an equivalence of spectra.
	\end{enumerate}
\end{proposition}

\begin{proof}
	(c) We prove this formula by induction.\ By Theorem \ref{thm:projective_space} (b), we have that the map $Hq^*: K(S) \to \GW^{[n]}(\mathbb{P}^1,\SO(1))$ is an equivalence. Note that $\mu^1 = Hq^*$ by definition. The induction step follows from the formula  $\iota_*\mu^{m-1} = \mu^{m}$ (cf. Lemma \ref{lma:twist_closed_immersion}) and the composition of pushforward.

	(a) We apply Theorem \ref{thm:projective_space} (a). Then, we use the formula $\iota_*\mu^{m-1} = \mu^{m}$ again together with (c) proved above.

	(b) It follows directly from Theorem \ref{thm:projective_space} (b) that the map
	$$ (H q^\ast, \iota_*  )  : \K(S) \oplus \GW^{[n-2]}(S) \to \GW^{[n]}(\mathbb{P}^2, \SO(1))$$
	is an equivalence of spectra. We deduce the general case by inductively applying Theorem \ref{thm:projective_space} (b), the formula $\iota_* \mu^{m-2} = \mu^m $ and the composition of pushforward.

	(d) If $r$ is odd and the twist is even, this follows from (b), the formula $\iota_*\mu^{m-1} = \mu^{m}$ and Theorem \ref{thm:projective_space} (a). \end{proof}
\begin{proposition}\label{prop:projective_space_even_sequence}
	Suppose that $r$ is even. Then, the sequence
	\[ 0 \to \bigoplus_{1 \leq m \leq r }^{m \textnormal{ odd}} K_i(S) \xrightarrow{\Theta_{\mathrm{odd}}} \GW_i^{[n]}(\mathbb{P}^r,\SO(1))  \xrightarrow{q_*} \GW^{[n-r]}_i(S) \to 0 \]
	of abelian groups is split exact.
\end{proposition}
\begin{proof}
	We claim that the composition $q_*\mu^1$ vanishes, that is the composition
	\[
		\K(S) \xrightarrow{H q^\ast} \GW^{[n]}(\mathbb{P}^{r},\SO(-1)) \xrightarrow{t} \GW^{[n]}(\mathbb{P}^{r},\SO(-r-1)) \xrightarrow{q_\ast} \GW^{[n-r]}(S)
	\]
	is homotopic to zero, where $t$ is the two-periodicity on the twists.

	Set $r' = r/2$. Consider the following diagram
	$$
		\xymatrix{
		K(\mathbb{P}^{r},\SO(-1)) \ar[r]^-{ \SO(-r') \cup } \ar[d]^-{H} \ar@{}[dr]|-{\diagram \label{diag:projective_space_H_t}} & K(\mathbb{P}^{r},\SO(-r-1)) \ar[r]^-{q_\ast} \ar[d]^-{H} \ar@{}[dr]|-{\diagram \label{diag:projective_space_H_q}} & K(S) \ar[d]^-{H}\\
		\GW^{[n]}(\mathbb{P}^{r},\SO(-1)) \ar[r]^-{t} & \GW^{[n]}(\mathbb{P}^{r},\SO(-r-1)) \ar[r]^-{q_\ast} & \GW^{[n-r]}(S)
		}
	$$
	where the diagram $\diag{\ref{diag:projective_space_H_t}}$ commutes by Lemma \ref{lem:FH_cup} and the diagram $\diag{\ref{diag:projective_space_H_q}}$ commutes by Lemma \ref{lma:push_FH}. Then to show that $q_*\mu^1$ vanishes, it suffices to prove that the following composition
	\[
		\K(S) \xrightarrow{q^\ast} K(\mathbb{P}^{r}) \xrightarrow{\SO(-r')\cup} K(\mathbb{P}^{r}) \xrightarrow{q_\ast} K(S)
	\]
	is homotopic to zero, which is true by the projection formula (cf. \cite[Proposition 3.17]{thomason1990higher}) and the fact that $\mathrm{R}q_\ast (\SO(-r')) \cong 0$ (cf. \cite[Chapter III Proposition 2.1.15]{ega31}).\

	For all $m$ such that $1 \leq m \leq r$, we note that the composition $q_\ast  \mu^m$ is also trivial, as  $q_\ast \mu^m = q_* \iota_* \mu^1 = q_* \mu^1 = 0$.    Since $q \iota = \id$,  we see  that $\iota_*$ provides a right splitting of $q_*$. Therefore, in view of Proposition \ref{thm:proj-bundle-E-trivial} (b), the sequence
	\[
		0\to \bigoplus_{1 \leq m \leq r }^{m \textnormal{ odd}} \K_i(S) \xrightarrow{\Theta} \GW^{[n]}_i(\mathbb{P}^{r},\SO(1)) \xrightarrow{q_\ast} \GW^{[n-r]}_i(S) \to 0
	\]
	is split exact.
\end{proof}

\begin{remark}\label{rmk:projective_space_even_sequence}
	Suppose that $r$ is even. If we consider the projective bundle $\mathbb{P}(\mathcal{E})$, the composition
	\[\bigoplus_{1 \leq m \leq r }^{m \textnormal{ odd}} K_i(S) \xrightarrow{\Theta} \GW_i^{[n]}(\mathbb{P}(\mathcal{E}),\SO(1))  \xrightarrow{q_*} \GW^{[n-r]}_i(S, \det \mathcal{E})\]
	is still trivial, which follows from the same argument as in Proposition \ref{prop:projective_space_even_sequence}.
\end{remark}

\subsection{The general non-split case}
Consider the projective bundle $p: \mathbb{P}(\mathcal{E} \oplus \SO) \to X $.\ The canonical split exact sequence
$0 \rightarrow \SO \rightarrow \mathcal{E} \oplus \SO \rightarrow \mathcal{E} \rightarrow 0$ induces two closed embeddings
$\iota : \mathbb{P}(\SO_S) \rightarrow \mathbb{P}(\mathcal{E} \oplus \SO) \textnormal{ and } \iota':\mathbb{P}(\mathcal{E}) \rightarrow \mathbb{P}(\mathcal{E} \oplus \SO) .$
Let $v: U \hookrightarrow \mathbb{P}(\mathcal{E} \oplus \SO)$ (resp. $v': U' \hookrightarrow \mathbb{P}(\mathcal{E} \oplus \SO)$) be the associated open complements of $\iota$ (resp. $\iota '$). Note that $U' \cong \mathcal{E}$ and $U \cong \SO_{\mathbb{P}(\mathcal{E})}(1)$. Let $\alpha \colon U \to \mathbb{P}(\mathcal{E})$ (resp. $\alpha'\colon U' \to S$) be the projection.

\begin{proof}[Proof of Theorem \ref{thm:projective_bundle_thm} \textnormal{(C)}]
	By the Mayer-Vietoris theorem (cf. \cite[Theorem 1]{schlichting2010mayer}), it suffices to prove this case locally when $\mathcal{E}$ is trivial, i.e. $\mathbb{P}(\mathcal{E}) \cong \mathbb{P}^{r}$, which follows from Proposition \ref{thm:proj-bundle-E-trivial} (c).
\end{proof}

\begin{proof}[Proof of Theorem \ref{thm:projective_bundle_thm} \textnormal{(A)}]
	By the localization theorem, the sequence
	\[
		\xymatrix{\GW^{[n]}(\mathbb{P}(\mathcal{E})) \ar[r]^-{\iota '_\ast} & \GW^{[n+1]}(\mathbb{P}(\mathcal{E}\oplus \SO_S), \SO(1)) \ar[r]^-{v'^\ast} & \GW^{[n+1]}(U')\ar[r]^-{\partial} & \GW^{[n]}(\mathbb{P}(\mathcal{E}))[1] }
	\]
	is a distinguished triangle in $\mathcal{SH}$.\ It is enough to show that the diagrams
	$$
		\xymatrix{
		\GW^{[n]}(S) \ar[d]^-{q^*} \ar@{}[dr]|-{\diagram \label{diag:projective_iota_prime_HF_1}} \ar[r]^-{F} & K(S) \ar[d]^-{\mu^1}\ar[r]^{H} & \GW^{[n+1]}(S) \ar[d]^-{\alpha'^*}\ar[r]^-{\eta \cup} & \GW^{[n]}(S)[1] \ar[d]^-{q^*[1]}
		\\	
		\GW^{[n]}(\mathbb{P}(\mathcal{E})) \ar[r]^-{\iota '_\ast} & \GW^{[n+1]}(\mathbb{P}(\mathcal{E}\oplus \SO_S), \SO(1)) \ar[r]^-{v'^\ast} & \GW^{[n+1]}(U')\ar[r]^-{\partial} & \GW^{[n]}(\mathbb{P}(\mathcal{E}))[1]
		}
	$$
	and
	$$
		\xymatrix{
		\bigoplus\limits_{2 \leq m \leq r }^{m \textnormal{ even}} K(S)  \ar@{}[dr]|-{\diagram \label{diag:projective_iota_prime_HF_2}}  \ar[d]^-{\Theta_{\mathrm{even}}} \ar[r]^-{\id} & \bigoplus\limits_{2 \leq m \leq r }^{m \textnormal{ even}} K(S) \ar[d]^-{\sum \mu^{m+1}}\ar[r] \ar@{}[dr]|-{\diagram \label{diag:projective_iota_prime_HF_3}} & 0 \ar[d] \ar[r] &  \bigoplus\limits_{2 \leq m \leq r }^{m \textnormal{ even}} K(S)[1] \ar[d]^-{\Theta_{\mathrm{even}}[1]}
		\\ 
		\GW^{[n]}(\mathbb{P}(\mathcal{E})) \ar[r]^-{\iota '_\ast} & \GW^{[n+1]}(\mathbb{P}(\mathcal{E}\oplus \SO_S), \SO(1)) \ar[r]^-{v'^\ast} & \GW^{[n+1]}(U')\ar[r]^-{\partial} & \GW^{[n]}(\mathbb{P}(\mathcal{E}))[1]
		}
	$$
	are maps of distinguished triangles. All the squares in the upper ladder diagram commute by similar arguments as in the proof of Theorem \ref{thm:projective_space} (b). The diagram $\diag{\ref{diag:projective_iota_prime_HF_2}}$ commutes by Lemma \ref{lma:twist_closed_immersion}. The diagram $\diag{\ref{diag:projective_iota_prime_HF_3}}$ commutes since $\mu^{m+1}$ factor through $\iota'_\ast$. Take the sum of two maps of distinguished triangles, and note that $\Theta_{\mathrm{odd}}=\big(\sum\mu^{m+1},~  \mu^{1} \big) $ is a stable equivalence by (a). It follows that $(\Theta_{\mathrm{even}},q^*)$ is a stable equivalence.
\end{proof}

\begin{proof}[Proof of Theorem \ref{thm:projective_bundle_thm} \textnormal{(B)}]
	Recall the following setting (cf. \cite[(A.2)]{hudson2022witt})
	\[
		\xymatrix{
		\mathbb{P}(\SO_S) \ar@{^{(}->}[r]^-{\iota} & \mathbb{P}(\mathcal{E}\oplus \SO_S)  & \ar@{_{(}->}[l]_-{v}  \ar@{_{(}->}[ld]_-{\tilde{v}} U \ar[d]^-{\alpha}  \\
		E\ar[u]^-{\widetilde{\pi}} \ar@{^{(}->}[r]_-{\tilde{\iota}} & B  \ar[u]^-{\pi} \ar@{-->}[r]_-{\tilde{\alpha}} & \mathbb{P}(\mathcal{E}),
		}
	\]
	where $\mathbb{P}(\SO_S) \cong S, E = \mathbb{P}(\SO_{\mathbb{P}(\mathcal{E})}) \cong \mathbb{P}(\mathcal{E})$ and $B = \mathbb{P}(\SO_{\mathbb{P}(\mathcal{E})}(-1) \oplus \SO_{\mathbb{P}(\mathcal{E})})$. This setting satisfies Hypothesis \ref{hypo:BC}, and so by Theorem \ref{theo:connecting-hom} (A), we have the following diagram
	$$
		\xymatrix{
		0 \ar[r] &   \GW^{[n-r]}_i(S,\det \mathcal{E}) \ar[r]^-{\iota_\ast} \ar[d]_-{c(\mathcal{Q})  q^\ast} \ar@{}[dr]|-{\diagram \label{diag:proj_bundle_even_twsit_excess}}
		&
		\GW^{[n+1]}_i(\mathbb{P}(\mathcal{E}\oplus \SO_S)) \ar[r]^-{v^\ast} \ar[d]_-{\pi^\ast}
		&
		\GW^{[n+1]}_i(U) \ar@{=}[d] \ar[r] &0
		\\ 
		0 \ar[r] & \GW^{[n]}_i(\mathbb{P}(\mathcal{E}), \SO(1)) \ar[r]^-{\tilde{\iota}_\ast} \ar[d]_-{q_\ast}
		&
		\GW^{[n+1]}_i(B) \ar[r]^-{\tilde{v}^\ast} \ar[d]_-{\pi_\ast} \ar@{}[dr]|-{\diagram \label{diag:proj_bundle_even_bese}}
		&
		\GW^{[n+1]}_i(U) \ar@{=}[d] \ar[r] &0
		\\ 
		0 \ar[r] &\GW^{[n-r]}_i(S,\det \mathcal{E}) \ar[r]^-{\iota_\ast}
		&
		\GW^{[n+1]}_i(\mathbb{P}(\mathcal{E}\oplus \SO_S)) \ar[r]^-{v^\ast}
		&
		\GW^{[n+1]}_i(U) \ar[r] &0
		}
	$$
	consists of short split exact sequences in horizontal lines (note that $\tilde{v}^*$ has a right splitting $ \tilde{\alpha}^* (\alpha^*)^{-1}$).\ The diagram $\diag{\ref{diag:proj_bundle_even_twsit_excess}}$ commutes by the excess intersection formula (cf.\ Theorem \ref{thm:excess_intersection}).\ The diagram $\diag{\ref{diag:proj_bundle_even_bese}}$ commutes by the base change formula (cf.\ Lemma \ref{prop:regular_smooth_base_change}).
	Since the map $\pi$ factors as the composition $ B \stackrel{i}\hookrightarrow \mathbb{P}(\mathcal{E}) \times_S \mathbb{P}(\mathcal{E} \oplus \SO_S) \xrightarrow{\wp} \mathbb{P}(\mathcal{E} \oplus \SO_S)$
	where $i$ is a regular immersion and $\wp$ is the projection, we see that $\pi_\ast \pi^\ast = \mathrm{Id}$. Moreover, we have that 
	\[
		\begin{aligned}
			\pi_\ast  \pi^\ast(y) & =\wp_*i_* i^* \wp^*(y) && \text{(Composition of pushforwards and pullbacks)}\\
			                           & =  \wp_*i_* (1_S \cup i^* \wp^*(y)) && \text{($1_S \in \GW_0^{[0]}(S)$ is the identity element)}\\
			                           & =  \wp_*i_* (1_S) \cup y && \text{(Projection formulas, cf. Proposition  \ref{prop:projection_formula} and  \ref{prop:projection_projective_space})}\\
			                           & =  \pi_* (1_S) \cup y && \text{(Composition of pushforwards)}\\
			                           & =  y && \text{($\pi_*(1_S) = 1$, cf.\ \cite[Proposition 3.15 (a)]{balmer2012witt})}
		\end{aligned}
	\]
	It follows that $c(\mathcal{Q} )q^\ast$ is a right splitting of $q_*$, i.e., $q_\ast c(\mathcal{Q} ) q^\ast  = \id$. Therefore, the map
	\[(\Theta_{\textnormal{odd}}, c(\mathcal{Q}) q^*) \colon \bigoplus_{1 \leq m \leq r }^{m \textnormal{ odd}} \K_i(S) \oplus \GW^{[n-r]}_i(S,\det \mathcal{E}) \xrightarrow{\simeq} \GW^{[n]}_i(\mathbb{P}^{r},\SO(1))
	\]
	is an isomorphism by Proposition \ref{prop:projective_space_even_sequence}.\ For the general case, the result follows from the Mayer-Vietoris theorem.
\end{proof}

\begin{proof}[Proof of Theorem \ref{thm:projective_bundle_thm} \textnormal{(D)}]
	We want to show that the candidate triangle
	\begin{equation}\label{eq:proj-bundle-odd-gw-exact}
		\xymatrix{ \mathbf{X}' \ar[r]^-{\mathbf{u}'} & \mathbf{Y}' \ar[r]^-{\mathbf{v}'} & \mathbf{Z}' \ar[r]^-{\mathbf{w}'} & \mathbf{X}'[1]}
	\end{equation}
	is a distinguished triangle in the stable homotopy category $\mathcal{SH}$ of spectra, where
	\[\mathbf{X}'= \GW^{[n]}(\mathbb{P}(\mathcal{E})), \quad \mathbf{Y}' = \GW^{[n-r]}(S,\det\mathcal{E}), \quad \mathbf{Z}' = \bigoplus_{2 \leq m \leq r }^{m \textnormal{ even}}  \K(S)[1] \oplus \GW^{[n]}(S)[1]\]
	and
	\[\mathbf{u}' = q_\ast,  \quad \mathbf{v}' = {\begin{pmatrix} 0\\ \eta \cup c(\mathcal{E})\end{pmatrix}}, \quad \mathbf{w}'= \big(\Theta_{\mathrm{even}}, q^\ast \big).\]

	1) Let $p : \mathbb{P}(\mathcal{E}\oplus \SO_S) \to S$ be the canonical projection.\ By Lemma \ref{lem:codim-one} and the diagram $\diag{\ref{diag:proj_space_partial}}$, the following sequence
	\[ \xymatrix{\mathbf{X} \ar[r]^-{\mathbf{f}} & \mathbf{Y} \ar[r]^-{\mathbf{v}} & \mathbf{Z} \ar[r]^-{\mathbf{w}} & \mathbf{X}[1]} \]
	is a distinguished triangle, where
	\[\mathbf{X}= \GW^{[n]}(\mathbb{P}(\mathcal{E})), \quad \mathbf{Y}=  \GW^{[n+1]}(\mathbb{P}(\mathcal{E}\oplus \SO_S), \SO(1)), \quad \mathbf{Z} = \GW^{[n+1]} (S) \]
	and
	\[\mathbf{f}= \iota '_\ast, \quad \mathbf{v} = (\alpha^\ast)^{-1} v'^\ast,  \quad \mathbf{w} =\eta \cup q^\ast .\]

	2) Note that $\mathbf{X} = \mathbf{X}'$, and we form the following diagram
	\begin{equation}\label{eq:map-of-dis-triang}
		\xymatrix{
		\mathbf{X} \ar@{=}[d] \ar[r]^-{\mathbf{f}} & \mathbf{Y} \ar[d]_-{\mathbf{g}} \ar[r]^-{\mathbf{v}} & \mathbf{Z}
		\ar[r]^-{\mathbf{w}}\ar[d]_-{\mathbf{h}} & \mathbf{X}[1] \ar@{=}[d]
		\\
		\mathbf{X}'  \ar[r]^-{\mathbf{u}'} & \mathbf{Y}' \ar[r]^-{\mathbf{v}'} & \mathbf{Z}' \ar[r]^-{\mathbf{w}'} & \mathbf{X}'[1] }
	\end{equation}
	where
	$$\mathbf{g} = p_\ast, \quad \mathbf{h} = \begin{pmatrix} 0 \\ \eta\cup \end{pmatrix}.$$
	Note that $\mathbf{gf} = p_\ast \iota '_\ast = q_\ast = \mathbf{u}'$.

	3) To show the candidate triangle \eqref{eq:proj-bundle-odd-gw-exact} is a distinguished triangle, it is enough to prove that the diagram \eqref{eq:map-of-dis-triang} is a map of candidate triangles and the cone of the diagram \eqref{eq:map-of-dis-triang} is a distinguished triangle.

	\begin{sublemma}\label{sublem:map-disting-trian}
		The diagram \eqref{eq:map-of-dis-triang} is a map of candidate triangles.
	\end{sublemma}
	\begin{proof}
		All we need to show is that $\mathbf{v'}\mathbf{g} = \mathbf{h}\mathbf{v} $. Consider the following exact sequence:
		\[
			0 \to \SO_{\mathbb{P}(\mathcal{E}\oplus \SO_S)}(-1) \to p^\ast (\mathcal{E}\oplus \SO_S) \to \mathcal{G} \to 0
		\]
		with $\mathcal{G}$ the universal quotient bundle. The diagrams
		$$
			\xymatrix@C=10pt{
			\K(S)    \ar[d]_-{\mu^{1}} \ar[r]^-{ H } \ar@{}[dr]|-{\diagram \label{diag:projective_bundle_odd_diag_2}} &     	 \GW^{[n-r]}(S, \det \mathcal{E})    \ar[d]_-{\alpha'^\ast}
			& \GW^{[n+1]} (S) \ar[d]_-{ c(\mathcal{G}) p^\ast}  \ar[r]^-{ c(\mathcal{E})} \ar@{}[dr]|-{\diagram \label{diag:projective_bundle_odd_diag_3}} &\GW^{[n+1]} (S) \ar[d]_-{\alpha'^\ast}
			\\ 
			\GW^{[n+1]}(\mathbb{P}(\mathcal{E}\oplus \SO_S), \SO(1)) \ar[r]^-{v'^\ast}         &
			\GW^{[n+1]} (U')  & 	\GW^{[n+1]}(\mathbb{P}(\mathcal{E}\oplus \SO_S), \SO(1)) \ar[r]^-{v'^\ast}         &
			\GW^{[n+1]} (U')
			}
		$$
		commute. Note that $\mu^{1} = H p^\ast$ and $p v' = \alpha'$. Hence, $v'^\ast  \mu^{1} \simeq H  \alpha'^\ast \simeq \alpha'^\ast H$, which shows the commutativity of the diagram $\diag{\ref{diag:projective_bundle_odd_diag_2}}$. The commutativity of the diagram $\diag{\ref{diag:projective_bundle_odd_diag_3}}$ follows from the proof of \cite[Theorem A.7]{hudson2022witt}. Moreover, since $\eta \cup H =0 $ and $\alpha'^*$ is an isomorphism, we obtain the following formula
		\[
			\eta \cup  \big(\alpha'^\ast\big)^{-1}v'^\ast (\mu^{1} , c(\mathcal{G}) p^\ast)  \simeq \eta \cup (H, c(\mathcal{E}))   \simeq  (0, \eta\cup c(\mathcal{E})).
		\]
		Furthermore, $v'^\ast \mu^m = 0$ for $m \neq 1$ odd by the diagram $\diag{\ref{diag:projective_iota_prime_HF_3}}$. It follows that
		$$ \eta \cup (\alpha'^\ast)^{-1}v'^\ast \Theta_{\mathrm{odd}} = 0.$$
		In view of (B) and Remark \ref{rmk:projective_space_even_sequence}, we have $p_\ast \big(\Theta_{\mathrm{odd}}, c(\mathcal{G}) p^\ast\big) = (0,\id)$. Therefore,
		\[
			\eta \cup (\alpha'^\ast)^{-1}v'^\ast \big(\Theta_{\mathrm{odd}}, c(\mathcal{G}) p^\ast\big)
			\simeq \big(0,\eta\cup c(\mathcal{E})\big)          \simeq \eta\cup c(\mathcal{E}) p_\ast \big(\Theta_{\mathrm{odd}}, c(\mathcal{G}) p^\ast\big).
		\]
		By (B) the map $\big(\Theta_{\mathrm{odd}}, c(\mathcal{G}) p^\ast\big)  $
		is an equivalence. Hence, $\eta \cup (\alpha'^\ast)^{-1}v'^\ast = \eta\cup c(\mathcal{E}) p_\ast$.
	\end{proof}

	\begin{sublemma}\label{sublem:distingsh}
		The mapping cone of the diagram \eqref{eq:map-of-dis-triang} is a distinguished triangle.
	\end{sublemma}
	\begin{proof}

		The mapping cone of the diagram \eqref{eq:map-of-dis-triang} can be written as a direct sum of the following two candidate triangles (cf. \cite[the proof of Lemma 1.4.3]{neeman2001triangulated})
		\[\xymatrix{\mathbf{X} \ar[r] & 0 \ar[r] & \mathbf{X}[1] \ar[r]^-{1} & \mathbf{X}[1] } \]
		and
		\begin{equation}\label{eq:candidate-triangle}
			\xymatrix{
			\mathbf{Y} \ar[r]^-{\everymath={\scriptstyle}\begin{pmatrix}\mathbf{g} \\ -\mathbf{v} \end{pmatrix}} & \mathbf{Y}' \oplus \mathbf{Z} \ar[r]^-{\big(\mathbf{v}', \mathbf{h}\big)} & \mathbf{Z}' \ar[r]^-{\mathbf{f}[1] \mathbf{w}'} & \mathbf{Y}[1]
			}
		\end{equation}
		It is enough to show that the candidate triangle \eqref{eq:candidate-triangle} is distinguished.\ Consider the following three maps from distinguished triangles to the candidate triangle \eqref{eq:candidate-triangle}:
		$$
			\xymatrix@R=30pt@C=60pt{
			\mathbf{Y}' \ar[r]^-{\id} \ar[d]_-{c(\mathcal{G})p^*} \ar@{}[dr]|-{\diagram \label{diag:proejctive_bundle_D_1}} & \mathbf{Y}' \ar[r] \ar[d]_-{\everymath={\scriptstyle}\begin{pmatrix}
				\id \\ -c(\mathcal{E})
			\end{pmatrix}} & 0 \ar[r] \ar[d]  &   \mathbf{Y}'[1] \ar[d]_-{c(\mathcal{G})p^*}
			\\ 
			\mathbf{Y} \ar[r]_-{\everymath={\scriptstyle}\begin{pmatrix}
				p_* \\ - (\alpha'^*)^{-1} v'^*
			\end{pmatrix}} & \mathbf{Y}' \oplus \mathbf{Z} \ar[r]_-{\everymath={\scriptstyle}\begin{pmatrix}
				0 & 0 \\ \eta\cup c(\mathcal{E}) & \eta\cup
			\end{pmatrix}} & \mathbf{Z}' \ar[r]_-{ \iota '_\ast \big(\Theta_{\mathrm{even}}, q^\ast \big) } & \mathbf{Y}[1] \\
			\mathbf{A} \ar[r]^-{H} \ar[d]_-{\mu^1} \ar@{}[dr]|-{\diagram \label{diag:projective_bundle_D_4}} & \mathbf{Z} \ar[r]^-{\eta \cup } \ar[d]_-{\everymath={\scriptstyle}\begin{pmatrix}
				0 \\ -\id
			\end{pmatrix}} & \mathbf{C} \ar[r]^-{-F} \ar[d]_-{\everymath={\scriptstyle}\begin{pmatrix}
				0 \\ -\id
			\end{pmatrix}} \ar@{}[dr]|-{\diagram \label{diag:proejctive_bundle_D_6}} &   \mathbf{A}[1] \ar[d]_-{\mu^1}
			\\ 
			\mathbf{Y} \ar[r]_-{\everymath={\scriptstyle}\begin{pmatrix}
				p_* \\ - (\alpha'^*)^{-1} v'^*
			\end{pmatrix}} & \mathbf{Y}' \oplus \mathbf{Z} \ar[r]_-{\everymath={\scriptstyle}\begin{pmatrix}
				0 & 0 \\ \eta\cup c(\mathcal{E}) & \eta\cup
			\end{pmatrix}} & \mathbf{Z}' \ar[r]_-{ \iota '_\ast \big(\Theta_{\mathrm{even}}, q^\ast \big) } & \mathbf{Y}[1] \\
			\mathbf{E} \ar[r] \ar[d]_-{\sum\mu^{m+1}} \ar@{}[dr]|-{\diagram \label{diag:proejctive_bundle_D_7}} & 0  \ar[r] \ar[d] & \mathbf{E}[1] \ar[r]^-{\id} \ar[d]_-{\everymath={\scriptstyle}\begin{pmatrix}
				\id \\ 0
			\end{pmatrix}} \ar@{}[dr]|-{\diagram \label{diag:proejctive_bundle_D_9}} &   \mathbf{E}[1] \ar[d]_-{\sum \mu^{m+1}}
			\\ 
			\mathbf{Y} \ar[r]_-{\everymath={\scriptstyle}\begin{pmatrix}
				p_* \\ - (\alpha'^*)^{-1} v'^*
			\end{pmatrix}} & \mathbf{Y}' \oplus \mathbf{Z} \ar[r]_-{\everymath={\scriptstyle}\begin{pmatrix}
				0 & 0 \\ \eta\cup c(\mathcal{E}) & \eta\cup
			\end{pmatrix}} & \mathbf{Z}' \ar[r]_-{ \iota '_\ast \big(\Theta_{\mathrm{even}}, q^\ast \big) } & \mathbf{Y}[1]}
		$$
		where $ \mathbf{E} = \bigoplus\limits _{2 \leq m \leq r }^{m \textnormal{ even}}  \K(S)$ and the sequence
		\[ \xymatrix{\mathbf{A} \ar[r]^-{H} & \mathbf{Z} \ar[r]^-{\eta \cup } & \mathbf{C} \ar[r]^-{-F} & \mathbf{A}[1]} \]
		is the (shifted) algebraic Bott sequence
		\[ \xymatrix{\K(S) \ar[r]^-{H} & \GW^{[n+1]}(S) \ar[r]^-{\eta \cup} & \GW^{[n]}(S)[1] \ar[r]^-{-F} &\K(S)[1]}.  \]

		\begin{itemize}
			\item [-] The diagram $\diag{\ref{diag:proejctive_bundle_D_1}}$ commutes by (B) and the diagram $\diag{\ref{diag:projective_bundle_odd_diag_3}}$.
			\item [-] The diagram $\diag{\ref{diag:projective_bundle_D_4}}$ commutes by the diagram $\diag{\ref{diag:projective_bundle_odd_diag_2}}$.
			\item [-] The diagram $\diag{\ref{diag:proejctive_bundle_D_6}}$ commutes by the diagram $\diag{\ref{diag:projective_iota_prime_HF_1}}$.
			\item [-] The diagram $\diag{\ref{diag:proejctive_bundle_D_7}}$ commutes by the diagram $\diag{\ref{diag:projective_iota_prime_HF_3}}$ and Remark \ref{rmk:projective_space_even_sequence}.
			\item [-] The diagram $\diag{\ref{diag:proejctive_bundle_D_9}}$ commutes as $\iota '_\ast \mu^m = \mu^{m+1}$.
			\item [-] Other diagrams commute by definition.
		\end{itemize}
		Finally, we form the sums of these three maps of candidate triangles
		\[
			\xymatrix{
				\mathbf{Y'} \oplus \mathbf{A} \oplus \mathbf{E} \ar[r] \ar[d] & \mathbf{Y'}\oplus \mathbf{Z} \ar[r]  \ar[d]  & \mathbf{C} \oplus \mathbf{E}[1] \ar[r] \ar[d]  &   \mathbf{A}[1] \oplus \mathbf{E}[1] \ar[d]
				\\ 
				\mathbf{Y} \ar[r] & \mathbf{Y}' \oplus \mathbf{Z} \ar[r]
				& \mathbf{Z}' \ar[r] & \mathbf{Y}[1]  }
		\]
		where all the vertical maps are equivalences (Note that the left vertical map is an equivalence by (B).\ The upper horizontal line is a distinguished triangle, and the result follows.
	\end{proof}

	Theorem \ref{thm:projective_bundle_thm} (D) follows from Sublemma \ref{sublem:map-disting-trian} and \ref{sublem:distingsh}.
\end{proof}

Suppose that $\phi : S' \to S$ is a morphism of regular schemes. Then, we can form the following fibre product
\[\xymatrix{\mathbb{P}(\mathcal{E})  \times_S S' \ar[r]^-{\phi'} \ar[d]^-{q'} &   \mathbb{P}(\mathcal{E}) \ar[d]^-{q} \\
	S' \ar[r]^-{\phi} & S     }\]
Note that $\mathbb{P}(\mathcal{E})  \times_S S' \cong \mathbb{P}(\phi^* \mathcal{E})$ which is a projective bundle over $S'$, cf. \cite[Remark 13.27 p.\ 383]{gortz2020algebraic}.
The following result is a direct application of the projective bundle theorem:
\begin{proposition}\label{prop:base-change-formula-proj-bundle}
	For any line bundle $L$ on $S$, the following diagram
	\[\xymatrix{ \GW^{[n]}(\mathbb{P}(\phi^*\mathcal{E}), \omega_{q'} \otimes q'^*\phi^*L ) \ar[d]^-{q'_*} & \ar[l]_-{\phi'^*}  \GW^{[n]}(\mathbb{P}(\mathcal{E}), \omega_q \otimes q^*L) \ar[d]^-{q_*} \\
		\GW^{[n-r]}( S', \phi^*L) & \ar[l]_-{\phi^*}  \GW^{[n-r]}(S, L)
		}\]
	of spectra	commutes up to homotopy.
\end{proposition}
\begin{proof}
	Suppose that the rank of the bundle $\mathcal{E}$ is odd (i.e. $r$ is even). Since we have the isomorphism
	$$(\Theta_{\textnormal{odd}}, c(\mathcal{Q}) q^*) \colon \bigoplus_{1 \leq m \leq r }^{m \textnormal{ odd}} \K_i(S) \oplus \GW^{[n-r]}_i(S,\det \mathcal{E}) \xrightarrow{\simeq} \GW^{[n]}_i(\mathbb{P}(\mathcal{E}),\SO(1)),   $$
	it is enough to show that
	$$ q'_\ast \phi'^* (\Theta_{\textnormal{odd}}, c(\mathcal{Q})  q^*) = \phi^* q_\ast (\Theta_{\textnormal{odd}}, c(\mathcal{Q})  q^*) .$$
	Note that $\phi'^*\Theta_{\textnormal{odd}} = \Theta'_{\textnormal{odd}} \phi^* $ follows from the definition of $\Theta$. Therefore, we see that
	$$ q'_\ast \phi'^* \Theta_{\textnormal{odd}} = q'_\ast \Theta'_{\textnormal{odd}} \phi^* = 0 = \phi^* q_\ast \Theta_{\textnormal{odd}} .$$
	Moreover, we have that
	$$ q'_\ast \phi'^*  c(\mathcal{Q})  q^* = q'_\ast c(\phi'^*\mathcal{Q}) \phi'^* q^* = q'_\ast c(\phi'^*\mathcal{Q}) q'^* \phi^* = \phi^* = \phi^* q_\ast c(\mathcal{Q})  q^* .$$

	Now, assume that the rank of bundle $\mathcal{E}$ is even (i.e. $r$ is odd). Consider the following diagram
	\[ \xymatrix{
		\mathbb{P}(\phi^*\mathcal{E}) \ar[r]^-{\iota'}  \ar[d]^-{\phi'} & \mathbb{P}(\phi^*(\mathcal{E} \oplus \SO_{S}) ) \ar[d]^-{\phi''} \ar[r]^-{p'} & S' \ar[d]^-{\phi}
		\\
		\mathbb{P}(\mathcal{E}) \ar[r]^-{\iota}& \mathbb{P}(\mathcal{E} \oplus \SO_S ) \ar[r]^-{p} & S
		}\]
	where we have $q=\iota p$ and $q'= \iota' p' $. Note that the left square is tor-independent. It follows that
	$$
		\phi^* q_*  =   \phi^* p_* \iota_*
		=     p'_* \phi''^* \iota_*
		=     p'^* \iota'_* \phi'^*
		= q'_* \phi'^*
	$$
	where the second equation follows from the case when $r$ is even, and the third equation is true by the excess intersection formula (cf. Theorem \ref{thm:excess_intersection}).
\end{proof}

\begin{corollary}[Base change formula for projective morphisms]\label{coro:projective_bundle_base_change}
	Let
	\[
		\xymatrix{{X'} \ar[r]^-{\bar{g}} \ar[d]_-{\bar{f}}& X \ar[d]_-{f}\\
		Y' \ar[r]^-{g}& Y}
	\]
	be a fibre product of regular schemes. Suppose that $f$ is flat and $g$ is a projective morphism. Then, the formula
	$$f^* g_* =\bar{g}_* \bar{f}^*$$
	holds on the level of $\GW$-spectra.
\end{corollary}
\begin{proof}
	This follows from Theorem \ref{thm:excess_intersection} and Proposition \ref{prop:base-change-formula-proj-bundle}.
\end{proof}
\section{Grassmannians}\label{sect:grassmanian}
Let $S$ be a regular scheme. Let $\SV$ be a vector bundles over $S$ of rank $r$. Following \cite{balmer2012witt}, we adopt the notation $\SP_d \lhd \SV$ to indicate that $\SP_d$ is a subbundle of $\SV$ of rank $d$ which is locally a direct summand (i.e. $\SV/\SP_d$ is a vector bundle), and we call $\SP_d$ an \textit{admissible subbundle} of $\SV$.

\begin{definition}[{\cite[Theorem 9.7.4]{ega1}}] Let $d$ be an integer satisfying $0\leq d \leq r$.\ The \textit{Grassmann functor} is the functor
	\[\arraycolsep=1pt
		\begin{array}{lcl}
			\Gr_{d}(\SV)(-) \colon
			 & \Sch_S & \to \Sets                                  \\
			 & X      & \mapsto \Big\{ \SP_{d} \lhd f^*\SV  \Big\}
		\end{array}
	\]
	for every given $S$-scheme $f: X \to S$.\ The $S$-scheme representing the functor $ \Gr_{d}(\SV)(-)$ is denoted by $p: \Gr_{d}(\SV) \to S $ and is called the \textit{Grassmann variety} over $S$.\ The universal element of the Grassmann functor is given by a unique admissible subbundle $\mathcal{T}_d \lhd \SV_{\Gr_{d}(\SV)}$ called the \textit{tautological bundle} on $\Gr_{d}(\SV)$.
	Define $m:= r-d$ and let
	$$\Gr_d(m): = \Gr_{d}( \SO_S^{\oplus r}). $$
\end{definition}

The determinant line bundle of the tautological bundle   $\mathcal{T}_d$ which is denoted by $\Delta_d:=\det(\mathcal{T}_d)$ plays an important role in our computation.

\subsection{The flag variety}
Fix a complete flag
\[
	0 = \SV_{0} \lhd \SV_{1} \lhd \ldots \lhd \SV_{r} = \SV
\]
of admissible subbundles of $\SV$ with $\mathrm{rank}(\SV_s) = s$. Recall the following construction in {\cite[Definition 1.6]{balmer2012witt}}.

\begin{definition}\label{def:flag-variety}
	Given a pair of $k$-tuples $ \underline{d} = (d_1 , \ldots , d_k)$ and $\underline{m}:=(m_1, \ldots, m_k)$ such that $$0<d_1 < d_2 < \ldots < d_k \quad \textnormal{   and   } \quad d_1 + m_1 \leq d_2 + m_2 \leq \ldots \leq d_k+m_k .$$ Define $\underline{r} : = (r_1 ,\ldots, r_k)$ where $r_i : = d_i +m_i$ for $i = 1, \ldots, k$. Set
	$$\SV_{\bullet }: = \big\{ \SV_{r_1} \lhd \ldots  \lhd \SV_{r_k} \big\}$$ The \textit{flag functor} is the functor
	\[\arraycolsep=1pt
		\begin{array}{lcl}
			\mathrm{Fl}_{\underline{d}}(\SV_{\bullet })(-) \colon
			 & \Sch_S & \to \Sets                                               \\
			 & X      & \mapsto 	\mathrm{Fl}_{\underline{d}}(\SV_{\bullet })(X)
		\end{array}
	\]
	defined by
	\[
		\mathrm{Fl}_{\underline{d}}(\SV_{\bullet })(X)  :=
		{\left\{\begin{matrix}\SP_{d_1}        & \lhd & \SP_{d_2}        & \lhd & \ldots & \lhd & \SP_{d_k}        \\
              \triangle        &      & \triangle        &      &        &      & \triangle        \\
              f^\ast \SV_{r_1} & \lhd & f^\ast \SV_{r_2} & \lhd & \ldots & \lhd & f^\ast \SV_{r_k}\end{matrix}\right\}}
	\]
	for any $S$-scheme $f:X\to S$. The $S$-scheme representing the functor $ \mathrm{Fl}_{\underline{d}}(\SV_{\bullet })(-)$ is denoted by $p: \mathrm{Fl}_{\underline{d}}(\SV_{\bullet }) \to S $ and is called the \textit{flag variety} over $S$. If no confusion occurs, we shall drop the mention of $f^\ast$ in the sequel.   The universal element of the functor $\mathrm{Fl}_{\underline{d}}(\SV_{\bullet })(-)$ is given by
	\[\begin{matrix}\mathcal{T}_{d_1} & \lhd & \mathcal{T}_{d_2} & \lhd & \ldots & \lhd & \mathcal{T}_{d_k}\\ \triangle & & \triangle & &  & & \triangle \\  p^\ast \SV_{r_1} & \lhd & p^\ast \SV_{r_2} & \lhd & \ldots & \lhd & p^\ast \SV_{r_k}. \end{matrix}\]
	The determinant line bundle of the vector bundle  $\mathcal{T}_{d_i}$ is denoted by $\Delta_{d_i}:=\det (\mathcal{T}_{d_i})$ for $1 \leq i \leq k$.
\end{definition}

\begin{remark}
	Note that $\Gr_{d}(\SV) = \Fl_{d}(\SV)$. The flag variety $p:\mathrm{Fl}_{\underline{d}}(\SV_{\bullet }) \to S $ is smooth over $S$ (cf. \cite[Proposition 1.13]{balmer2012witt}).
\end{remark}

\begin{proposition}\label{prop:Pic-flag}
	The map
	\[
		\begin{aligned}
			\Pic(S) \oplus \bigoplus_{1\leq i \leq k} \Z\Delta_{d_i} & \xrightarrow{\cong} \Pic(\mathrm{Fl}_{\underline{d}}(\SV_{\bullet }))                       \\
			(L, m_1, m_2, \ldots, m_k)                                          & \mapsto p^\ast L \otimes \bigotimes_{1\leq i \leq k} \Delta_{d_i}^{\otimes m_i}.
		\end{aligned}
	\]
	is an isomorphism.
\end{proposition}
\begin{proof}
	See \cite[Proposition 1.13]{balmer2012witt}.
\end{proof}
\begin{remark}
	In particular, Proposition \ref{prop:Pic-flag} states that $\Pic(\Gr_{d}(\SV)) \cong \Pic(S) \oplus \Z\Delta_{d}$. Since $\Pic(S)$ is a direct summand of $\Pic(\Gr_{d}(\SV))$, we may omit $p^\ast$ and denote by $L$ the corresponding class in $\Pic(\Gr_{d}(\SV))$ for $L \in \Pic(S)$.

\end{remark}
\subsection{Canonical maps between flag varieties}\label{sect:maps-flag-variety}
Suppose that we have another $k$-tuple $\underline{\tilde{r}} : = (\tilde{r}_1, \ldots, \tilde{r}_k)$  such that  $\tilde{r}_1\leq \ldots\leq  \tilde{r}_k$ and $r_i \leq \tilde{r}_i$ for $1 \leq i \leq k$. Set $\widetilde{\SV}_{\bullet }$ to be the flag $\SV_{\tilde{r}_1} \lhd \ldots \lhd \SV_{\tilde{r}_k}$. Then, we have a regular closed embedding
$$
	\begin{aligned}
		\iota: \mathrm{Fl}_{\underline{d}}(\SV_{\bullet }) & \hookrightarrow \Fl_{\underline{d}}(\widetilde{\SV}_{\bullet })
		\\
		{\left\{\begin{matrix}\SP_{d_1} & \lhd & \ldots & \lhd & \SP_{d_k} \\
              \triangle &      &        &      & \triangle \\
              \SV_{r_1} & \lhd & \ldots & \lhd & \SV_{r_k}\end{matrix}\right\}}
		                                                   & \mapsto
		{\left\{\begin{matrix}\SP_{d_1}         & \lhd & \ldots & \lhd & \SP_{d_k}         \\
              \triangle         &      &        &      & \triangle         \\
              \SV_{\tilde{r}_1} & \lhd & \ldots & \lhd & \SV_{\tilde{r}_k}\end{matrix}\right\}}.
	\end{aligned}
$$

Now, let $J= \{j_1, \ldots, j_{\widehat{k}}\}$ be a subset of the set $I:=\{1, \ldots, k\}$. Consider the $\widehat{k}$-tuples $\underline{\widehat{d}}:= (d_{j_1} , \ldots , d_{j_{\widehat{k}}})$.  Set $\widehat{\SV}_{\bullet } $ to be the flag $\SV_{r_{j_1}} \lhd \ldots \lhd \SV_{r_{j_{\widehat{k}}}}$.\ Then, one can construct a natural projective morphism
$$
	\begin{aligned}
		\pi : \mathrm{Fl}_{\underline{d}}(\SV_{\bullet }) & \rightarrow \Fl_{\underline{\widehat{d}}}(\widehat{\SV}_{\bullet })
		\\
		{\left\{\begin{matrix}\SP_{d_1} & \lhd & \ldots & \lhd & \SP_{d_k} \\
              \triangle &      &        &      & \triangle \\
              \SV_{r_1} & \lhd & \ldots & \lhd & \SV_{r_k}\end{matrix}\right\}}
		                                                  & \mapsto
		{\left\{\begin{matrix}\SP_{d_{j_1}} & \lhd & \ldots & \lhd & \SP_{d_{j_{\widehat{k}}}} \\
              \triangle     &      &        &      & \triangle                 \\
              \SV_{r_{j_1}} & \lhd & \ldots & \lhd & \SV_{r_{j_{\widehat{k}}}}\end{matrix}\right\}}.
	\end{aligned}
$$
\begin{remark}
	If neccesary, we may decorate the notations $\iota$ and $\pi$ without defining them cumbersomely, e.g. $\iota_+$ and $\pi_+$ etc..
\end{remark}

\subsection{The blow-up  setup}\label{subsec:grassmannian_setup}
Consider the regular closed embedding $\iota: \Gr_{d}(\SV^1) \hookrightarrow \Gr_{d}(\SV)$, of which the open complement shall be denoted by $v: \Gr_{d}(\SV) \backslash \Gr_{d}(\SV^1)  \to \Gr_{d}(\SV)$.
The scheme $\Gr_{d}(\SV) \backslash \Gr_{d}(\SV^1)$ can also be interpreted as the non-vanishing locus of the composition $\mathcal{T}_d \to p^*\SV \to p^*\SV/p^*\SV^1$. By forming the fibre square
\[
	\xymatrix{
		p^\ast \SV^1 \ar[r] & p^\ast \SV
		\\
		\mathcal{T}_d \cap p^*\SV^1 \ar[u] \ar[r] & \mathcal{T}_d \ar[u]
	}
\]
we observe that the points of $\Gr_{d}(\SV) \backslash \Gr_{d}(\SV^1)$ can be characterized by the set
\[
	\Big\{ \SP_d  \lhd \SV : \SP_d   \overline{\ntriangleleft} \SV^1 \Big\},
\]
where $\SP_d   \overline{\ntriangleleft} \SV^1$ signifies $\SP_d \cap \SV^1 \lhd \SV^1$ and $\SP_d \not \subset \SV^1$.\ Recall the diagram from \cite[Section 5]{balmer2012witt}:
\begin{equation}\label{eqn:grass_blow_up_pre}
	\xymatrix{
	\Gr_{d}( \SV^{1}) \ar@{^{(}->}[r]^-{\iota} & \Gr_{d}(\SV)  & \ar[l]_-{v}  \ar[ld]_-{\tilde{v}} \Gr_{d}(\SV) \backslash \Gr_{d}(\SV^1)  \ar[d]^-{\alpha}  \\
	\Fl_{d-1,d}(\SV^1,\SV^1) \ar[u]^-{\widetilde{\pi}} \ar@{^{(}->}[r]^-{\tilde{\iota}} & \Fl_{d-1,d}(\SV^1,\SV)  \ar[u]^-{\pi} \ar[r]^-{\tilde{\alpha}} & \Gr_{d-1}( \SV^{1})
	}
\end{equation}
where the maps are defined on the functors of points as follows:
\[
	\xymatrix@C=55pt{
	{\left\{\begin{smallmatrix}\SP_{d}\\ \triangle \\ \SV^{1}\end{smallmatrix}\right\}} \ar@{^{(}->}[r] & {\left\{\begin{smallmatrix}\SP_{d}\\ \triangle \\ \SV\end{smallmatrix}\right\}}  & \ar[l]  \ar[ld]  {\left\{\begin{smallmatrix}\SP_{d} & \overline{\ntriangleleft} \SV^{1}\\ \triangle \\ \SV\end{smallmatrix}\right\}} \ar[d]
	\\
	{\left\{\begin{smallmatrix}\SP_{d-1} & \lhd & \SP_{d}\\ \triangle & & \triangle \\ \SV^{1} & = & \SV^{1} \end{smallmatrix}\right\}} \ar[u] \ar@{^{(}->}[r] & {\left\{\begin{smallmatrix}\SP_{d-1} & \lhd & \SP_{d}\\ \triangle & & \triangle \\ \SV^{1} & \lhd & \SV \end{smallmatrix}\right\}}  \ar[u] \ar[r] & {\left\{\begin{smallmatrix}\SP_{d-1}\\ \triangle \\ \SV^{1}\end{smallmatrix}\right\}}.
	}
\]
Here $\tilde{\iota}, \pi, \tilde{\alpha}$ and $\widetilde{\pi}$ are the canonical morphisms defined in Section \ref{sect:maps-flag-variety}. The morphism $\tilde{v}$ maps $\SP_d$ to the flag $\SP_{d-1} \lhd \SP_d$ with $\SP_{d-1} = \SP_d \cap \SV^{1}$. Moreover, $\alpha$ is defined as $\tilde{\alpha} \tilde{v}$. The diagram (\ref{eqn:grass_blow_up_pre}) satisfies the assumption of Theorem \ref{theo:connecting-hom}.

\subsection{Notations on line bundles}\label{subsec:notation_line_bundle}
Let $L$ be a line bundle over $S$.

\begin{table}[H]
	\begin{tabular}{|l|l|c|c|}
		\hline
		Line bundle                                                  & Scheme               & $d$ odd                                              & $d$ even                                           \\ \hline \hline
		$\widetilde{\mathcal{H}}_d^{(0)}= \widetilde{\mathcal{H}}  $ & $\Gr_d(\SV)$         & $L$                                                  & $L \otimes \Delta_d$                               \\ \hline
		$\widetilde{\mathcal{H}}^{(1)}_d$                            & $\Gr_d(\SV^{1})$     & $ L\otimes(\SV/\SV^1)\otimes\Delta_{d}$              & $L$                                                \\ \hline
		$\widetilde{\mathcal{H}}^{(1)}_{d-1}$                        & $\Gr_{d-1}(\SV^{1})$ & $L$                                                  & $L\otimes(\SV/\SV^1)\otimes\Delta_{d-1}$           \\ \hline \hline
		$\mathcal{H}_d^{(0)}= \mathcal{H}$                           & $\Gr_d(\SV)$         & $L \otimes \Delta_d$                                 & $L$                                                \\ \hline
		$\mathcal{H}^{(2)}_d$                                        & $\Gr_d(\SV^{2})$     & $L\otimes\mathrm{det}(\SV/\SV^2)\otimes\Delta_{d}$   & $L$                                                \\ \hline
		$\mathcal{H}^{(2)}_{d-2}$                                    & $\Gr_{d-2}(\SV^{2})$ & $L\otimes\mathrm{det}(\SV/\SV^2)\otimes\Delta_{d-2}$ & $L$                                                \\ \hline \hline
		$\mathcal{H}^{(1)}_{d}$                                      & $\Gr_d(\SV^{1})$     & $L \otimes (\SV/\SV^{1})$                            & $L \otimes \Delta_d $                              \\ \hline
		$\mathcal{H}^{(1)}_{d-1}$                                    & $\Gr_{d-1}(\SV^{1})$ & $L \otimes (\SV/\SV^{1}) \otimes \Delta_{d-1}$       & $L$                                                \\ \hline
		$\mathcal{H}^{(2)}_{d-1}$                                    & $\Gr_{d-1}(\SV^{2})$ & $L \otimes (\SV/\SV^{1})$                            & $L \otimes (\SV^{1}/\SV^{2}) \otimes \Delta_{d-1}$ \\ \hline
	\end{tabular}
	\caption{Notations of line bundles on the corresponding Grassmannians}
 \label{tab:tab-grass-line}
\end{table}

Moreover, we introduce the following notations
$$\widetilde{\mathbb{G}}_e(\SV^j): = (\Gr_e(\SV^j), \widetilde{\mathcal{H}}_e^{(j)}) \textnormal{ and } \mathbb{G}_e(\SV^j): = (\Gr_e(\SV^j), \mathcal{H}_e^{(j)})$$
for Grassmannians with fixed line bundles.
\subsection{The induction step}
The subsequent theorem is crucial for the proof of Theorem \ref{theo:GW-Grassmannian}.
\begin{theorem}\label{thm:grass_induction}
	Suppose that $d\geq2$ and $r \geq d+2$.
	\begin{enumerate}[leftmargin=20pt]
		\item [\rm (A)] The map
		      \[
			      \GW^{[n]}(\widetilde{\mathbb{G}}_d(\SV)) \xleftarrow[\simeq]{\big(\iota_\ast, \pi_\ast \tilde{\alpha}^\ast\big)} \GW^{[n-d]}(\widetilde{\mathbb{G}}_d(\SV^{1})) \oplus \GW^{[n]}(\widetilde{\mathbb{G}}_{d-1}(\SV^{1}))
		      \]
		      is an equivalence of spectra.
		\item [\rm (B)] There is an equivalence
		      \[
			      \GW^{[n]}(\mathbb{G}_d(\SV)) \simeq   \K(\Gr_{d-1}(\SV^{2})) \oplus \GW^{[n-2d]}(\mathbb{G}_d(\SV^{2})) \oplus \GW^{[n]}(\mathbb{G}_{d-2}(\SV^{2}))
		      \]
		      of spectra.
	\end{enumerate}
\end{theorem}

In order to apply Theorem \ref{theo:connecting-hom} to prove Theorem \ref{thm:grass_induction},  one needs to know the number $\lambda(\widetilde{\mathcal{H}}) -d$ and $\lambda(\mathcal{H}) -d $ in $\Z/2\Z $. Recall that $\lambda(L)=0$ and $\lambda(L\otimes \Delta_{d})=1$ for any $L \in \mathrm{Pic}(S)$,  cf. \cite[Remark 5.6]{balmer2012witt}.

\begin{proof}[Proof of Theorem \ref{thm:grass_induction} {\rm (A)}]
	In this case, we have $\lambda(\widetilde{\mathcal{H}}) \equiv d - 1 \mod 2$. Note that
	$
		\omega_{\iota} \cong (\SV/\SV^{1})^{d} \otimes\Delta_d^{-1},
	$
	by {\it loc.\ cit.}.\ The result follows, since the homotopy fibration sequence
	\[
		\xymatrix@C=40pt{
		\GW^{[n-d]}(\widetilde{\mathbb{G}}_d(\SV^1)) \ar[r]^-{\iota_\ast} & \GW^{[n]}(\widetilde{\mathbb{G}}_d(\SV)) \ar[r]^-{(\alpha^\ast)^{-1}v^\ast} & \GW^{[n]}(\widetilde{\mathbb{G}}_{d-1}(\SV^1))
		}
	\]
	splits with a right splitting given by $\pi_\ast \tilde{\alpha}^\ast$ (cf.\ Theorem \ref{theo:connecting-hom}(A)).
\end{proof}

\begin{proof}[Proof of Theorem \ref{thm:grass_induction} {\rm (B)}]
	By the localization theorem, the following sequence
	\begin{equation}\label{eqn:grass_conn_loc}
		\resizebox{\textwidth}{!}{$\GW^{[n-d]}(\mathbb{G}_d(\SV^1)) \xrightarrow{\iota_\ast}  \GW^{[n]}(\mathbb{G}_d(\SV)) \xrightarrow{(\alpha^\ast)^{-1}v^\ast}  \GW^{[n]}(\mathbb{G}_{d-1}(\SV^1)) \xrightarrow{\partial}  \GW^{[n-d]}(\mathbb{G}_d(\SV^1))[1]$}
	\end{equation}
	is a distinguished triangle in $\mathcal{SH}$.\ Since $\lambda(\mathcal{H}) \equiv d \mod 2$, we see that Theorem \ref{theo:connecting-hom} (B) applies to this case. Therefore, the connecting homomorphism $\partial$ can be interpreted geometrically by the following commutative diagram:
	\[
		\xymatrix@C=5pt{
		\GW^{[n]}(\mathbb{G}_{d-1}(\SV^1))  \ar[r]^-{\partial} \ar[d]_-{\breve{\pi}^\ast} & \GW^{[n-d]}(\mathbb{G}_d(\SV^1))[1] & \GW^{[n-d+1]}(\mathbb{G}_d(\SV^1)) \ar[l]_-{\eta \cup}
		\\
		\GW^{[n]}(\mathbb{F}l_{d-1,d}(\SV^1,\SV^1))  \ar[rr]^-{\simeq} &&\GW^{[n]}(\mathbb{F}l_{d-1,d}(\SV^1,\SV^1)) \ar[u]^-{\widetilde{\pi}_\ast} }
	\]
	where 	$\breve{\pi}$ is defined as the composition $\tilde{\alpha} \tilde{\iota}$, and $$\mathbb{F}l_{d-1,d}(\SV^1,\SV^1):= (\Fl_{d-1,d}(\SV^1,\SV^1),\breve{\pi}^\ast\mathcal{H}^{(1)}_{d-1}) \quad  \mathbb{F}l_{d-1,d}(\SV^1,\SV^1):=(\Fl_{d-1,d}(\SV^1,\SV^1), \widetilde{\pi}^*\mathcal{H}^{(1)}_{d} \otimes \omega_{\widetilde{\pi}}). $$

	Next, we need to study the composition $\widetilde{\pi}_\ast \breve{\pi}^\ast$.\ In the spirit of the diagram \eqref{eqn:grass_blow_up_pre}, we construct diagrams as follows:
	\[
		\xymatrix{
		\Gr_{d}(\SV^2) \ar@{^{(}->}[r]^-{\iota_-} & \Gr_{d}(\SV^1)  & \ar[l]_-{v_-}  \ar[ld]|-{\tilde{v}_-} \Gr_{d}(\SV^1) \backslash \Gr_{d}(\SV^2)  \ar[d]^-{\alpha_-}  \\
		\Fl_{d-1,d}(\SV^2,\SV^2)  \ar[u]^-{\widetilde{\pi}_-} \ar@{^{(}->}[r]_-{\tilde{\iota}_-} & \Fl_{d-1,d}(\SV^2,\SV^1)  \ar[u]^-{\pi_-} \ar[r]_-{\tilde{\alpha}_-} & \Gr_{d-1}(\SV^2)
		}	\]
	and
	\[
		\xymatrix{
		\Gr_{d-1}(\SV^2) \ar@{^{(}->}[r]^-{\iota_+} & \Gr_{d-1}(\SV^1)  & \ar[l]_-{v_+}  \ar[ld]|-{\tilde{v}_+} \Gr_{d-1}(\SV^1) \backslash 	\Gr_{d-1}(\SV^2)  \ar[d]^-{\alpha_+}  \\
		\Fl_{d-2,d-1}(\SV^2,\SV^2) \ar[u]^-{\widetilde{\pi}_+} \ar@{^{(}->}[r]_-{\tilde{\iota}_+} &  \Fl_{d-2,d-1}(\SV^2,\SV^1)  \ar[u]^-{\pi_+} \ar[r]_-{\tilde{\alpha}_+} & \Gr_{d-2}(\SV^2)
		}	\]
	The following sublemma is the key for proving the case (B).
	\begin{sublemma}\label{sublem:key-sublem-B} The following diagram
		\begin{equation}
			\xymatrix@C=60pt@R=20pt{
			\GW^{[n]}(\mathbb{G}_{d-1}(\SV^1)) \ar[dd]^-{\widetilde{\pi}_\ast \breve{\pi}^\ast} &   \ar[l]_-{(\iota_{+ \ast}, \pi_{+ \ast} \tilde{\alpha}_+^\ast)} \GW^{[n-d+1]}(\mathbb{G}_{d-1}(\SV^2)) \oplus \GW^{[n]}(\mathbb{G}_{d-2}(\SV^2)) \ar[dd]^-{\everymath={\scriptstyle} \begin{pmatrix}
				0   & 0 \\
				\id & 0
			\end{pmatrix}} \\  \\
			\GW^{[n-d+1]}(\mathbb{G}_d(\SV^1)) & \ar[l]_-{(\iota_{- \ast}, \pi_{- \ast} \tilde{\alpha}_-^\ast)}  \GW^{[n-2d+1]}(\mathbb{G}_{d}(\SV^2)) \oplus \GW^{[n-d+1]}(\mathbb{G}_{d-1}(\SV^2))  }
		\end{equation}
		commutes, where the horizontal maps are stable equivalences.
	\end{sublemma}
	\begin{proof}[Proof of Sublemma \ref{sublem:key-sublem-B}]
		Since $\lambda(\mathcal{H}_{d-1}^{(1)}) \equiv d-2 \mod 2$ and $\lambda(\mathcal{H}_{d}^{(1)}) \equiv d-1 \mod 2$, we see that horizontal maps are stable equivalences by the case (A).\

		\noindent (i) Construct the following diagram:
		$$
			\xymatrix@C=30pt@R=30pt{
			\Fl_{d-1,d}(\SV^{2},\SV^{1}) \ar[d]_-{\pi_-} \ar@{=}[r] & \Fl_{d-1,d}(\SV^{2},\SV^{1}) \ar[r]^-{\tilde{\alpha}_-} \ar[d]^-{\ddot{\iota}} \ar@{}[dr]|{\diagram \label{diag:grass_iota_car}} & \Gr_{d-1}(\SV^2) \ar[d]_-{\iota_+} \\
			\Gr_{d}(\SV^1) & \ar[l]_-{\widetilde{\pi}} 	\Fl_{d-1,d}(\SV^{1},\SV^{1}) \ar[r]^-{\breve{\pi}} & \Gr_{d-1}(\SV^1).
			}
		$$
		Since the diagram $\diag{\ref{diag:grass_iota_car}}$ is a fibre product which is tor-independent, we see that
		\begin{equation}\label{eq:grass_maps_trial_1}
			\begin{aligned}
				\widetilde{\pi}_*   \breve{\pi}^*  (\iota_+)_*
				=  \widetilde{\pi}_*  \ddot{\iota}_* \tilde{\alpha}_-^*
				=  (\pi_-)_* \tilde{\alpha}_-^*
			\end{aligned}
		\end{equation}
		where the first equality is obtained by the excess intersection formula (cf. Theorem \ref{thm:excess_intersection}).

		\noindent (ii) Construct the following diagram:
		$$
			\xymatrix@C=20pt@R=30pt{
			\Fl_{d-2,d}(\SV^{2},\SV^{1}) \ar[rr]^-{\pi_2} \ar[dd]_-{\pi_1} && \Gr_{d}(\SV^1)\\
			& \Fl_{d-2,d-1,d}(\SV^{2},\SV^{1},\SV^{1}) \ar[r]^-{\pi_{23}} \ar[d]_-{\pi_{12}} \ar[ul]^-{\pi_{13}} \ar@{}[dr]|{\diagram \label{diag:grass_v_car}}& \Fl_{d-1,d}(\SV^{1},\SV^{1}) \ar[u]^-{\widetilde{\pi}} \ar[d]_-{\breve{\pi}}\\
			\Gr_{d-2}(\SV^2) & \Fl_{d-2,d-1}(\SV^{2},\SV^{1}) \ar[r]^-{\pi_+} \ar[l]_-{\tilde{\alpha}_+}& \Gr_{d-1}(\SV^1).
			}
		$$
		From this diagram, we see that
		\begin{equation}\label{eqn:grass_conn_1}
			\begin{aligned}
				\widetilde{\pi}_\ast  \breve{\pi}^\ast  \pi_{+ \ast}  \tilde{\alpha}_+^\ast  = \widetilde{\pi}_*  (\pi_{23})_*   \pi_{12}^* \tilde{\alpha}_+^\ast  = \pi_{2 \ast} \pi_{13,\ast}  \pi_{13}^\ast  \pi_1^\ast  = 0.
			\end{aligned}
		\end{equation}
		Note that the first equality holds by the base change formula (cf. Corollary \ref{coro:projective_bundle_base_change}). For the last equality, observe that the flag variety $\Fl_{d-2,d-1,d}(\SV^{2},\SV^{1},\SV^{1})$ is isomorphic to the projective bundle
		$ \mathbb{P}_{\Fl_{d-2,d}(\SV^{2},\SV^{1})}(\mathcal{T}_{d} / \mathcal{T}_{d-2})$
		over $\Fl_{d-2,d}(\SV^{2},\SV^{1})$, and $(\pi_{13})_*\pi_{13}^* = 0$ by Theorem \ref{thm:projective_bundle_thm}(D).
	\end{proof}

	Set $\mathbf{D} = \GW^{[n-2d]}(\mathbb{G}_{d}(\SV^2))$ and $\mathbf{E} = \GW^{[n]}(\mathbb{G}_{d-2}(\SV^2))$. Let
	\[ \xymatrix{\mathbf{A} \ar[r]^-{F} & \mathbf{B} \ar[r]^-{H} & \mathbf{C} \ar[r]^-{\eta \cup} & \mathbf{A}[1]} \]
	be the algebraic Bott sequence
	$$ \GW^{[n-d]}(\mathbb{G}_{d-1}(\SV^2)) \xrightarrow{F} \K(\Gr_{d-1}(\SV^2)) \xrightarrow{H} \GW^{[n-d+1]}(\mathbb{G}_{d-1}(\SV^2)) \xrightarrow{\eta \cup}  \GW^{[n-d]}(\mathbb{G}_{d-1}(\SV^2))[1] $$
	and let
	\[
		\xymatrix@C=35pt{
		\mathbf{X}  \ar[r]^-{\iota_\ast} & \mathbf{Y} \ar[r]^-{(\alpha^\ast)^{-1}v^\ast} & \mathbf{Z} \ar[r]^-{\partial} & \mathbf{X}[1]
		}
	\]
	represent the localization sequence \eqref{eqn:grass_conn_loc}.\ Then by \cite[Theorem II.2.9 (T3)]{schwede2012symmetric} there exist maps $\mathbf{f}, \mathbf{g}$ and $\mathbf{h}$ extend the following diagram to maps of distinguished triangles in $\mathcal{SH}$.
	$$
		\xymatrix@C=50pt@R=12pt{
		\mathbf{A} \ar[dd]_-{\pi_{- \ast} \tilde{\alpha}_-^\ast} \ar[r]^-{F} \ar@{}[ddr]|-{\diagram \label{diag:grass_K_1}} & \mathbf{B} \ar[r]^-{H} \ar@{-->}[dd]_-{\mathbf{f}} \ar@{}[ddr]|-{\diagram \label{diag:grass_K_2}} & \mathbf{C} \ar[r]^-{\eta \cup}\ar[dd]_-{\iota_{+ \ast}} \ar@{}[ddr]|-{\diagram \label{diag:grass_maps_tria_1}} & \mathbf{A}[1] \ar[dd]^-{\pi_{- \ast} \tilde{\alpha}_-^\ast}\\\\
		\mathbf{X}  \ar[r]^-{\iota_\ast} & \mathbf{Y} \ar[r]^-{(\alpha^\ast)^{-1}v^\ast} & \mathbf{Z} \ar[r]^-{\partial} & \mathbf{X}[1]\\
		\mathbf{D} \ar@{}[ddr]|-{\diagram \label{diag:grass_maps_tria_4}}  \ar[dd]_-{\iota_{- \ast}} \ar@{=}[r] & \mathbf{D} \ar[r] \ar@{-->}[dd]_-{\mathbf{g}} & 0 \ar[r] \ar[dd] & \mathbf{D}[1] \ar[dd]_-{\iota_{- \ast}}\\\\
		\mathbf{X}  \ar[r]^-{\iota_\ast} & \mathbf{Y} \ar[r]^-{(\alpha^\ast)^{-1}v^\ast} & \mathbf{Z} \ar[r]^-{\partial} & \mathbf{X}[1] \\
		0 \ar[dd] \ar[r] & \mathbf{E} \ar@{=}[r] \ar@{-->}[dd]_-{\mathbf{h}} \ar@{}[ddr]|-{\diagram \label{diag:grass_maps_tria_3}}  & \mathbf{E} \ar[r] \ar[dd]|-{\pi_{+ \ast} \tilde{\alpha}_+^\ast} \ar@{}[ddr]|-{\diagram \label{diag:grass_maps_tria_2}}  & 0 \ar[dd]\\\\
		\mathbf{X}  \ar[r]^-{\iota_\ast} & \mathbf{Y} \ar[r]^-{(\alpha^\ast)^{-1}v^\ast} & \mathbf{Z} \ar[r]^-{\partial} & \mathbf{X}[1],
		}
	$$
	Note that the diagram $\diag{\ref{diag:grass_maps_tria_1}}$ commutes by \eqref{eq:grass_maps_trial_1} and the diagram $\diag{\ref{diag:grass_maps_tria_2}}$ commutes by \eqref{eqn:grass_conn_1}.
	Finally, we sum up these three maps to obtain the following map of distinguished triangles
	\begin{equation}\label{eq:diagram-grass-case-b}
		\xymatrix{
		\mathbf{D} \oplus \mathbf{A} \ar[r] \ar[d] & \mathbf{D} \oplus \mathbf{B} \oplus \mathbf{E} \ar[r]  \ar[d]^-{(\mathbf{f},\mathbf{g},\mathbf{h})}  & \mathbf{C} \oplus \mathbf{E} \ar[r] \ar[d]  &   \mathbf{D}[1] \oplus \mathbf{A}[1] \ar[d]
		\\ 
		\mathbf{X}  \ar[r] & \mathbf{Y} \ar[r] & \mathbf{Z} \ar[r] & \mathbf{X}[1].
		}
	\end{equation}
	Since the first and third vertical maps are equivalences by (A), so is the map $(\mathbf{f},\mathbf{g},\mathbf{h})$. This completes the proof of Theorem \ref{thm:grass_induction} {\rm (B)}.
\end{proof}
\begin{remark} If we do not pursue the explicit description of the additive basis, then Theorem \ref{theo:GW-Grassmannian} follows directly from Theorem \ref{thm:grass_induction} by means of induction. However, it requires more effort to construct the additive basis, which will occupy the rest of this section.
\end{remark}

\subsection{On the map $(\mathbf{f},\mathbf{g},\mathbf{h})$} The purpose of this subsection is to describe the underlying map in Theorem \ref{thm:grass_induction} (B) explicitly using pushforwards and pullbacks. Let
$$\widehat{\pi}_1: \Fl_{d-1, d}(\SV^2,\SV) \to \Gr_{d-1}(\SV^2) \quad \quad \widehat{\pi}_2: \Fl_{d-1, d}(\SV^2,\SV) \to \Gr_{d}(\SV) $$
$$\widecheck{\pi}_1: \Fl_{d-2, d}(\SV^2,\SV) \to \Gr_{d-2}(\SV^2) \quad \quad \widecheck{\pi}_2: \Fl_{d-2, d}(\SV^2,\SV) \to \Gr_{d}(\SV) $$
be the canonical morphisms. Then, we can form the following pushforwards:
\[\GW^{[n]}(\mathbb{FL}_{d-1, d}(\SV^2,\SV)) \xrightarrow{\widehat{\pi}_{2 \ast}} \GW^{[n]}(\mathbb{G}_d(\SV)) \quad \quad \GW^{[n]}(\mathbb{FL}_{d-2, d}(\SV^2,\SV)) \xrightarrow{\widecheck{\pi}_{2 \ast}} \GW^{[n]}(\mathbb{G}_d(\SV))\]
where we define
$$\mathbb{FL}_{d-1, d}(\SV^2,\SV) = (\Fl_{d-1, d}(\SV^2,\SV), \omega_{\widehat{\pi}_2} \otimes \widehat{\pi}_2^*\mathcal{H} ) \quad\quad \mathbb{FL}_{d-2, d}(\SV^2,\SV) = (\Fl_{d-2, d}(\SV^2,\SV), \omega_{\widecheck{\pi}_2} \otimes \widecheck{\pi}_2^*\mathcal{H} ). $$

\begin{lemma}\label{lem:grassman-case-B-map} Let $\widehat{\iota}: \Gr_d(\SV^2) \hookrightarrow  \Gr_d(\SV)$ be the canonical embedding.\ Then, the map
	$$(\mathbf{f}, \mathbf{g}, \mathbf{h}):        \K(\Gr_{d-1}(\SV^{2})) \oplus \GW^{[n-2d]}(\mathbb{G}_d(\SV^{2})) \oplus \GW^{[n]}(\mathbb{G}_{d-2}(\SV^{2})) \longrightarrow    \GW^{[n]}(\mathbb{G}_d(\SV))  $$
	in the diagram (\ref{eq:diagram-grass-case-b}) is an equivalence of spectra, if we let
	$$\mathbf{f}  = (\widehat{\pi}_2)_* H (\widehat{\pi}_1)^*  \quad\quad
		\mathbf{g} = \widehat{\iota}_* \quad\quad
		\mathbf{h} =  (\widecheck{\pi}_2)_*(\widecheck{\pi}_1)^* $$
	where $H$ is the hyperbolic map
 \begin{equation}\label{eqn:grass_f_hyperbolic}
     H: K(\Fl_{d-1, d}(\SV^2,\SV)) \to \GW^{[n]}(\mathbb{FL}_{d-1, d}(\SV^2,\SV)).
 \end{equation}
\end{lemma}
\begin{proof}
	The diagram (\ref{diag:grass_maps_tria_4}) commutes by the functoriallity of pushforward. Thus, it is enough to show the commutativity of the diagrams (\ref{diag:grass_K_1}), (\ref{diag:grass_K_2}) and (\ref{diag:grass_maps_tria_3}). Form the following diagram
	$$ \xymatrix{
		& \Gr_{d-1}(\SV^2) \\
		\Fl_{d-1,d}(\SV^2,\SV)\ar[d]^-{\widehat{\pi}_2} \ar[ur]^-{\widehat{\pi}_1} & \ar[l]_-{\widehat{v}} \widehat{U} \ar[d]^-{\widehat{\pi}} \ar[r]^-{\widehat{\alpha}} &  \ar[d]^-{\iota_+} \Gr_{d-1}(\SV^2) \ar@{=}[ul] \\
		\Gr_d(\SV) & \ar[l]_-{v} \Gr_d(\SV) \backslash \Gr_{d}(\SV^1) \ar[r]^-{\alpha} & \Gr_{d-1}(\SV^1) \\
		\Fl_{d-2,d}(\SV^2,\SV) \ar[dr]_-{\widecheck{\pi}_1} \ar[u]_-{\widecheck{\pi}_2} & \ar[l]_-{\widecheck{v}} \widecheck{U} \ar[u]_-{\widecheck{\pi}} \ar[r]^-{\widecheck{\alpha}} & \Fl_{d-2,d-1}(\SV^2,\SV^1) \ar[u]_-{\pi_+} \ar[dl]^-{\tilde{\alpha}_+} \\
		& \Gr_{d-2}(\SV^2)
		}$$
	where all the diagrams are commutative and the four square diagrams are cartesian (cf. \cite[Proof of Lemma 5.7]{balmer2012witt}).

	\noindent (i) The commutativity of the diagram  (\ref{diag:grass_maps_tria_3}) can be obtained by the following identites
	$$\begin{aligned}
			\alpha^* (\pi_+)_*	\tilde{\alpha}_+^* & =  \widecheck{\pi}_*  \widecheck{\alpha}^* \tilde{\alpha}_+^* &  & \textnormal{(Base change formula)}      \\
			                                      & =  \widecheck{\pi}_* \widecheck{v}^* (\widecheck{\pi}_1)^*    &  & \textnormal{(Composition of pullbacks)} \\
			                                      & =  v^* (\widecheck{\pi}_2)_* (\widecheck{\pi}_1)^*            &  & \textnormal{(Base change formula)}
		\end{aligned}$$

	\noindent (ii) The commutativity of the diagram  (\ref{diag:grass_K_2}) follows from the following identites
	$$\begin{aligned}
			\alpha^* (\iota_+)_* H & = \widehat{\pi}_*  \widehat{\alpha}^* H                &  & \textnormal{(Base change formula)}        \\
			                       & = \widehat{\pi}_*  H \widehat{\alpha}^*                &  & \textnormal{(Pullbacks commute with hyperbolic maps)} \\
			                       & = \widehat{\pi}_* H \widehat{v}^*  (\widehat{\pi}_1)^* &  & \textnormal{(Composition of pullbacks)}   \\
			                       & = \widehat{\pi}_* \widehat{v}^* H  (\widehat{\pi}_1)^* &  & \textnormal{(Pullbacks commute with hyperbolic maps)} \\
			                       & = v^* (\widehat{\pi}_2)_* H (\widehat{\pi}_1)^*        &  & \textnormal{(Base change formula)}
		\end{aligned}$$
	Note that one source of confusion is that $H(\widehat{\pi}_1)^* \neq (\widehat{\pi}_1)^*H $ since the targets have different dualities, although $\widehat{v}^*H(\widehat{\pi}_1)^* = \widehat{v}^*(\widehat{\pi}_1)^*H $.

	\noindent (iii) To show the commutativity of the diagram  (\ref{diag:grass_K_1}), we form the following diagram
	$$\xymatrix{
		\Gr_d(\SV^1) \ar[d]^-{\iota} & \ar[l]_-{\pi_-} \Fl_{d-1,d}(\SV^2,\SV^1) \ar[d]^-{\breve{\iota}} \ar[r]^-{\tilde{\alpha}_-} & \Gr_{d-1}(\SV^2) \\
		\Gr_d(\SV) & \ar[l]_-{\widehat{\pi}_2} \ar[ur]_-{\widehat{\pi}_1} \Fl_{d-1,d}(\SV^2,\SV)
		}$$
	Now, we see that
	$$\begin{aligned}
			\iota_* (\pi_-)_* (\tilde{\alpha}_-)^* & = (\widehat{\pi}_2)_*	\breve{\iota}_*  (\tilde{\alpha}_-)^* &  & \textnormal{(Composition of pullbacks)}                      \\
			                                       & = (\widehat{\pi}_2)_* H (\widehat{\pi}_1)^* F               &  & \textnormal{(The diagram }\diag{\ref{diag:projective_space_1}} )
		\end{aligned}$$
	Note that in the last equality, we have identified the map $\breve{\iota}: \Fl_{d-1,d}(\SV^2,\SV^1) \to \Fl_{d-1,d}(\SV^2,\SV)$ with the closed embedding of projective bundles $\iota: \mathbb{P}(\SV^1/\mathcal{T}_{d-1}) \to \mathbb{P}(\SV/\mathcal{T}_{d-1}) $ over $\Gr_{d-1}(\SV^2)$.
\end{proof}

\subsection{Buffalo-check Young diagrams}\label{sect:buffalo-check-Young-diagram}
Any $d$-tuple $\Lambda = (\Lambda_1, \ldots, \Lambda_d)$ of integers such that
$$ m \geq \Lambda_1 \geq \Lambda_2 \geq \ldots \geq \Lambda_d \geq 0 $$
shall be called a \textit{Young diagram} in $(d\times m)$-frame (cf.\ \cite[Definition 2.1]{balmer2012witt}).\ See Figure \ref{Fig:Framed_fig_Young_diagram} in the introduction. The non-increasing sequence
$$ m \geq \Lambda_1 \geq \Lambda_2 \geq \ldots \geq \Lambda_d \geq 0 $$
can be rewritten uniquely as follows:
$$ m \geq \Lambda_1 = \ldots = \Lambda_{d_1} > \Lambda_{d_1+1} = \ldots = \Lambda_{d_2} > \ldots > \Lambda_{d_{k-1}+1} = \ldots = \Lambda_{d_k} \geq 0, $$
where $d_k=d$.
One can form pairs of $k$-tuples of integers
\begin{equation}\label{eqn:young_diagram_associated_tuples}
	\underline{d} = (d_1 , \ldots, d_k) \quad \textnormal{ and }  \quad
	\underline{m} = (m_1 , \ldots, m_k),
\end{equation}
where $m_i : = m - \Lambda_{d_i}$ for $1 \leq i \leq k$. Set $d_0 = m_0 =0$. Define another pair of $k$-tuples
\[
	\underline{v} = (v_1, \ldots, v_k) \quad \textnormal{ and }  \quad
	\underline{h} = (h_1 , \ldots, h_k),
\]
where $v_i = d_i - d_{i-1}$ and $h_i = m_i - m_{i-1}$ for $1 \leq i \leq k$.
Construct a $2k+1$-tuple
$$\underline{s} = (s_1,s_2, \ldots,s_{2k+1}) = \begin{cases}
		(\Lambda_d,v_k,h_k, v_{k-1}, h_{k-1} \ldots, v_{2},h_{2}, v_{1},h_{1}) & \textnormal{if $\Lambda_d>0$}  \\
		(v_k,h_k, v_{k-1}, h_{k-1} \ldots, v_{2},h_{2}, v_{1},h_{1},0)         & \textnormal{if  $\Lambda_d=0$}
	\end{cases}$$
Note that the number of segements is precisely the number of non-zero terms in $\underline{s}$, and $s_i$ is the length of the $i$-th segments provided it is non-zero (cf. the introduction). One may find \cite[Figure 3, page 611]{balmer2012witt} helpful.
\begin{definition}
	A Young diagram $\Lambda$ in a $(d\times m)$-frame is called \textit{$K$-even} if one of the following conditions holds:
	\begin{enumerate}
		\item $\Lambda_{d_k}=0$;
		\item there exists an integer $r(\Lambda)\geq 2$ such that $s_2, \ldots, s_{r(\Lambda)-1}$ are even numbers, $s_{r(\Lambda)}$ is an odd number, and $s_{r(\Lambda)} = v_{l}$ for some $1 \geq l \geq k $.
	\end{enumerate}
	A pair of integers $(i,j)$ is a called a \textit{center} of $\Lambda$, if $(i,j)=(1,v_k)$ when condition (i) holds, or if $(i,j) = (\sum_{t=1}^{l-1} h_t + \Lambda_{d_k} + 1, \sum_{t= l}^{k} v_t )$ when condition (ii) holds. See Figure \ref{Fig:K-even}.
\end{definition}

Recall that we order the rows (resp. columns) in $\Xi$ from left to right (resp. bottom to top). Therefore, the pair of integers $(i,j)$ could be considered the \textit{$(i,j)$-box} in $(d \times m)$-frame.
\begin{definition}
	A \textit{buffalo-check pattern} of size $(d \times m)$ is a $d \times m$-frame such that:
	\begin{enumerate}
		\item[(1)] The $(i, j)$-box is colored \textit{black} if both $i$ and $j$ are odd,
		\item[(2)] The $(i, j)$-box is colored \textit{white} if both $i$ and $j$ are even,
		\item[(3)] All other boxes are colored \textit{grey}.
	\end{enumerate}
\end{definition}
See Figure \ref{Fig:Buffalo-check} for an example of a buffalo-check pattern of size $(6 \times 6)$.

\begin{definition}
	A \textit{buffalo-check Young diagram} in $(d\times m)$-frame is a pair $(\Lambda, c_{\Lambda})$ consisting of a $K$-even Young diagram $\Lambda$ in $(d\times m)$-frame and a center $c_\Lambda$ of $\Lambda$.
\end{definition}

\begin{definition}
	If $d \equiv l \mod 2$ (resp. $d \equiv l-1 \mod 2$) for $l \in \Z$, we let $\mathfrak{B}^{l}_{d,m}$ be the set of buffalo-check Young diagrams $(\Lambda, (i,j))$ for both $i$ and $j$ odd (resp. either $i$ or $j$ even).
\end{definition}
In particular, $\mathfrak{B}^{d}_{d,m}$ (resp. $\mathfrak{B}^{d+1}_{d,m}$) is the set of buffalo-check Young diagrams centered at the black boxes (resp. grey boxes) in the buffalo-check pattern.\ Note that there are no buffalo-check Young diagrams centered at the white boxes in the buffalo-check pattern. Moreover, if we fix an integer $l$, every buffalo-check Young diagram $(\Lambda,c_\Lambda) \in \mathfrak{B}^{l}_{d,m}$ is uniquely determined by the underlying $K$-even Young diagram $\Lambda$. For this reason, we use the $K$-even Young diagram $\Lambda$ to indicate the element of $\mathfrak{B}^{l}_{d,m}$ instead of the pair $(\Lambda,c_\Lambda)$.

\subsection{Combinatorial operations on Young diagrams}
Let $\mathfrak{C}_{d,m}$ be the set of Young diagrams in $(d\times m)$-frame.
\begin{definition}
	Let $\Lambda= (\Lambda_1, \ldots, \Lambda_d) $ be a Young diagram in $\mathfrak{C}_{d,m}$.
	\begin{enumerate}
		\item If $\Lambda_d \geq 1$, we define the Young diagram $\iota(\Lambda)$ in $(d\times (m-1))$-frame as the $d$-tuple $\iota(\Lambda) = (\Lambda_1 - 1, \ldots, \Lambda_d - 1)$.
		\item If $\Lambda_d = 0$, we define the Young diagram $v(\Lambda)$ in $((d-1)\times m)$-frame as the $(d-1)$-tuple $v(\Lambda) = (\Lambda_1 , \ldots, \Lambda_{d-1})$.
	\end{enumerate}
\end{definition}

\begin{remark}\label{rmk:young_diagram_manipulation}
	Note that $\iota(\Lambda)$ (resp. $v(\Lambda)$) deletes the first column (resp.\ row) of $\Lambda$.
\end{remark}

\begin{lemma}\label{lma:lambda_iota_v}
	Let $\Lambda = (\Lambda, c_\Lambda)  \in \mathfrak{B}^{l}_{d,m}$. Then, we have
	\begin{enumerate}
		\item $\iota(\Lambda) \in \mathfrak{B}^{l-1}_{d,m-1}$ if $\Lambda_d \geq 1$;
		\item $v(\Lambda) \in \mathfrak{B}^{l}_{d-1,m}$ if $\Lambda_d = 0$ and $c_{\Lambda} \neq (1,1)$.
	\end{enumerate}
\end{lemma}
\begin{proof}
	(i) If $\Lambda_d \geq 1$, we note that the sum of the length of all the vertical segments on the left of the center $c_\Lambda$ is odd. Thus, the buffalo-check Young diagram $\Lambda$ can only center at odd rows, where grey and black boxes appear alternately. Now, the operation $\iota(\Lambda)$ simultaneously deletes the first column of $\Lambda$ and its underlying frame, and observe that the whole buffalo-check pattern therefore moved one column to the left. It follows that the color of the center $c_\Lambda$ changes from black (resp. grey) to grey (resp. black).

	(ii) If $\Lambda_d =0$ and $c_{\Lambda} \neq (1,1)$, then the buffalo-check Young diagram $\Lambda$ can only center at odd columns by a similar reason as in (i). Therefore, the color of the center $c_\Lambda$ changes from black (resp. grey) to grey (resp. black) under the operation $v$.
\end{proof}

\begin{corollary}\label{coro:lambda_iota_v}
	Let $\Lambda \in \mathfrak{B}^{d}_{d,m}$. Then, we have
	\begin{enumerate}
		\item $\iota\iota(\Lambda) \in \mathfrak{B}^{d}_{d,m-2}$ if $\Lambda_d \geq 2$;
		\item $vv(\Lambda) \in \mathfrak{B}^{d-2}_{d-2,m}$ if $\Lambda_{d-1}= \Lambda_d = 0$.
	\end{enumerate}
\end{corollary}
\begin{proof}
	If $\Lambda_d \geq 2$ or $\Lambda_{d-1} = \Lambda_d = 0$, the buffalo-check Young diagram can not center at $(1,1)$. The result follows by applying Lemma \ref{lma:lambda_iota_v} twice in each case. \end{proof}

Suppose that $\Pi= (\Pi_1, \ldots, \Pi_d)$ is an even Young diagram. Let $\rho(\Pi)$ be the number of non-zero rows of $\Pi$ and let $t(\Pi)$ be the half-perimeter of $\Pi$. Recall that $ t(\Pi) \equiv \rho(\Pi) + \Pi_1 \mod 2$ (cf. \cite[Definition 4.3]{balmer2012witt}).\

\begin{definition}
	For any $l \in \Z$, we define $\mathfrak{A}^{l}_{d,m}$ to be the set of even Young diagrams $\Pi$ in a $(d \times m)$-frame such that $t(\Pi) \equiv l \mod 2 $.
\end{definition}
\begin{lemma}\label{lma:pi_iota_v}
	Let $\Pi$ be an element in $ \mathfrak{A}^{l}_{d,m} $. Then, we have
	\begin{enumerate}
		\item $\iota(\Pi) \in \mathfrak{A}^{l-1}_{d,m-1}$, if $\Pi_d \geq 1$;
		\item $v(\Pi) \in \mathfrak{A}^{l}_{d-1,m}$, if $\Pi_d = 0$.
	\end{enumerate}
\end{lemma}

\begin{proof}
	(i) If $\Pi_d \geq 1$, note that the Young diagram $\iota(\Pi)$ is even by \cite[Proposition 2.12 (a)]{balmer2012witt}. Moreover, we must have $\rho(\Pi)\equiv \rho(\iota(\Pi)) \mod 2$. To see this, we observe that $\rho(\Pi)= \rho(\iota(\Pi))$ if $\Pi_d \geq 2$. If $\Pi_d =1$, then $\rho(\Lambda) - \rho(\iota(\Lambda))$ is equal to the length the first vertical segment of the boundary of $\Lambda$, which is even. Finally, by definition $\iota(\Pi)_1 = \Pi_1 -1 $, we see that $t(\iota(\Pi)) \equiv l-1 \mod 2$.

	(ii) If $\Pi_d = 0$, note that the Young diagram $v(\Pi)$ is even by \cite[Proposition 2.12 (b)]{balmer2012witt}. By definition $\rho(v(\Pi)) = \rho(\Pi)$ and $v(\Pi)_1 = \Pi_1$ and the result follows.
\end{proof}

\begin{corollary}\label{coro:pi_iota_v}
	Let $\Pi$ be an element in $ \mathfrak{A}^{l}_{d,m} $. Then, we have
	\begin{enumerate}
		\item $\iota\iota(\Pi) \in \mathfrak{A}^{l}_{d,m-2}$, if $\Pi_d \geq 2$;
		\item $vv(\Pi) \in \mathfrak{A}^{l}_{d-2,m}$, if $\Pi_{d-1} = \Pi_d = 0$.
	\end{enumerate}
\end{corollary}

\begin{proposition}\label{prop:young_bijection}
	\begin{enumerate}[leftmargin=20pt]
		\item[\rm (A)] The following maps of sets
			\[
				\begin{aligned}
					 & \{\Lambda \in \mathfrak{B}^{d+1}_{d,m} \vert \Lambda_d \geq 1\} \xrightarrow[\cong]{\iota} \mathfrak{B}^{d-1}_{d-1,m} & \quad & \{\Lambda \in \mathfrak{B}^{d+1}_{d,m} \vert \Lambda_d = 0\} \xrightarrow[\cong]{v} \mathfrak{B}^{d}_{d,m-1} \\
					 & \{\Pi \in \mathfrak{A}^{d+1}_{d,m} \vert \Pi_d \geq 1\} \xrightarrow[\cong]{\iota} \mathfrak{A}^{d-1}_{d-1,m}         & \quad & \{\Pi \in \mathfrak{A}^{d+1}_{d,m} \vert \Pi_d = 0\} \xrightarrow[\cong]{v} \mathfrak{A}^{d}_{d,m-1}
				\end{aligned}
			\]
			are bijective.
		\item[\rm (B)] The following maps of sets
			\[
				\begin{aligned}
					 & \{\Lambda \in \mathfrak{B}^{d}_{d,m} \vert \Lambda_d \geq 2\} \xrightarrow[\cong]{\iota \iota} \mathfrak{B}^{d-2}_{d-2,m} & \quad & \{\Lambda \in \mathfrak{B}^{d}_{d,m} \vert \Lambda_{d-1} = \Lambda_d = 0\} \xrightarrow[\cong]{vv} \mathfrak{B}^{d}_{d,m-2} \\
					 & \{\Pi \in \mathfrak{A}^{d}_{d,m} \vert \Pi_d \geq 2\} \xrightarrow[\cong]{\iota \iota} \mathfrak{A}^{d-2}_{d-2,m}         & \quad & \{\Pi \in \mathfrak{A}^{d}_{d,m} \vert \Pi_{d-1} = \Pi_d = 0\} \xrightarrow[\cong]{vv} \mathfrak{A}^{d}_{d,m-2}
				\end{aligned}
			\]
			\[
				\{\Lambda \in \mathfrak{B}^{d}_{d,m} \vert \Lambda_d = 0, \Lambda_{d-1} \geq 1\}\xrightarrow[\cong]{\iota v}  \mathfrak{C}_{d-1,m-1}
			\]
			are bijective.
	\end{enumerate}
\end{proposition}
\begin{proof}
	All the maps of sets are well-defined by Lemma \ref{lma:lambda_iota_v}, \ref{lma:pi_iota_v} and Corollary \ref{coro:lambda_iota_v}, \ref{coro:pi_iota_v}. The proof of the case (A) follows mostly from \cite[Proposition 2.12]{balmer2012witt}. The maps $\iota$ (resp. $v$) are bijective, since we can construct the inverse maps by adding the a full column (resp. empty row) to the left (resp. bottom) of a Young diagram $\Lambda = (\Lambda_1 , \ldots, \Lambda_d)$ (resp. $\Lambda = (\Lambda_1 , \ldots, \Lambda_{d-1})$) to obtain a Young diagram $(\Lambda_1 + 1, \ldots, \Lambda_d + 1)$ (resp. $(\Lambda_1 , \ldots, \Lambda_{d-1}, 0)$) in $(d\times m)$-frame. For the case (B), the bijectivity of all the maps can be obtained from the case (A). For the last bijection, the buffalo-check Young diagram $\Lambda \in \mathfrak{B}^{d}_{d,m}$ that satisfies the conditions $\Lambda_d = 0$ and $\Lambda_{d-1} \geq 1$ if and only if it has center $(1,1)$. Furthermore, adding a full column to the left and then an empty row to the bottom of a Young diagram in $((d-1) \times (m-1))$-frame yields a buffalo-check Young diagram $\Lambda$ in $(d\times m)$-frame centered at $(1,1)$.
\end{proof}

\subsection{Flag variety of a Young diagram}
Let $\Lambda$ be a Young diagram in the $(d \times m)$-frame $\Xi$. By the previous section, this amounts to $k$-tuples $\underline{d}$ and $\underline{m}$. By Section \ref{sect:buffalo-check-Young-diagram} and  Definition \ref{def:flag-variety}, we have an associated flag variety $\Fl(\Lambda) := \Fl_{\underline{d}}(\SV_{\bullet }) = \Fl_{d_1, \ldots, d_k}(\SV_{r_1}, \ldots, \SV_{r_k})$ where $\SV_{\bullet } = \{ \SV_{r_1}\lhd \ldots \lhd \SV_{r_k} \}$ with $r_i = d_i+m_i$ for $ 1 \leq i \leq k$.\ Let
$$p_\Lambda: \Fl(\Lambda) \to S$$
be the canonical projection. Recall that $r = d+m$, and we define the following projective morphism
$$ f_{\Lambda} : \Fl(\Lambda) \xhookrightarrow{\iota} \Fl_{\underline{d}}(\SV_{r}, \ldots, \SV_{r}) \xrightarrow{\pi} \Gr_d(\SV_{r}) = \Gr_d(\SV)  $$
which agrees with the map constructed in  \cite[(15)]{balmer2012witt}.\ This yields a commutative diagram
$$
	\xymatrix{
	\Fl(\Lambda) \ar[r]^-{f_\Lambda} \ar[dr]_-{p_{\Lambda}} & \Gr_{d}(\SV_{r}) \ar[d]_-{p} \\
	&	S
	}. $$
\begin{remark}
	Recall that, by deleting the first row and (resp. column) of $\Lambda$, we obtain the Young diagram $\iota(\Lambda)$ (resp. $v(\Lambda)$) in $(d \times (m-1))$-frame (resp. $((d-1) \times m)$-frame). Under our notation, we have canonical projective morphisms $$
		f_{\iota(\Lambda)} : \Fl(\iota(\Lambda)) \to \Gr_d(\SV_{r-1}) = \Gr_d(\SV^1) \quad \quad f_{v(\Lambda)}: \Fl(v(\Lambda)) \to \Gr_{d-1}(\SV_{r-1}) = \Gr_{d-1}(\SV^1) $$
	There is a source of confusion. Note that the targets of the map $f_{\iota(\Lambda)}$ (resp. $f_{v(\Lambda)}$) is $\Gr_d(\SV^1)$ (resp. $\Gr_{d-1}(\SV^1)$) rather than $\Gr_d(\SV)$.
\end{remark}

\subsection{Construction of the maps}
Let $\omega_{\Lambda}$ be the relative canonical bundle of $f_\Lambda$.

\begin{definition}\label{def:grass_mu}
	Let $l \in \mathbb{Z}$, and let
	$$H_{\Lambda}^{l,n}(L): \K(\Fl(\Lambda)) \rightarrow \GW^{[n - |\Lambda|]}(\Fl(\Lambda), \omega_{\Lambda}\otimes f_\Lambda^*(L \otimes \Delta_d^{\otimes l})) $$
	be the hyerbolic map. Define the map
	\[
		\Theta^{l}: \bigoplus_{\Lambda \in \mathfrak{B}^{l}_{d,m}} \K(S) \xrightarrow{\sum \mu_{\Lambda}^{l,n}(L)} \GW^{[n]}(\Gr_{d}(\SV), L\otimes \Delta_d^{\otimes l}).
	\]
	where $\mu_{\Lambda}^{l,n}(L)$ is defined to be the composition
	\[
		\K(S) \xrightarrow{p_{\Lambda}^\ast} \K(\Fl(\Lambda)) \xrightarrow{H_{\Lambda}^{l,n}(L)} \GW^{[n - |\Lambda|]}(\Fl(\Lambda), \omega_{\Lambda}\otimes f_\Lambda^*(L \otimes \Delta_d^{\otimes l})) \xrightarrow{{f_\Lambda}_\ast } \GW^{[n]}(\Gr_{d}(\SV), L\otimes \Delta_d^{\otimes l})
	\]
	of maps of spectra.
\end{definition}

\begin{lemma}
	Suppose that $\Pi$ is an even Young diagram. Then, the class of the relative canonical bundle $\omega_{\Pi}$ in $\Pic(\Fl(\Pi))/2$ is given by
	$$\omega_{\Pi} = \det(\SV)^{\otimes \rho(\SV)} \otimes \Delta_d^{\otimes t(\Pi)}. $$
	Moreover, $p_{\Pi}^\ast (\det (\SV)^{\otimes \rho(\Pi)}) = \omega_{\Pi}\otimes f_\Pi^*(\Delta_d^{\otimes t(\Pi)}) \in \Pic(\Fl(\Pi))/2$.
\end{lemma}

\begin{proof}
	See \cite[Proposition 4.8]{balmer2012witt}.
\end{proof}
\begin{definition}
	Define the map
	\[
		\Omega^{l}:	\bigoplus_{\Pi \in \mathfrak{A}^{l}_{d,m}} \GW^{[n-|\Pi|]}(S, L \otimes \det(\SV)^{\otimes \rho(\Pi)}) \xrightarrow{\sum {\xi_\Pi}} \GW^{[n]}(\Gr_{d}(\SV), L \otimes \Delta_d^{\otimes l}).
	\]
	where the map $\xi_\Pi$ is the composition
	$$\resizebox{\hsize}{!}{$\GW^{[n-|\Pi|]}(S, L \otimes \det (\SV)^{\otimes \rho(\Pi)}) \xrightarrow{p_{\Pi}^\ast} \GW^{[n-|\Pi|]}(\Fl(\Pi), \omega_{\Pi}\otimes f_\Pi^*(L \otimes \Delta_d^{\otimes l})) \xrightarrow{{f_\Pi}_\ast} \GW^{[n]}(\Gr_{d}(\SV), L \otimes \Delta_d^{\otimes l})$}$$
	of maps of spectra.
\end{definition}
\begin{remark}
	By the projection formula, we see that
	\[
		\xi_\Pi(\alpha) := {f_{\Pi}}_\ast p_{\Pi}^\ast (\alpha) = {f_{\Pi}}_\ast f_{\Pi}^\ast p^\ast (\alpha) = {f_{\Pi}}_\ast (1) \cup p^\ast (\alpha)
	\]
	for any $\Pi \in \mathfrak{A}_{d,m}$ and $\alpha \in \GW^{[n-|\Pi|]}_i(S, L \otimes \det(\SV)^{\otimes \rho(\Pi)})$. Moreover, by the projection formula on $K$-theory (cf. \cite[Proposition 3.17]{thomason1990higher}),
	$$
		\mu^{l,n}_{\Lambda}(L)(\beta) := {f_{\Lambda}}_\ast H^{l, n}_{\Lambda}(L) p_{\Lambda}^\ast (\beta) = H^{l,n} {f_{\Lambda}}_\ast p_{\Lambda}^\ast (\beta) = H^{l,n} ({f_{\Lambda}}_\ast (1) \cup p^\ast (\beta))
	$$
	for any $\Lambda \in \mathfrak{B}^{l}_{d,m}$ and $\beta \in \K_i(S)$, where $H^{l,n}$ is the hyperbolic map
 \begin{equation}\label{eqn:grass_hyp}
		H^{l,n}: \K(\Gr_{d}(\SV))  \to \GW^{[n]}(\Gr_{d}(\SV), L \otimes \Delta_d^{\otimes l}).
 \end{equation}
\end{remark}

\subsection{On the main theorem}
Now, we are ready to state our main theorem:
\begin{theorem}\label{thm:grass_flag_result}
	Let $L$ be a line bundle on $S$. Then, the  map
	\[
		(\Theta^{l}, \Omega^{l}): \bigoplus_{\Lambda \in \mathfrak{B}^{l}_{d,m}} \K(S) \oplus \bigoplus_{\Pi \in \mathfrak{A}^{l}_{d,m}} \GW^{[n-|\Pi|]}(S, L \otimes \det(\SV)^{\otimes \rho(\Pi)}) \xrightarrow{\simeq} \GW^{[n]}(\Gr_{d}(\SV), L \otimes \Delta_d^{\otimes l})
	\]
	is an equivalence for any $l \in \mathbb{Z}$.
\end{theorem}

To prove Theorem \ref{thm:grass_flag_result}, we need the following propositions in view of notations introduced in Table \ref{tab:mu}.

\begin{table}[H]
	\begin{tabular}{|ll|l|l|l|c|}
		\hline
		\multicolumn{2}{|c|}{Notation}                        & \multicolumn{1}{c|}{Map}      & \multicolumn{1}{c|}{Source}                                      & \multicolumn{1}{c|}{Target} & Condition\\ \hline \hline
		\multicolumn{1}{|l|}{$\mu^{1}_{\Lambda}$}             & $H^{1}_{\Lambda}$             & $\mu^{d+1,n}_\Lambda (\widetilde{\mathcal{H}})$                  & \multirow{8}{*}{$K(S)$}     & $\GW^{[n]}(\widetilde{\mathbb{G}}_d(\SV))$         &                                 \\ \cline{1-3} \cline{5-6}
		\multicolumn{1}{|l|}{$\mu^{1}_{\iota(\Lambda)}$}      & $H^{1}_{\iota(\Lambda)}$      & $\mu^{d,n-d}_{\iota(\Lambda)} (\widetilde{\mathcal{H}}^{(1)}_d)$ &                             & $\GW^{[n-d]}(\widetilde{\mathbb{G}}_d(\SV^{1}))$   & $\Lambda_d \geq 1$              \\ \cline{1-3} \cline{5-6}
		\multicolumn{1}{|l|}{$\mu^{1}_{v(\Lambda)}$}          & $H^{1}_{v(\Lambda)}$          & $\mu^{d-1,n}_{v(\Lambda)} (\widetilde{\mathcal{H}}^{(1)}_{d-1})$ &                             & $\GW^{[n]}(\widetilde{\mathbb{G}}_{d-1}(\SV^{1}))$ & $\Lambda_d = 0$                 \\ \cline{1-3} \cline{5-6}
		\multicolumn{1}{|l|}{$\mu^{0}_{\Lambda}$}             & $H^{0}_{\Lambda}$             & $\mu^{d,n}_\Lambda (\mathcal{H})$                                &                             & $\GW^{[n]}(\mathbb{G}_d(\SV))$                     &                                 \\ \cline{1-3} \cline{5-6}
		\multicolumn{1}{|l|}{$\mu^{0}_{\iota(\Lambda)}$}      & $H^{0}_{\iota(\Lambda)}$      & $\mu^{d,n-d}_{\iota(\Lambda)} (\mathcal{H}^{(1)}_{d})$           &                             & $\GW^{[n-d]}(\mathbb{G}_d(\SV^{1}))$              & $\Lambda_d \geq 1$              \\ \cline{1-3} \cline{5-6}
		\multicolumn{1}{|l|}{$\mu^{0}_{v(\Lambda)}$}          & $H^{0}_{v(\Lambda)}$          & $\mu^{d-1,n}_{v(\Lambda)} (\mathcal{H}^{(1)}_{d-1})$             &                             & $\GW^{[n]}(\mathbb{G}_{d-1}(\SV^{1}))$             & $\Lambda_d = 0$                 \\ \cline{1-3} \cline{5-6}
		\multicolumn{1}{|l|}{$\mu^{0}_{\iota\iota(\Lambda)}$} & $H^{0}_{\iota\iota(\Lambda)}$ & $\mu^{d,n-2d}_{\iota\iota(\Lambda)} (\mathcal{H}^{(2)}_{d})$     &                             & $\GW^{[n-2d]}(\mathbb{G}_d(\SV^{2}))$              & $\Lambda_d \geq 2$              \\ \cline{1-3} \cline{5-6}
		\multicolumn{1}{|l|}{$\mu^{0}_{vv(\Lambda)}$}         & $H^{0}_{vv(\Lambda)}$         & $\mu^{d-2,n}_{vv(\Lambda)} (\mathcal{H}^{(2)}_{d-2})$            &                             & $\GW^{[n]}(\mathbb{G}_{d-2}(\SV^{2}))$             & $\Lambda_{d-1} = \Lambda_d = 0$ \\ \hline
	\end{tabular}
	\caption{Simplified notations of maps defined in Definition \ref{def:grass_mu} in which the line bundles are defined in Subsection \ref{subsec:notation_line_bundle}. Notations in the second column define the hyperbolic maps appeared in the construction of the maps in the third column.} 
 \label{tab:mu}
\end{table}

\begin{proposition}\label{prop:compatability-mu}
    Suppose that $\Lambda \in \mathfrak{C}_{d,m}$. Then, for any $\ell = 0$ or $1$, we have
	\begin{enumerate}
      \item [\rm (1)] $\mu^{\ell}_{\Lambda}= \iota_\ast \mu^\ell_{\iota(\Lambda)}$, if $\Lambda_d \geq 1$.
      \item [\rm (2)] $\mu^\ell_\Lambda = \pi_\ast \alpha^\ast \mu^\ell_{v(\Lambda)}$, if $\Lambda_d = 0$.
      \item [\rm (3)] $\mu^{0}_{\Lambda} = \widehat{\iota}_* \mu^{0}_{\iota\iota(\Lambda)}$, if $\Lambda_d \geq 2$.
      \item [\rm (4)] $\mu^{0}_\Lambda = (\widecheck{\pi}_2)_*(\widecheck{\pi}_1)^* \mu^{0}_{vv(\Lambda)}$, if $\Lambda_{d-1} = \Lambda_d = 0$.
	\end{enumerate}
\end{proposition}

\begin{proof}
To prove (1), we note that $\Fl(\iota(\Lambda))= \Fl(\Lambda)$ by definition, and we can form the square
	\[
		\xymatrix{\Fl(\iota(\Lambda)) \ar[d]_-{f_{\iota(\Lambda)}} \ar@{=}[r] & \Fl(\Lambda) \ar[d]_-{f_\Lambda} \\
		\Gr_{d}(\SV^1) \ar[r]^-{\iota} & \Gr_{d}(\SV)
		}
	\]
 which is commutative. Observe that 
	\[
 \begin{aligned}
    \mu^\ell_{\Lambda}:=& {f_\Lambda}_\ast H^\ell_{\Lambda} p_{\Lambda}^\ast && \text{(Definition)}\\
    =& {f_\Lambda}_\ast H^\ell_{\iota(\Lambda)} p_{\iota(\Lambda)}^\ast  && \text{(Compare the source of $f_{\Lambda *}$ and the target of $p_\Lambda^*$)}\\
    =& \iota_\ast {f_{\iota(\Lambda)}}_\ast H^\ell_{\iota(\Lambda)} p_{\iota(\Lambda)}^\ast  && \text{(Composition of pushforwards)}\\
    =& \iota_\ast \mu^\ell_{\iota(\Lambda)}  && \text{(Definition)}.
 \end{aligned}
	\]

For the proof of (2), we rewrite $\Fl(v(\Lambda))$ into the following two cases:
	$$\Fl(v(\Lambda)) =
		\begin{cases}
			\Fl_{d_1, \ldots, d_{k-2}, d_{k-1}}(\SV_{r_1}, \dots, \SV_{r_{k-2}},\SV_{r_{k-1}}) & \textnormal{ if } d_{k-1} = d-1\\
			\Fl_{d_1, \ldots, d_{k-1}, d-1}(\SV_{r_1}, \dots, \SV_{r_{k-1}},\SV^1)             & \textnormal{ if } d_{k-1} < d-1
		\end{cases}$$
  Define the following flag variety 
  $$ \Fl_d(v(\Lambda); \SV) := \begin{cases}
			\Fl_{d_1, \ldots, d_k}(\SV_{r_1}, \dots, \SV_{r_{k-1}})  & \textnormal{ if } d_{k-1} = d-1\\
			\Fl_{d_1, \ldots, d_{k-1}, d-1, d}(\SV_{r_1}, \dots, \SV_{r_{k-1}},\SV^1, \SV)             & \textnormal{ if } d_{k-1} < d-1
		\end{cases} $$
      Construct the following commutative diagrams
      \begin{equation}\label{eqn:diagram_mu_v}
		\xymatrix{\Fl(v(\Lambda)) \ar[d]_-{f_{v(\Lambda)}} & \Fl_d(v(\Lambda); \SV) \ar[d]_-{\mathring{\pi}_{v(\Lambda)}} \ar[l]_-{\acute{\pi}_{v(\Lambda)}} \ar[r]^-{\grave{\pi}_{v(\Lambda)}} & \Fl(\Lambda) \ar[d]_-{f_{\Lambda}}\\
		\Gr_{d-1}(\SV^1) & \ar[l]_-{\alpha} \Fl_{d-1,d}(\SV^1,\SV) \ar[r]^-{\pi} & \Gr_d(\SV)
		}
      \end{equation}
  where the left square is a fibre product, and all the maps are canonical. Note that the morphism $\grave{\pi}_{v(\Lambda)}$ is projective and birational, and it is even an identity if $d_{k-1} = d-1$. Then by \cite[Lemme VII.3.5, p. 441]{SGA6}, we have that $\mathrm{R}\grave{\pi}_{v(\Lambda) \ast} (\SO_{B_\Lambda}) = \SO_{\Fl(\Lambda)}$, and therefore $\grave{\pi}_{v(\Lambda) \ast} (1_{B_\Lambda}) = 1_{\Fl(\Lambda)}$ under the pushforward $\grave{\pi}_{v(\Lambda) \ast}: K_0(\Fl_d(v(\Lambda); \SV)) \to K_0(\Fl(\Lambda))$. By the projection formula on $K$-theory (cf. \cite[Proposition 3.17]{thomason1990higher}), we have $\grave{\pi}_{v(\Lambda) \ast} \grave{\pi}_{v(\Lambda)}^\ast = 1$. Now we can conclude that
	\[
		\begin{aligned}
			\mu^\ell_\Lambda := & {f_\Lambda}_\ast H^\ell_{\Lambda} p_{\Lambda}^\ast                                                               &  & \text{(Definition)}                                           \\
			=                          & {f_\Lambda}_\ast H^\ell_{\Lambda} \grave{\pi}_{v(\Lambda) \ast} \grave{\pi}_{v(\Lambda)}^\ast p_{\Lambda}^\ast       &  & \text{(Use $\grave{\pi}_{v(\Lambda) \ast} \grave{\pi}_{v(\Lambda)}^\ast = 1$)} \\
			=                          & {f_\Lambda}_\ast \grave{\pi}_{v(\Lambda) \ast} \widehat{H}^\ell \acute{\pi}_{v(\Lambda)}^\ast p_{v(\Lambda)}^\ast &  & \text{(Pushforwards commute with hyperbolic maps)}\\
			=                          & \pi_\ast \mathring{\pi}_{v(\Lambda) \ast} \widehat{H}^\ell  \acute{\pi}_{v(\Lambda)}^\ast p_{v(\Lambda)}^\ast &  & \text{(Composition of pushforwards)}                              \\
			=                          & \pi_\ast \mathring{\pi}_{v(\Lambda) \ast} \acute{\pi}_{v(\Lambda)}^\ast H^\ell_{v(\Lambda)} p_{v(\Lambda)}^\ast      &  & \text{(Pullbacks commute with hyperbolic maps)}                          \\
			=                          & \pi_\ast \alpha^\ast {f_{v(\Lambda)}}_\ast H^\ell_{v(\Lambda)} p_{v(\Lambda)}^\ast                               &  & \text{(Base change formula)}                                          \\
			=                          & \pi_\ast \alpha^\ast \mu^\ell_{v(\Lambda)}.                                                                      &  & \text{(Definition)}
		\end{aligned}
	\]
where $\widehat{H}^\ell$ is the hyperbolic map $K(S) \to \GW^{[n-|\Lambda|]}\big(\Fl_d(\iota(\Lambda); \SV), \omega_{f_{\Lambda} \grave{\pi}_{v(\Lambda)}} \otimes \grave{\pi}_{v(\Lambda)}^\ast f_{\Lambda}^\ast (\mathcal{H}^{\ell})\big)$. Here, the line bundle $\mathcal{H}^{\ell}$ is defined to be $\mathcal{H}$ if $\ell =0$ and $\widetilde{\mathcal{H}}$ if $\ell =1$. The result of (2) follows. The proof of (3) is similar to the proof of (1). For the proof of (4), we define the following flag variety 
  $$ \Fl_d(vv(\Lambda); \SV) := \begin{cases}
			\Fl_{d_1, \ldots, d_k}(\SV_{r_1}, \dots, \SV_{r_{k-1}})  & \textnormal{ if } d_{k-1} = d-2\\
			\Fl_{d_1, \ldots, d_{k-1}, d-2, d}(\SV_{r_1}, \dots, \SV_{r_{k-1}},\SV^2, \SV)             & \textnormal{ if } d_{k-1} < d-2
		\end{cases} $$
      Construct the following commutative diagrams
      \begin{equation}\label{eqn:diagram_mu_vv}
		\xymatrix{\Fl(vv(\Lambda)) \ar[d]_-{f_{vv(\Lambda)}} & \Fl_d(vv(\Lambda); \SV) \ar[d]_-{\mathring{\pi}_{vv(\Lambda)}} \ar[l]_-{\acute{\pi}_{vv(\Lambda)}} \ar[r]^-{\grave{\pi}_{vv(\Lambda)}} & \Fl(\Lambda) \ar[d]_-{f_{\Lambda}}\\
		\Gr_{d-2}(\SV^2) & \ar[l]_-{\widecheck{\pi}_1} \Fl_{d-2,d}(\SV^2,\SV) \ar[r]^-{\widecheck{\pi}_2} & \Gr_d(\SV)
		}
      \end{equation}
 where the left square is a fibre product, and all the maps are canonical. Note that the morphism $\grave{\pi}_{vv(\Lambda)}$ is projective and birational. The rest of the proof of (4) is then similar to that of (2).
\end{proof}
\begin{remark}
    We do not define   $\mu^1_{\iota\iota(\Lambda)}$ and $\mu^1_{vv(\Lambda)}$, as they are not used in this paper. However, Proposition \ref{prop:compatability-mu} (3) and (4) are also true if we define them in the spirit of Table \ref{tab:mu}.\   
\end{remark}
\begin{proposition}\label{prop:compatability-xi}
Suppose that $\Pi \in \mathfrak{A}_{d,m}$. Then, we have
	\begin{enumerate}
		\item[\rm (1)] $\xi_\Pi = \iota_\ast \xi_{\iota(\Pi)}$, if $\Pi_d \geq 1$.
		\item[\rm (2)] $\xi_\Pi = \pi_\ast \alpha^\ast \xi_{v(\Pi)}$, if $\Pi_d = 0$ and $ t(\Pi) \equiv d+1 \mod 2$.
		\item[\rm (3)] $\xi_\Pi = \widehat{\iota}_\ast \xi_{\ \iota(\Pi)}$, if $\Pi_d \geq 2$.
		\item[\rm (4)] $\xi_\Pi = (\widecheck{\pi}_2)_*(\widecheck{\pi}_1)^* \xi_{vv(\Pi)}$, if $\Pi_{d-1} = \Pi_d = 0$ and $t(\Pi) \equiv d \mod 2$.
	\end{enumerate}
\end{proposition}

\begin{proof}
The proof of (1) and (3) are similar to that of Proposition \ref{prop:compatability-mu} (1) and (3), as long as we delete the mention of the hyperbolic maps showed up in the proof and replace $\mu$ by $\xi$. For the proof of (2), the proof aligned with the proof of Proposition \ref{prop:compatability-mu} (2) in view of the diagram \eqref{eqn:diagram_mu_v}, except that we only have the identity $\grave{\pi}_{v(\Lambda) \ast} \grave{\pi}_{v(\Lambda)}^\ast = 1$ in $\GW$-theory when $ \rho(\Pi) \equiv d+1 \mod 2$. To explain, we observe that 
	\[
		\omega_{\grave{\pi}_{v(\Lambda)}} = (\SV/\SV^{1})^{\otimes d-1-d_{k-1}} \otimes \Delta_{d-1}^{\otimes d-1-d_{k-1}} \otimes \Delta_{d}^{\otimes 1-d+d_{k-1}} 
	\]
    in $\Pic(\Fl_d(v(\Lambda); \SV))$, and $\omega_{\grave{\pi}_{v(\Lambda)}} = 1 \in \Pic(\Fl_d(v(\Lambda); \SV))/2$ when $ t(\Pi) = d+1 \mod 2$. Here, we use $ t(\Pi) \equiv d_{k-1} \mod 2$, cf. \cite[Proposition 4.5]{balmer2012witt}. The proof of (4) is similar to the proof of (2) in view of the diagram \eqref{eqn:diagram_mu_vv}. Note that the canonical bundle
	\[
		\omega_{\grave{\pi}_{vv(\Lambda)}} = \det(\SV/\SV^{2})^{\otimes d-2-d_{k-1}} \otimes \Delta_{d-2}^{\otimes d-2-d_{k-1}} \otimes \Delta_{d}^{\otimes 2-d+d_{k-1}} 
	\]
    in $\Pic(\Fl_d(vv(\Lambda); \SV))$, and $\omega_{\grave{\pi}_{vv(\Lambda)}} = 1 \in \Pic(\Fl_d(vv(\Lambda); \SV))/2$ when $ t(\Pi) = d\mod 2$. 
\end{proof}

\begin{proposition}\label{prop:grass_additive_basis_pi}
	Let $\Lambda$ be a Young diagram in $\mathfrak{C}_{d,m}$ such that $\Lambda_d = 0$ and $\Lambda_{d-1} \geq 1$. Then
	\[
		\mu_\Lambda^0 = \widehat{\pi}_{2 \ast} H (\widehat{\pi}_1)^\ast {f_{\iota v(\Lambda)}}_\ast p_{\iota v(\Lambda)}^\ast.
	\]
\end{proposition}
\begin{proof}
Note that if $\Lambda_d = 0$ and $\Lambda_{d-1} \geq 1$, then $\Lambda$ is a $K$-even Young diagram and the black box $(1,1)$ is a center of $\Lambda$. It follows that $(\Lambda, (1,1)) \in \mathfrak{B}^{d}_{d,m}$. Recall that $\iota v(\Lambda)$ is a Young diagram in $((d-1)\times (m-1))$-frame, cf. Proposition \ref{prop:young_bijection} (B). To proceed, we rewrite the flag variety
$$ 
    \Fl(\iota v(\Lambda)) = \Fl_{d_1,\ldots, d_{k-1}}(\SV_{r_1}, \ldots , \SV_{r_{k-1}})     
$$
If $\Lambda_d = 0$ and $\Lambda_{d-1} \geq 1$, then we must have
	\[
		d_k=d, d_{k-1}=d-1, r_k = d+m \textnormal{ and } r_{k-1} \leq d+m-2.
	\]
	Next, we form the following commutative diagram
	$$
		\xymatrix{
		\Fl(\iota v(\Lambda)) \ar[d]_-{f_{\iota v(\Lambda)}} \ar@{}[dr]|-{\diagram{\label{diag:flag_fiber_product}}} & \Fl(\Lambda) \ar[l]_-{\bar{\pi}_{\Lambda}} \ar[d]_-{\bar{f}_{\Lambda}} \ar[dr]^-{f_{\Lambda}} \\
		\Gr_{d-1}(\SV^2) & \ar[l]_-{\widehat{\pi}_1} \Fl_{d-1, d}(\SV^2,\SV) \ar[r]^-{\widehat{\pi}_2} & \Gr_{d}(\SV)
		}
	$$
	where the left diagram is a fiber square and the map $\bar{f}_\Lambda$ is defined to be the following composition
 $$\Fl(\Lambda) = \Fl_{d_1,\ldots, d_k}(\SV_{r_1}, \ldots , \SV_{r_k}) \stackrel{\iota}\hookrightarrow \Fl_{d_1, \ldots, d_{k-2},d-1,d}(\SV^2, \ldots ,\SV^2, \SV) \xrightarrow{\pi} \Fl_{d-1, d}(\SV^2,\SV).$$ 
	Then by the base change formula on $K$-theory (cf. \cite[Proposition 3.18]{thomason1990higher}) and Lemma \ref{lma:push_FH}, we conclude that
	\[
		\begin{aligned}
			\mu^{0}_{\Lambda} := & {f_{\Lambda}}_\ast H^{0}_{\Lambda} p_{\Lambda}^\ast                                                   &  & \textnormal{(Definition)}                                 \\
			=                    & H^{d,n} {f_{\Lambda}}_\ast p_{\Lambda}^\ast                                                                 &  & \textnormal{(Pushforwards commute with hyperbolic maps)}             \\
			=                    & H^{d,n} \widehat{\pi}_{2 \ast} \bar{f}_{\Lambda \ast} \bar{\pi}_{\Lambda}^\ast p_{\iota v(\Lambda)}^\ast  &  & \textnormal{(Compositions of pushforwards and pullbacks)} \\
			=                    & \widehat{\pi}_{2 \ast} H \bar{f}_{\Lambda \ast} \bar{\pi}_{\Lambda}^\ast p_{\iota v(\Lambda)}^\ast  &  & \textnormal{(Pushforwards commute with hyperbolic maps)}             \\
			=                    & \widehat{\pi}_{2 \ast} H (\widehat{\pi}_1)^\ast {f_{\iota v(\Lambda)}}_\ast p_{\iota v(\Lambda)}^\ast &  & \textnormal{(Base change formula)}
		\end{aligned}
	\]
 where $H^{d,n}$ is the hyperbolic map defined in \eqref{eqn:grass_hyp} and $H$ is the hyperbolic map defined in \eqref{eqn:grass_f_hyperbolic}. 
\end{proof}

\subsection{Proof of Theorem \ref{thm:grass_flag_result}}
	If $d=1$, the statement of Theorem \ref{thm:grass_flag_result} is true by Theorem \ref{thm:proj-bundle-E-trivial}. If $d\geq 1$, the results follow by the induction setp below. Assume that the theorem holds for any $\Gr_{d'}(\SV')$ with either $d'<d$ and $\mathrm{rank}(\SV') \leq \mathrm{rank}(\SV)$ or $d'\leq d$ and $\mathrm{rank}(\SV') < \mathrm{rank}(\SV)$.

	\noindent \textbf{Case (A)} If $l \equiv d+1 \mod 2$, then we need to show that the  map
	\[
		(\Theta^{d+1}, \Omega^{d+1}): \bigoplus_{{\Lambda \in \mathfrak{B}^{d+1}_{d,m}}} \K(S) \oplus \bigoplus_{{\Pi \in \mathfrak{A}^{d+1}_{d,m}}} \GW^{[n-|\Pi|]}(S, L \otimes \det(\SV)^{\otimes \rho(\Pi)}) \xrightarrow{} \GW^{[n]}(\widetilde{\mathbb{G}}_d(\SV))
	\]
	is an equivalence in view of Table \ref{tab:tab-grass-line}. By Theorem \ref{thm:grass_induction}(A), we have an equivalence
	\[
		\GW^{[n-d]}(\widetilde{\mathbb{G}}_d(\SV^{1})) \oplus \GW^{[n]}(\widetilde{\mathbb{G}}_{d-1}(\SV^{1})) \xrightarrow[\simeq]{\big(\iota_\ast, \pi_\ast \tilde{\alpha}^\ast\big)} \GW^{[n]}(\widetilde{\mathbb{G}}_d(\SV)).
	\]
	By the induction hypothesis and  Proposition \ref{prop:young_bijection}(A), the maps
	\[
		\bigoplus_{\Lambda \in \mathfrak{B}^{d+1}_{d,m}}^{\Lambda_d \geq 1} K(S) \oplus \bigoplus_{\Pi \in \mathfrak{A}^{d+1}_{d,m}}^{\Pi_d \geq 1} \GW^{[n-|\Pi|]}(S, L \otimes \det(\SV)^{\otimes \rho(\Pi)}) \xrightarrow{\big(\sum \mu^{1}_{\iota(\Lambda)},\sum \xi_{\iota(\Pi)}\big)} \GW^{[n-d]}(\widetilde{\mathbb{G}}_d(\SV^{1}))
	\]
	\[
		\bigoplus_{\Lambda \in \mathfrak{B}^{d+1}_{d,m}}^{\Lambda_d = 0} K(S) \oplus \bigoplus_{\Pi \in \mathfrak{A}^{d+1}_{d,m}}^{\Pi_d = 0} \GW^{[n-|\Pi|]}(S, L \otimes \det(\SV)^{\otimes \rho(\Pi)}) \xrightarrow{\big(\sum \mu^{1}_{v(\Lambda)}, \sum \xi_{v(\Pi)}\big)} \GW^{[n]}(\widetilde{\mathbb{G}}_{d-1}(\SV^{1}))
	\]
	are equivalences. The result follows from Proposition \ref{prop:compatability-mu}(1) and (2), and \ref{prop:compatability-xi}(1) and (2).

	\noindent \textbf{Case (B)} Suppose that $l \equiv d \mod 2$. we need to show that the  map
	\[
		(\Theta^{d}, \Omega^{d}): \bigoplus_{{\Lambda \in \mathfrak{B}^{d}_{d,m}}} \K(S) \oplus \bigoplus_{{\Pi \in \mathfrak{A}^{d}_{d,m}}} \GW^{[n-|\Pi|]}(S, L \otimes \det(\SV)^{\otimes \rho(\Pi)}) \xrightarrow{} \GW^{[n]}(\mathbb{G}_{d}(\SV))
	\]
	is an equivalence in view of Table \ref{tab:tab-grass-line}.\ By Lemma \ref{lem:grassman-case-B-map}, the map
	\[
		\K(\Gr_{d-1}(\SV^{2})) \oplus \GW^{[n-2d]}(\mathbb{G}_d(\SV^{2})) \oplus \GW^{[n]}(\mathbb{G}_{d-2}(\SV^{2})) \xrightarrow{((\widehat{\pi}_2)_* H (\widehat{\pi}_1)^*,\widehat{\iota}_*, (\widecheck{\pi}_2)_*(\widecheck{\pi}_1)^*)} \GW^{[n]}(\mathbb{G}_{d}(\SV))
	\]
is an equivalence.
	By the induction hypothesis and Proposition \ref{prop:young_bijection}(B), the maps
	\begin{gather*}
		\bigoplus_{\Lambda \in \mathfrak{B}^{d}_{d,m}}^{\Lambda_d \geq 2} K(S) \oplus \bigoplus_{\Pi \in \mathfrak{A}^{d}_{d,m}}^{\Pi_d \geq 2} \GW^{[n-|\Pi|]}(S, L \otimes \det(\SV)^{\otimes \rho(\Pi)}) \xrightarrow{\big(\sum \mu^0_{\iota\iota(\Lambda)},\sum \xi_{\iota(\iota(\Pi))}\big)} \GW^{[n-2d]}(\mathbb{G}_d(\SV^{2}))\\
		\bigoplus_{{\Lambda \in \mathfrak{B}^{d}_{d,m}}}^{\Lambda_{d-1} = \Lambda_{d} = 0} K(S) \oplus \bigoplus_{{\Pi \in \mathfrak{A}^{d}_{d,m}}}^{\mathclap{\Pi_{d-1} = \Pi_d = 0 }} \GW^{[n-|\Pi|]}(S, L \otimes \det(\SV)^{\otimes \rho(\Pi)}) \xrightarrow{\big(\sum \mu^0_{vv(\Lambda)},\sum \xi_{v(v(\Pi))}\big)} \GW^{[n]}(\mathbb{G}_{d-2}(\SV^{2})) \\
  	 \bigoplus_{\Lambda \in \mathfrak{B}^{d}_{d,m}}^{\Lambda_{d-1} \geq 1, \Lambda_d = 0 } K(S) \xrightarrow{(\sum {f_{\iota v(\Lambda)}}_\ast p_{\iota v(\Lambda)}^\ast)} K(\Gr_{d-1}(\SV^2))
	\end{gather*}
	are equivalences.\ The result follows from Proposition \ref{prop:compatability-mu}(3)(4), \ref{prop:compatability-xi}(3)(4) and \ref{prop:grass_additive_basis_pi}.

\begin{theorem}\label{theo:GW-Grassmannian}
	The map  
 $$ (\Theta^0+\Theta^1, \Omega^0 +\Omega^1): \bigoplus_{(\Lambda, c_\Lambda) \in \mathfrak{B}_{d,m}} \K_i(S)  \oplus \bigoplus_{\Pi \in \mathfrak{A}_{d,m}} \GW^{[n-|\Pi|]}_i(S,L) \to \GW^{[n]}_i(\Gr_d(m),L)^{\mathrm{tot}}  $$
 is an equivalence of spectra.
\end{theorem}
\begin{proof}
	Note that the trivial bundle $\SV = \SO^{\oplus d+m}_S$ admits a canonical complete flag. Moreover, we have the following disjoint union of sets
	\[
		\mathfrak{A}_{d,m} = \mathfrak{A}_{d,m}^{d} \bigsqcup \mathfrak{A}_{d,m}^{d+1} \text{\quad and \quad } \mathfrak{B}_{d,m} =  \mathfrak{B}_{d,m}^{d}  \bigsqcup  \mathfrak{B}_{d,m}^{d+1}.
	\]
	The result follows immediately from Theorem \ref{thm:grass_flag_result}.
\end{proof}                                          

\begin{example}
See Figure \ref{fig:example-2-2} and \ref{fig:example-3-3} for examples of the indexing set for the additive basis of Hermitian $K$-theory of $\Gr_2(2)$ and $\Gr_3(3)$.  
	\begin{figure}[H]
		\begin{align*}
			 & \begin{tikzpicture}[scale=0.4]
				         \draw (1,0) -- (3,0) -- (3,-2) -- (1,-2) -- (1,0);
				         \draw (1,-1) -- (3,-1);
				         \draw (2,0) -- (2,-2);
				         \draw[fill, opacity = 0.2] (1,0) -- (3,0) -- (3,-2) -- (1,-2) -- (1,0);
			         \end{tikzpicture}
			\qquad
			\begin{tikzpicture}[scale=0.4]
				\draw (1,0) -- (3,0) -- (3,-2) -- (1,-2) -- (1,0);
				\draw (1,-1) -- (3,-1);
				\draw (2,0) -- (2,-2);
				\draw[fill,opacity=0.2] (1,0) -- (3,0) -- (3,-2) -- (3,-1) -- (1,-1) -- (1,0);
				\draw[very thick]  (3,-1)--(1,-1);
			\end{tikzpicture}
			\qquad
			\begin{tikzpicture}[scale=0.4]
				\draw (1,0) -- (3,0) -- (3,-2) -- (1,-2) -- (1,0);
				\draw (1,-1) -- (3,-1);
				\draw (2,0) -- (2,-2);
				\draw[fill,opacity=0.2] (1,0) -- (2,0) -- (2,-2) -- (1,-2) -- (1,0);
				\draw[very thick]  (2,0)--(2,-2);
			\end{tikzpicture}
			\qquad
			\begin{tikzpicture}[scale=0.4]
				\draw (1,0) -- (3,0) -- (3,-2) -- (1,-2) -- (1,0);
				\draw (1,-1) -- (3,-1);
				\draw (2,0) -- (2,-2);
			\end{tikzpicture}                                       \\
			 & \begin{tikzpicture}[scale=0.4]
				         \draw (1,0) -- (3,0) -- (3,-2) -- (1,-2) -- (1,0);
				         \draw (1,-1) -- (3,-1);
				         \draw (2,0) -- (2,-2);
				         \draw[fill,opacity=0.2] (1,0) -- (3,0) -- (3,-2) -- (3,-1) -- (1,-1) -- (1,0);
				         \draw[pattern=north west lines] (1,-1) -- (2,-1) -- (2,-2) -- (1,-2) -- (1,-1);
			         \end{tikzpicture}
			\qquad
			\begin{tikzpicture}[scale=0.4]
				\draw (1,0) -- (3,0) -- (3,-2) -- (1,-2) -- (1,0);
				\draw (1,-1) -- (3,-1);
				\draw (2,0) -- (2,-2);
				\draw[fill,opacity=0.2] (1,0) -- (2,0) -- (2,-1) -- (1,-1) -- (1,0);
				\draw[pattern=north west lines] (1,-1) -- (2,-1) -- (2,-2) -- (1,-2) -- (1,-1);
			\end{tikzpicture}          \\
			 & \begin{tikzpicture}[scale=0.4]
				         \draw (1,0) -- (3,0) -- (3,-2) -- (1,-2) -- (1,0);
				         \draw (1,-1) -- (3,-1);
				         \draw (2,0) -- (2,-2);
				         \draw[pattern=north west lines] (1,0) -- (2,0) -- (2,-1) -- (1,-1) -- (1,0);
			         \end{tikzpicture}
			\qquad
			\begin{tikzpicture}[scale=0.4]
				\draw (1,0) -- (3,0) -- (3,-2) -- (1,-2) -- (1,0);
				\draw (1,-1) -- (3,-1);
				\draw (2,0) -- (2,-2);
				\draw[fill,opacity=0.2] (1,0) -- (3,0) -- (3,-1) -- (2,-1) -- (2,-2) -- (1,-2) -- (1,0);
				\draw[pattern=north west lines] (2,-1) -- (3,-1) -- (3,-2) -- (2,-2) -- (2,-1);
			\end{tikzpicture}
		\end{align*}
		\caption{ The indexing set for the additive basis of $\GW^{[n]}(\Gr_2(2))^{\mathrm{tot}}$. The first line consists of even Young diagrams in $(2\times 2)$-frame with thickened inner-frame boundary segments; The second line (resp. third line) includes buffalo-check Young diagrams in $\mathfrak{B}^{2}_{2,2}$ (resp. $\mathfrak{B}^{3}_{2,2}$) with center hatched.
		}
  \label{fig:example-2-2}
	\end{figure}
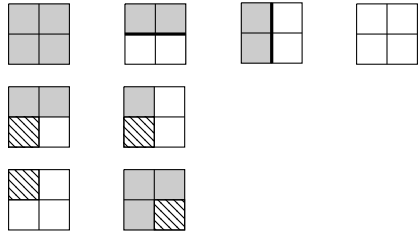
	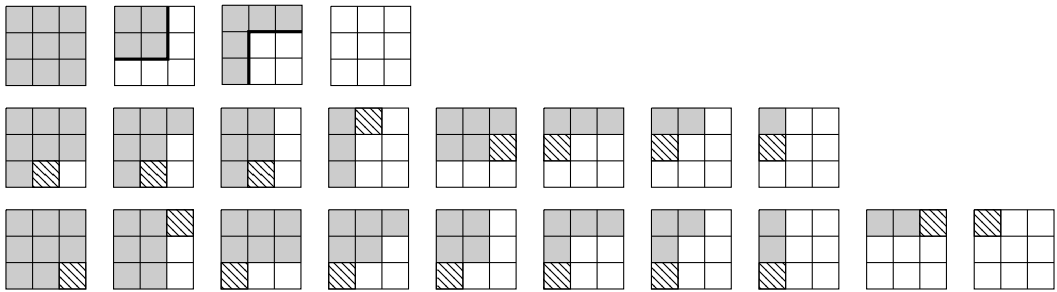
\begin{figure}[H]
		\begin{align*}
			 & \begin{tikzpicture}[scale=0.35]
				         \draw (1,0) -- (4,0) -- (4,-3) -- (1,-3) -- (1,0);
				         \draw (1,-1) -- (4,-1);
				         \draw (1,-2) -- (4,-2);
				         \draw (2,0) -- (2,-3);
				         \draw (3,0) -- (3,-3);
				         \draw[fill, opacity = 0.2] (1,0) -- (4,0) -- (4,-3) -- (1,-3) -- (1,0);
			         \end{tikzpicture}
			\quad
			\begin{tikzpicture}[scale=0.35]
				\draw (1,0) -- (4,0) -- (4,-3) -- (1,-3) -- (1,0);
				\draw (1,-1) -- (4,-1);
				\draw (1,-2) -- (4,-2);
				\draw (2,0) -- (2,-3);
				\draw (3,0) -- (3,-3);
				\draw[fill, opacity = 0.2] (1,0) -- (3,0) -- (3,-2) -- (1,-2) -- (1,0);
				\draw[very thick]  (3,0)--(3,-2)--(1,-2);
			\end{tikzpicture}
			\quad
			\begin{tikzpicture}[scale=0.35]
				\draw (1,0) -- (4,0) -- (4,-3) -- (1,-3) -- (1,0);
				\draw (1,-1) -- (4,-1);
				\draw (1,-2) -- (4,-2);
				\draw (2,0) -- (2,-3);
				\draw (3,0) -- (3,-3);
				\draw[fill, opacity = 0.2] (1,0) -- (4,0) -- (4,-1) -- (2,-1) -- (2,-3) -- (1,-3) -- (1,0);
				\draw[very thick]  (4,-1)--(2,-1)--(2,-3);
			\end{tikzpicture}
			\quad
			\begin{tikzpicture}[scale=0.35]
				\draw (1,0) -- (4,0) -- (4,-3) -- (1,-3) -- (1,0);
				\draw (1,-1) -- (4,-1);
				\draw (1,-2) -- (4,-2);
				\draw (2,0) -- (2,-3);
				\draw (3,0) -- (3,-3);
			\end{tikzpicture}                                                   \\
			 & \begin{tikzpicture}[scale=0.35]
				         \draw (1,0) -- (4,0) -- (4,-3) -- (1,-3) -- (1,0);
				         \draw (1,-1) -- (4,-1);
				         \draw (1,-2) -- (4,-2);
				         \draw (2,0) -- (2,-3);
				         \draw (3,0) -- (3,-3);
				         \draw[fill, opacity = 0.2] (1,0) -- (4,0) -- (4,-2) -- (2,-2) -- (2,-3) -- (1,-3) -- (1,0);
				         \draw[pattern=north west lines] (2,-2) -- (3,-2) -- (3,-3) -- (2,-3) -- (2,-2);
			         \end{tikzpicture}
			\quad
			\begin{tikzpicture}[scale=0.35]
				\draw (1,0) -- (4,0) -- (4,-3) -- (1,-3) -- (1,0);
				\draw (1,-1) -- (4,-1);
				\draw (1,-2) -- (4,-2);
				\draw (2,0) -- (2,-3);
				\draw (3,0) -- (3,-3);
				\draw[fill, opacity = 0.2] (1,0) -- (4,0) -- (4,-1) -- (3,-1) -- (3,-2) -- (2,-2) -- (2,-3) -- (1,-3) -- (1,0);
				\draw[pattern=north west lines] (2,-2) -- (3,-2) -- (3,-3) -- (2,-3) -- (2,-2);
			\end{tikzpicture}
			\quad
			\begin{tikzpicture}[scale=0.35]
				\draw (1,0) -- (4,0) -- (4,-3) -- (1,-3) -- (1,0);
				\draw (1,-1) -- (4,-1);
				\draw (1,-2) -- (4,-2);
				\draw (2,0) -- (2,-3);
				\draw (3,0) -- (3,-3);
				\draw[fill, opacity = 0.2] (1,0) -- (3,0) -- (3,-2) -- (2,-2) -- (2,-3) -- (1,-3) -- (1,0);
				\draw[pattern=north west lines] (2,-2) -- (3,-2) -- (3,-3) -- (2,-3) -- (2,-2);
			\end{tikzpicture}
			\quad
			\begin{tikzpicture}[scale=0.35]
				\draw (1,0) -- (4,0) -- (4,-3) -- (1,-3) -- (1,0);
				\draw (1,-1) -- (4,-1);
				\draw (1,-2) -- (4,-2);
				\draw (2,0) -- (2,-3);
				\draw (3,0) -- (3,-3);
				\draw[fill, opacity = 0.2] (1,0) -- (2,0) -- (2,-3) -- (1,-3) -- (1,0);
				\draw[pattern=north west lines] (2,0) -- (3,0) -- (3,-1) -- (2,-1) -- (2,0);
			\end{tikzpicture}
			\quad
			\begin{tikzpicture}[scale=0.35]
				\draw (1,0) -- (4,0) -- (4,-3) -- (1,-3) -- (1,0);
				\draw (1,-1) -- (4,-1);
				\draw (1,-2) -- (4,-2);
				\draw (2,0) -- (2,-3);
				\draw (3,0) -- (3,-3);
				\draw[fill, opacity = 0.2] (1,0) -- (4,0) -- (4,-1) -- (3,-1) -- (3,-2)-- (1,-2) -- (1,0);
				\draw[pattern=north west lines] (3,-1) -- (4,-1) -- (4,-2) -- (3,-2) -- (3,-1);
			\end{tikzpicture}
			\quad
			\begin{tikzpicture}[scale=0.35]
				\draw (1,0) -- (4,0) -- (4,-3) -- (1,-3) -- (1,0);
				\draw (1,-1) -- (4,-1);
				\draw (1,-2) -- (4,-2);
				\draw (2,0) -- (2,-3);
				\draw (3,0) -- (3,-3);
				\draw[fill, opacity = 0.2] (1,0) -- (4,0) -- (4,-1) -- (1,-1) -- (1,0);
				\draw[pattern=north west lines] (1,-1) -- (2,-1) -- (2,-2) -- (1,-2) -- (1,-1);
			\end{tikzpicture}
			\quad
			\begin{tikzpicture}[scale=0.35]
				\draw (1,0) -- (4,0) -- (4,-3) -- (1,-3) -- (1,0);
				\draw (1,-1) -- (4,-1);
				\draw (1,-2) -- (4,-2);
				\draw (2,0) -- (2,-3);
				\draw (3,0) -- (3,-3);
				\draw[fill, opacity = 0.2] (1,0) -- (3,0) -- (3,-1) -- (1,-1) -- (1,0);
				\draw[pattern=north west lines] (1,-1) -- (2,-1) -- (2,-2) -- (1,-2) -- (1,-1);
			\end{tikzpicture}
			\quad
			\begin{tikzpicture}[scale=0.35]
				\draw (1,0) -- (4,0) -- (4,-3) -- (1,-3) -- (1,0);
				\draw (1,-1) -- (4,-1);
				\draw (1,-2) -- (4,-2);
				\draw (2,0) -- (2,-3);
				\draw (3,0) -- (3,-3);
				\draw[fill, opacity = 0.2] (1,0) -- (2,0) -- (2,-1) -- (1,-1) -- (1,0);
				\draw[pattern=north west lines] (1,-1) -- (2,-1) -- (2,-2) -- (1,-2) -- (1,-1);
			\end{tikzpicture}                      \\
			 & \begin{tikzpicture}[scale=0.35]
				         \draw (1,0) -- (4,0) -- (4,-3) -- (1,-3) -- (1,0);
				         \draw (1,-1) -- (4,-1);
				         \draw (1,-2) -- (4,-2);
				         \draw (2,0) -- (2,-3);
				         \draw (3,0) -- (3,-3);
				         \draw[fill, opacity = 0.2] (1,0) -- (4,0) -- (4,-2) -- (3,-2) -- (3,-3) -- (1,-3) -- (1,0);
				         \draw[pattern=north west lines] (3,-2) -- (4,-2) -- (4,-3) -- (3,-3) -- (3,-2);
			         \end{tikzpicture}
			\quad
			\begin{tikzpicture}[scale=0.35]
				\draw (1,0) -- (4,0) -- (4,-3) -- (1,-3) -- (1,0);
				\draw (1,-1) -- (4,-1);
				\draw (1,-2) -- (4,-2);
				\draw (2,0) -- (2,-3);
				\draw (3,0) -- (3,-3);
				\draw[fill, opacity = 0.2] (1,0) -- (3,0) -- (3,-3) -- (1,-3) -- (1,0);
				\draw[pattern=north west lines] (3,0) -- (4,0) -- (4,-1) -- (3,-1) -- (3,0);
			\end{tikzpicture}
			\quad
			\begin{tikzpicture}[scale=0.35]
				\draw (1,0) -- (4,0) -- (4,-3) -- (1,-3) -- (1,0);
				\draw (1,-1) -- (4,-1);
				\draw (1,-2) -- (4,-2);
				\draw (2,0) -- (2,-3);
				\draw (3,0) -- (3,-3);
				\draw[fill, opacity = 0.2] (1,0) -- (4,0) -- (4,-2) -- (1,-2) -- (1,0);
				\draw[pattern=north west lines] (1,-2) -- (2,-2) -- (2,-3) -- (1,-3) -- (1,-2);
			\end{tikzpicture}
			\quad
			\begin{tikzpicture}[scale=0.35]
				\draw (1,0) -- (4,0) -- (4,-3) -- (1,-3) -- (1,0);
				\draw (1,-1) -- (4,-1);
				\draw (1,-2) -- (4,-2);
				\draw (2,0) -- (2,-3);
				\draw (3,0) -- (3,-3);
				\draw[fill, opacity = 0.2] (1,0) -- (4,0) -- (4,-1) -- (3,-1) -- (3,-2) -- (1,-2) -- (1,0);
				\draw[pattern=north west lines] (1,-2) -- (2,-2) -- (2,-3) -- (1,-3) -- (1,-2);
			\end{tikzpicture}
			\quad
			\begin{tikzpicture}[scale=0.35]
				\draw (1,0) -- (4,0) -- (4,-3) -- (1,-3) -- (1,0);
				\draw (1,-1) -- (4,-1);
				\draw (1,-2) -- (4,-2);
				\draw (2,0) -- (2,-3);
				\draw (3,0) -- (3,-3);
				\draw[fill, opacity = 0.2] (1,0) -- (3,0) -- (3,-2) -- (1,-2) -- (1,0);
				\draw[pattern=north west lines] (1,-2) -- (2,-2) -- (2,-3) -- (1,-3) -- (1,-2);
			\end{tikzpicture}
			\quad
			\begin{tikzpicture}[scale=0.35]
				\draw (1,0) -- (4,0) -- (4,-3) -- (1,-3) -- (1,0);
				\draw (1,-1) -- (4,-1);
				\draw (1,-2) -- (4,-2);
				\draw (2,0) -- (2,-3);
				\draw (3,0) -- (3,-3);
				\draw[fill, opacity = 0.2] (1,0) -- (4,0) -- (4,-1) -- (2,-1) -- (2,-2) -- (1,-2) -- (1,0);
				\draw[pattern=north west lines] (1,-2) -- (2,-2) -- (2,-3) -- (1,-3) -- (1,-2);
			\end{tikzpicture}
			\quad
			\begin{tikzpicture}[scale=0.35]
				\draw (1,0) -- (4,0) -- (4,-3) -- (1,-3) -- (1,0);
				\draw (1,-1) -- (4,-1);
				\draw (1,-2) -- (4,-2);
				\draw (2,0) -- (2,-3);
				\draw (3,0) -- (3,-3);
				\draw[fill, opacity = 0.2] (1,0) -- (3,0) -- (3,-1) -- (2,-1) -- (2,-2) -- (1,-2) -- (1,0);
				\draw[pattern=north west lines] (1,-2) -- (2,-2) -- (2,-3) -- (1,-3) -- (1,-2);
			\end{tikzpicture}
			\quad
			\begin{tikzpicture}[scale=0.35]
				\draw (1,0) -- (4,0) -- (4,-3) -- (1,-3) -- (1,0);
				\draw (1,-1) -- (4,-1);
				\draw (1,-2) -- (4,-2);
				\draw (2,0) -- (2,-3);
				\draw (3,0) -- (3,-3);
				\draw[fill, opacity = 0.2] (1,0) -- (2,0) -- (2,-2) -- (1,-2) -- (1,0);
				\draw[pattern=north west lines] (1,-2) -- (2,-2) -- (2,-3) -- (1,-3) -- (1,-2);
			\end{tikzpicture}
			\quad
			\begin{tikzpicture}[scale=0.35]
				\draw (1,0) -- (4,0) -- (4,-3) -- (1,-3) -- (1,0);
				\draw (1,-1) -- (4,-1);
				\draw (1,-2) -- (4,-2);
				\draw (2,0) -- (2,-3);
				\draw (3,0) -- (3,-3);
				\draw[fill, opacity = 0.2] (1,0) -- (3,0) -- (3,-1) -- (1,-1) -- (1,0);
				\draw[pattern=north west lines] (3,0) -- (4,0) -- (4,-1) -- (3,-1) -- (3,0);
			\end{tikzpicture}
			\quad
			\begin{tikzpicture}[scale=0.35]
				\draw (1,0) -- (4,0) -- (4,-3) -- (1,-3) -- (1,0);
				\draw (1,-1) -- (4,-1);
				\draw (1,-2) -- (4,-2);
				\draw (2,0) -- (2,-3);
				\draw (3,0) -- (3,-3);
				\draw[pattern=north west lines] (1,0) -- (2,0) -- (2,-1) -- (1,-1) -- (1,0);
			\end{tikzpicture}
		\end{align*}
		\caption{The indexing set for the additive basis of $\GW^{[n]}(\Gr_3(3))^{\mathrm{tot}}$. The first line consists of even Young diagrams in $(3\times 3)$-frame with thickened inner-frame boundary segments; The second line (resp. third line) includes buffalo-check Young diagrams in $\mathfrak{B}^{3}_{3,3}$ (resp.\ $\mathfrak{B}^{4}_{3,3}$) with center hatched.
		}
      \label{fig:example-3-3}
	\end{figure}
\end{example}

\hfill \break
\noindent \textbf{Acknowledgment}. We would like to thank Marco Schlichting for providing the key idea in Appendix A. Both authors are partially supported NSFC Grant 12271529, NSFC Grant 12271500, and the Fundamental Research Funds from the Central Universities Sun Yat-sen University 34000-31610294. HX was partially supported by EPSRC Grant EP/M001113/1, the DFG priority programme SPP1786, and the DFG-funded research training group GRK 2240: Algebro-Geometric Methods in Algebra, Arithmetic and Topology.

\appendix

\section{The connecting homomorphism and the cone}\label{sect:connecting-and-cone}
The purpose of this appendix is to interpret the connecting homomorphism for Hermitian $K$-theory via a cone construction due to Balmer in Witt theory. This interpretation plays a crucial role in understanding the geometric description of the connecting homomorphism in Hermitian $K$-theory.

Let $\AC= (\AC, u, \sharp, \mathrm{can})$ be a dg category with weak equivalences and duality. Suppose that we have a larger set $w \subset Z^0 \AC$ of weak equivalences such that $(\AC, w, \sharp, \mathrm{can})$ is also a dg category with weak equivalences and duality. Consider the full dg subcategory $\mathrm{Mor}_{w} (\AC) := \mathrm{Fun}_{w}([1], \AC)$ of the dg category $\mathrm{Mor} (\AC) := \mathrm{Fun}([1], \AC)$. It consists of functors $[1]\to \AC$ whose images are in $w$ (see \cite[p.367]{schlichting2010mayer}). Note that $\mathrm{Mor}_{w} (\AC)$ (resp. $\mathrm{Mor} (\AC)$) is a category with weak equivalences equipped with a duality coming from $\AC$, cf. \cite[Section 1.8 and 1.10]{schlichting2017hermitian}.

By \cite[Proposition 4.7]{schlichting2017hermitian}, the sequence
\[
	\GW^{[n]}(\AC^w) \xrightarrow{} \GW^{[n]}(\AC) \xrightarrow{} \GW^{[n]}(\AC, w) \xrightarrow{\partial} \GW^{[n]}(\AC^w)[1]
\]
is a distinguished triangle in $\mathcal{SH}$. By \cite[Proposition 4.9]{schlichting2017hermitian}, the sequence
\[
	\GW^{[n]}(\AC^w) \xrightarrow{I} \GW^{[n]}(\mathrm{Mor}(\AC^w)) \xrightarrow{\mathrm{cone}} \GW^{[n+1]}(\AC^w) \xrightarrow{-\eta \cup} \GW^{[n]}(\AC^w) [1]
\]
is also a distinguished triangle in $\mathcal{SH}$.

Define the form functor $(P,\phi): \mathrm{Mor}_{w}(\AC) \to \AC$, where $P$ sends $f:A_0 \to A_1$ to $A_0$ and $\phi(f) = f^{\sharp}: A_1^{\sharp} \to A_0^{\sharp}$. Then the form functor of dg categories with weak equivalences and duality
\[
	P: (\mathrm{Mor}_{w}(\AC), w) \to (\AC,w)
\]
induces a stable equivalence of Grothendieck-Witt spectra.

\begin{lemma}\label{lem:balmer-schlichting-coin}
	The following diagram
	\[
		\xymatrix{
		\GW^{[n]}(\mathrm{Mor}_{w}(\AC)) \ar[r]^-{\mathrm{cone}} \ar[d]_-{P} & \GW^{[n+1]}(\AC^w) \ar[d]_-{\eta \cup}\\
		\GW^{[n]}(\AC, w) \ar[r]^-{\partial}& \GW^{[n]}(\AC^w)[1]
		}
	\]
	commutes up to homotopy.
\end{lemma}
\begin{proof}
	Consider the following commutative diagram of dg categories with weak equivalences and duality
	\[
		\xymatrix{
		\AC^w \ar[r]^-{I} \ar[d]_-{}& \mathrm{Mor}(\AC^w) \ar[d]_-{} \ar[r]^-{\mathrm{cone}} & \left(\AC^w\right)^{[1]} \ar@{=}[d]\\
		\AC \ar[r]^-{I} \ar[d]_-{} & \mathrm{Mor}_{w}(\AC) \ar[r]^-{\mathrm{cone}} \ar[d]_-{P} & \left(\AC^w\right)^{[1]}  \\
		(\AC, w) \ar@{=}[r]& (\AC, w),
		}
	\]
	where all rows and columns induce homotopy fibration sequences of Grothendieck-Witt spectra. By the naturality of the connecting morphisms and \cite[Theorem II.2.9 (T2)]{schwede2012symmetric}, there is a map of (shifted) homotopy fibration sequences
	$$
		\xymatrix{
		\GW^{[n+1]}(\AC^w)[-1] \ar[r]^-{\eta \cup} \ar@{=}[d]& \GW^{[n]}(\AC^w) \ar[r]^-{I} \ar[d]& \GW^{[n]}(\mathrm{Mor}(\AC^w)) \ar[d]_-{} \ar[r]^-{\mathrm{cone}} & \GW^{[n+1]}(\AC^w) \ar@{=}[d]\\
		\GW^{[n+1]}(\AC^w)[-1] \ar[r] & \GW^{[n]}(\AC) \ar[r]^-{I} & \GW^{[n]}(\mathrm{Mor}_{w}(\AC)) \ar[r]^-{\mathrm{cone}}[-1] & \GW^{[n+1]}(\AC^w)
		}
	$$

	Then by \cite[Theorem II.2.9 (T4)]{schwede2012symmetric}, there is a commutative diagram
	$$
		\xymatrix{
		\GW^{[n+1]}(\AC^w)[-1] \ar[r]^-{\eta \cup} \ar@{=}[d]& \GW^{[n]}(\AC^w) \ar[r]^-{I} \ar[d]_-{} & \GW^{[n]}(\mathrm{Mor}(\AC^w)) \ar[d]_-{} \ar[r]^-{\mathrm{cone}} & \GW^{[n+1]}(\AC^w) \ar@{=}[d]\\
		\GW^{[n+1]}(\AC^w)[-1] \ar[r] & \GW^{[n]}(\AC) \ar[r]^-{I} \ar[d]_-{} & \GW^{[n]}(\mathrm{Mor}_{w}(\AC)) \ar[r]^-{\mathrm{cone}} \ar[d]_-{P} \ar@{}[dr]|-{\diagram \label{diag:appendix}}& \GW^{[n+1]}(\AC^w) \ar[d]_-{\eta \cup}\\
		&\GW^{[n]}(\AC, w) \ar[d]_-{\partial} \ar@{=}[r]& \GW^{[n]}(\AC, w) \ar[d]_-{} \ar[r]^-{\partial} & \GW^{[n]}(\AC^w)[1] \\
		&\GW^{[n]}(\AC^w)[1] \ar[r]^-{I} & \GW^{[n]}(\mathrm{Mor}(\AC^w))[1],
		}
	$$
	and the diagram $\diag{\ref{diag:appendix}}$ is all we need.
\end{proof}

\bibliographystyle{halpha}
\bibliography{paper}

\end{document}